\documentclass[11pt, a4paper, UKenglish]{article}
\usepackage[UKenglish]{babel}

\usepackage{preamble_new} 

\author{Juhan Aru, Aleksandra Korzhenkova}
\date{\vspace{-5ex}}
\title{The~Spherical~model and the large-$N$ limit of the Spin~\texorpdfstring{$O(N)$}{O(N)}~model via the Gaussian free field}

\begin{document}
\maketitle

\pagenumbering{arabic}

\begin{abstract}

We revisit the relation between the spherical model of Berlin–Kac and the spin $O(N)$ model in the limit $N \to \infty$, and explain how they are connected via the discrete Gaussian free field (GFF).

Using probabilistic limit theorems and concentration of measure, we prove that the infinite-volume limit of the spherical model on a $d$-dimensional torus is a massive GFF in the high-temperature regime, a standard GFF at the critical temperature, and a standard GFF plus a Rademacher random constant in the low-temperature regime. The proof at the critical temperature appears to be new and relies on a fine analysis of the zero-average Green's function on the torus.

We study the spin $O(N)$ model in the double limit of spin dimensionality and torus size. Sending $N \to \infty$ first, and then the torus size to infinity, we show that the different spin coordinates become i.i.d. fields, distributed as a massive GFF in the high-temperature regime, a standard GFF at the critical temperature, and a standard GFF plus a Gaussian random constant in the low-temperature regime.

In particular, although the limiting free energies per site of the two models agree at all temperatures, their finite-dimensional laws still differ in terms of their zero modes in the low-temperature regime.
\end{abstract}

\section{Introduction}
\label{sec:intro}

\subsection{Spherical model}
Let $\T_n^d$ denote the $d$-dimensional discrete torus of side length $n$. To each vertex $x \in \T_n^d$, we associate a continuous ``spin'' $\theta_x \in \R$ in such a way that the full configuration $\theta = (\theta_x)_{x \in \T^d_n}$ lies on the $(n^d-1)$-dimensional sphere of radius $\sqrt{n^d}$, that is, $\ntwo{\theta}^2 \coloneqq \sum_x \theta_x^2 = n^d$. We say that $\theta$ is a configuration of the \emph{spherical model} on $\T^d_n$ at inverse temperature $\beta \geq 0$ if 
\begin{align*}
    \theta \sim \nu_{\T^d_n, \beta}[\d \theta] = \frac{1}{Z_{\T^d_n,\beta}} \exp \bigg( \frac{\beta}{2} \sum_{x \sim y} \theta_x \theta_y\bigg) \mathrm{Unif}_{\sqrt{n^d} \Ss^{n^d-1}}[\d \theta],
\end{align*}
where $x\sim y$ denotes the ordered pairs of neighboring vertices. 
The spherical model was first introduced by Berlin and Kac \cite{SpherModel} in 1952 as a simplification of the Ising model that still exhibits a phase transition in dimensions $d \geq 3$, but allows for an analytic treatment of the free energy in the infinite-volume limit. In their original work, Berlin and Kac used variations of the steepest-descent method to compute the limiting free energy, identify the phase transition in the three-dimensional lattice case, and argue that the limiting distribution of the model below the critical temperature is non-Gaussian. In a short note, Molchanov and Sudarev \cite{MS75} claimed that this method yields further results on the limiting distribution of the spherical model, but the proofs never appeared. In this article, we rigorously prove a subset of these results using probabilistic tools (cf. Theorem \ref{thm:spherical}), which we believe also helps clarify the conceptual picture.

An alternative approach, which implicitly leads to the results of our theorem in \cite{BD87} under unverified technical assumptions, is the so-called \emph{mean-spherical approach}, in which one adds a mass term ($s \sum_x \theta_x^2$) to the Hamiltonian rather than conditioning the spin configuration. The parameter $s$ is then chosen such that the average value of $\ntwo{\theta}^2$ equals $n^d$. This method was proposed by Lewis and Wannier \cite{LW52} shortly after the original work \cite{SpherModel}. They observed that various quantities, such as correlations and the free energy, are significantly simpler to compute in this formulation. While the method works well in the high-temperature regime, it was soon realized \cite{LW53_corr, Lax55} that it cannot be directly applied in the zero-magnetic-field setting below the critical temperature, as discrepancies with the original spherical model arise. However, in the presence of a nonzero external magnetic field $h \neq 0$, it was shown \cite{YW65} that the thermodynamic limits of the spherical and mean-spherical models agree, and that the zero-magnetic-field case of the spherical model can be recovered by first taking the thermodynamic limit and then sending $h \rightarrow 0$. Since then, this method has appeared in several works \cite{KT77, Shcherbina1988, BD87}, although none seem to provide a fully rigorous proof of the critical and low-temperature regimes for the zero-magnetic-field case.

Finally, the low-temperature regime of the spherical model can be viewed as a toy model for condensation phenomena. In this broader context, \cite{Lukkarinen2020} investigates a complexified version of the spherical model below the critical temperature. The author establishes convergence of local correlation functions, which, combined with tightness, implies local convergence in law in the aforementioned regime. While this proof also applies to the usual spherical model, it does not suffice to prove convergence at the critical temperature.

Let us now state our main result about the limiting distribution of the spherical model and explain the reasoning leading to the proof.
\begin{theorem}[Infinite-volume limit of the spherical model]
\label{thm:spherical}
    Let $d\geq 2$. The spherical model on $\T^d_n$ at inverse temperature $\beta > 0$, $\theta = (\theta_x)_{x \in \T^d_n}$, converges in law, uniformly over compact subsets of $\Z^d$ (viewed as subsets of $\T^d_n$ for each $n$), as $n \rightarrow \infty$ to: 
    \begin{enumerate}
        \item $\beta < \beta_c$: a massive Gaussian free field (GFF) on $\Z^d$ scaled by $1/\sqrt{\beta}$ with the mass $m^2> 0$ depending on $\beta$ and $d$ in a specific way;
        \item $\beta = \beta_c$: a GFF on $\Z^d$ scaled by $1/\sqrt{\beta}$;
        \item $\beta > \beta_c$: a GFF on $\Z^d$ scaled by $1/\sqrt{\beta}$ plus an independent constant random drift $\sqrt{\frac{\beta - \beta_c}{\beta}} X$ with $X$ being a Rademacher random variable.
    \end{enumerate}
    Here, $\beta_c = G_{\Z^d}(0,0)$. Furthermore, all local correlations of spins converge.
\end{theorem} 
We start by observing that for any $n$ and $m^2 > 0$, the constraint $\sum_x \theta_x^2 = n^d$ implies that
\begin{align*}
    \nu_{\T^d_n, \beta}[\d \theta] \propto \exp \bigg( -\frac{\beta}{2} \langle \theta, (-\Delta_{\T^d_n} + m^2)\theta\rangle\bigg) \mathrm{Unif}_{\sqrt{n^d} \Ss^{n^d-1}}[\d \theta].
\end{align*}
Hence, the law of $\sqrt{\beta}\theta$ is related to a massive GFF $\phi$ on $\T^d_n$ (for an arbitrary mass $m^2 > 0$) through conditioning on its norm to be equal to $\sqrt{\beta n^d}$, i.e., $\ntwo{\phi}^2 \coloneqq \sum_{x \in \T^d_n} \phi_x^2 = \beta n^d$ (see Proposition \ref{prop:rel_GFF_spher}). Now the idea is somewhat similar to the one of Lewis and Wannier \cite{LW52}, but from a probabilistic point of view: namely, we choose the mass $m_n^2$ in such a way that $\E[\ntwo{\phi}^2] = \beta n^d$ (or equivalently, $G_{\T^d_n, m_n^2}(0,0) = \beta$), which is natural due to the concentration properties of Gaussian measures. 

For large $n$ and fixed $m > 0$, $G_{\T^d_n, m^2}$ is roughly comparable to $G_{\Z^d, m^2}$; and thus, the question of the existence of a critical point boils down to whether there exists $m^2 \geq 0$ such that $G_{\Z^d, m^2}(0,0) = \beta$. Notice that this explains the difference between the 2D and higher-dimensional cases. Indeed, since in $d = 2$, $G_{\Z^d}(0,0)$ is infinite, by choosing the mass small enough, one could get an arbitrarily large value of the variance $G_{\Z^d, m^2}(0,0)$; whereas in $d \geq 3$, there is a maximal possible value $\beta_c \coloneqq G_{\Z^d}(0,0) < \infty$.

More precisely, we observe that in the high-temperature regime $\beta < \beta_c$, the aforementioned sequence of masses $(m_n^2)_n$ converges to a positive number $m^2>0$ in such a way that $$G_{\T^d_n, m_n^2}(0,0) \rightarrow G_{\Z^d, m^2}(0,0) = \beta.$$ In this case, the concentration of measure for $\ntwo{\phi}^2$ is sufficient to make the conditioning disappear in the limit, and we obtain a purely Gaussian limit. Somewhat refined versions of classical limiting theorems help us handle this phase rather directly. 

The low-temperature phase $\beta > \beta_c$ is already trickier. One can check that $$G_{\T^d_n, m_n^2}(x,y) \sim G_{\T^d_n}^{0\text{-}\mathrm{avg.}}(x,y) + \beta - \beta_c \overset{n \rightarrow \infty}{\sim} G_{\Z^d}(x,y) + \beta - \beta_c,$$ where $G_{\T^d_n}^{0\text{-}\mathrm{avg.}}$ is the correlation function of the zero-average GFF on torus. This heuristically allows us to restate the problem in terms of the zero-average GFF $\gamma$ plus an independent zero mode—a constant-in-space Gaussian drift of the form $Z \mathbf{1}_{\T^d_n}$—conditioned on the norm of the sum being $\sqrt{\beta n^d}$. Since the drift is constant across all sites, almost surely $\frac{1}{n^d}\ntwo{\gamma + Z\mathbf{1}_{\T^d_n}}^2 = \frac{1}{n^d}\ntwo{\gamma}^2 + \lvert Z\rvert^2$. Due to high concentration of $\frac{1}{n^d}\ntwo{\gamma}^2$ around $\beta_c$, conditioning on the norm of the sum forces $\lvert Z\rvert$ to become constant in the limit; and by symmetry, we obtain a Rademacher random variable for the zero mode. Making this precise is slightly more subtle, but after finding the right angle, the proof follows from combining basic concentration of measure results with relatively direct density bounds.

Finally, in the critical case, we observe that $G_{\T^d_n}^{0\text{-}\mathrm{avg.}}(x,y) \sim G_{\Z^d}(x,y)$. However, further refinement of this relation is necessary and is granted by Proposition \ref{prop:improved}: it shows that in $d \geq 3$, $$G_{\T^d_n}^{0\text{-}\mathrm{avg.}}(x,y) = G_{\Z^d}(x,y) + O(n^{2-d});$$ and in $d=3$, where we further need to determine the sign of the error on the diagonal, that $$G_{\T^d_n}^{0\text{-}\mathrm{avg.}}(x,x) < G_{\Z^d}(x,x)$$ uniformly for large $n$. Obtaining the sign was surprisingly tricky and is related to the so-called Madelung constant, arising in the electrostatic potential of certain ionic crystals—a quantity of interest in chemical physics since its introduction in the early 20th century \cite{Madelung}. See Subsection \ref{subsec:further_questions} for a discussion of this connection. 

To deduce convergence of local spin correlations from convergence in law (Corollary \ref{cor:spher_correlations}), we prove the existence of moments of single spherical spins (see Lemma \ref{lemma:spher_moments}), inspired by an analogous result for complexified spherical spins in \cite{Lukkarinen2020}.

\subsection{Spin \texorpdfstring{$O(N)$}{O(N)} model}
The spin $O(N)$ model, introduced (for general $N$) by Stanley in 1968 \cite{StanleyON}, is a fundamental model in statistical mechanics: for $N = 1$, one recovers the Ising model; for $N = 2$, the XY model; and for $N = 3$, the classical Heisenberg model. 
Here, we consider the infinite spin-dimensionality limit, $N \rightarrow \infty$, and its relation to the spherical model.

A connection between the spin $O(N)$ model in the $N \rightarrow \infty$ limit and the spherical model was first suggested by Stanley \cite{Stanley68} and later proved for any fixed non-critical temperature by Kac and Thompson \cite{SpherO(N)Model}; see also \cite{Shcherbina1988} and \cite{gough1993}, which together cover the case of the critical temperature. More precisely, combining the steepest-descent method with an approach similar to the aforementioned mean-spherical one, they showed the equality of the free energies per site of the two models: if $S \in \prod_{x \in \T^d_n} (\sqrt{N} \Ss^{N-1})$ is a configuration of the \emph{spin $O(N)$ model} on $\T^d_n$ at inverse temperature $\beta \geq 0$, i.e.,
\begin{align*}
    S \sim \frac{1}{Z_{\T^d_n,N,\beta}} \exp\bigg( \frac{\beta}{2} \sum_{x \sim y} {S}_x \cdot {S}_y\bigg) \prod_{x \in \T^d_n} \mathrm{Unif}_{\sqrt{N} \Ss^{N-1}} [\d {S}_x],
\end{align*}
then (for $\beta \neq \beta_c$ in \cite{SpherO(N)Model}),
\begin{align*}
    \lim_{n, N \rightarrow \infty} \frac{1}{n^d N} \log Z_{\T^d_n,N,\beta} = \lim_{n \rightarrow \infty} \frac{1}{n^d} \log Z_{\T^d_n,\beta}.
\end{align*}
for any ordering of the limits $N, n \rightarrow \infty$. 
In this article, we take a closer look at this connection and observe that it actually passes by the Gaussian free field. In the following theorem, we describe the local limit of the spin $O(N)$ model when we first take the infinite spin-dimensionality limit $N \to \infty$ and then the thermodynamic limit $\T_n^d \nearrow \Z^d$.
\begin{theorem}[Infinite spin-dimensionality thermodynamic limit of the spin \texorpdfstring{$O(N)$}{O(N)} model]
\label{thm:spinO(N)}
    Any projection on the finitely many ($M$) coordinates of the spin $O(N)$ model on $\T^d_n$ at inverse temperature $\beta > 0$, $\ovlineM{S} \coloneqq (S^i_x)_{x \in \T^d_n}^{i = 1, \ldots, M}$, converges in law as $N \rightarrow \infty$ to that of an $M$-vectorial massive GFF scaled by $1/\sqrt{\beta}$ with the mass $m_n^2$ depending on $\beta, d$ and $n$ in a specific way. The consequent local (uniform on compacts of $\Z^d$) infinite-volume distributional limit ($n \rightarrow \infty$) has the law of:
    \begin{enumerate}
        \item $\beta < \beta_c$: an $M$-vectorial massive GFF on $\Z^d$ scaled by $1/\sqrt{\beta}$ with the mass $m^2$ depending on $\beta$ and $d$ in a specific way;
        \item $\beta = \beta_c$: an $M$-vectorial GFF on $\Z^d$ scaled by $1/\sqrt{\beta}$;
        \item $\beta > \beta_c$: an $M$-vectorial GFF on $\Z^d$ scaled by $1/\sqrt{\beta}$ plus an independent constant random drift $\sqrt{\frac{\beta - \beta_c}{\beta}} \ovlineM{Z}$ with $\ovlineM{Z}$ being an $M$-dimensional standard normal vector.
    \end{enumerate}
    Here, $\beta_c = G_{\Z^d}(0,0)$.
\end{theorem}

As a corollary—maybe a little surprisingly—we see that at least in the order of limits we take, although the free energies per site of the two models agree at all temperatures, their limiting laws are, in fact, different in the low-temperature regime. We expect this to hold regardless of the order in which the limits are taken. The a posteriori explanation of this appearing discrepancy in the low-temperature regime is not particularly hard: differences in the law of a single degree of freedom—such as the drift part of the limiting field—vanish when computing the limiting free energy per site. 
\begin{corollary} \label{intro:cor:sphr_spin}
    In the high-temperature regime and at the critical temperature $\beta_c$, the local distributional limit of the spherical model and the corresponding limiting law (obtained by first taking $N \rightarrow \infty$ and then the volume of the torus to infinity) of a single spin-coordinate process of the spin $O(N)$ model agree; while in the low-temperature regime, they differ in the law of the global zero mode.
\end{corollary}

The proof of the theorem consists of two steps: first we show that as the spin-dimensionality $N$ tends to infinity, the law converges locally (in spin-dimensionality) to a vector-valued massive GFF with mass $m_n^2$ chosen in such a way that $G_{\T^d_n, m_n^2}(0,0) = \beta$; and then take the infinite volume limit of the Gaussian measure of $m_n^2$-massive GFF on torus. The first step follows from a multidimensional version of the classical local limit theorem, similarly to our proof of the high-temperature regime of the spherical model, the proof of the latter step is quite simple and requires only a few well-known results about the Green's functions on $\T^d_n$ and $\Z^d$.

We believe that the theorem should also hold in the opposite order of limits. In fact in the high-temperature regime this would follow directly from the work of A. Kupiainen \cite{o(n)expansion}; and we are hopeful that by combining some of his ideas with our approach, we would be able to extend this interchange of limits to all temperatures.

\subsection{Further questions and wider context} \label{subsec:further_questions}
 
To complete the probabilistic picture of the spherical model, which at the very least provides a statistical physics model that can be fully analyzed in all dimensions and phases, one would need to investigate its behavior across a variety of domains, under different boundary conditions, and in the presence of a magnetic field, ideally extending the analysis to more general graphs and interaction structures. The qualitative effects of adding boundary conditions or a constant magnetic field are not difficult to anticipate, and our proofs can be adapted to cover these cases; however, extending these methods to more general (also weighted) graphs may present new challenges.

As mentioned before, an interesting aspect of the critical case of the spherical model was the study of the zero-average Green's function on the torus and the question of whether it is smaller than the lattice Green's function. To this end, we analyzed the zero-average Green's function on the ``boundary'' (of the preimage of the torus in $\Z^d$ under the canonical projection), which is closely related to the so-called Madelung constant in chemical physics and mathematical chemistry \cite{Madelung, Mbook}. Let us explain this connection: one way to study the sign of $G^{0\text{-}\mathrm{avg.}}_{\T^d_n}(0, y)$ for $y \in (\partial[-\lfloor n/2\rfloor, \lfloor n/2\rfloor ]^d) \cap \Z^d$ is to take an appropriate scaling limit and examine the zero-average continuum Green's function $G_{\T^d}(0, x)$ for $x \in \partial[-1/2, 1/2]^d$, viewed as points on the torus $\T^d$ of side-length one. These quantities provide versions of the Madelung constant—for example, $d(2)$, $c(2)$, and $b(2)$ in \cite[(1.3.29) and Table 1.4]{Mbook} correspond precisely to $(2\pi)^2 G_{\T^3}(0, x)$ with $x = (1/2, 1/2, 1/2)$, $x = (0, 1/2, 1/2)$, and $x = (0, 0, 1/2)$, respectively. Determining or approximating the Madelung constants is a well-known problem. While explicit expressions exist—for instance, via Fourier series—the corresponding expansions are not absolutely convergent and may even diverge under natural truncations. In our case, we found a way to prove an inequality sufficient for our purposes directly on the lattice, although our bound is far from optimal. This naturally leads to several interesting questions of potential analysis. For example, it would be interesting to better understand the negativity region of the zero-average Green's function in all dimensions $d \geq 3$. 

Another perspective on the spherical model comes from its role as a simple framework for studying condensation phenomena \cite{Lukkarinen2020}. The connection between the spherical model and the ideal Bose gas—via their shared critical behavior \cite{PhysRev.166.152} and the emergence of Bose–Einstein condensation \cite{gough1993}—was first noted in the late 1960s and early 1990s, and has attracted renewed interest in recent years \cite{PhysRevResearch.1.023022, Lukkarinen2020}. It could be interesting to revisit this link, and to study other models that exhibit condensation, using the kind of probabilistic limit theorems and concentration-of-measure techniques we apply here, to see what extra insight such an approach might offer.

We now turn to questions concerning the spin $O(N)$ model. To complete the picture of its limiting distribution, the primary objective would be to find a way of interchanging the order of the infinite spin-dimensionality and infinite-volume limits at all temperatures. We expect this to be within reach.

In a broader context, an intriguing question is whether our approach could provide new insights into Polyakov's conjecture \cite{POLYAKOV} on the exponential decay of correlation functions in 2D spin $O(N)$ models with $N \geq 3$. While it is difficult to see how our techniques might help to resolve the conjecture, one could hope to extract useful quantitative bounds on the critical temperatures in the 2D spin $O(N)$ model as $N \to \infty$ (cf. \cite{o(n)expansion}).

Finally, our results for the large-$N$ limit of the spin $O(N)$ model—particularly in light of Proposition~\ref{prop:dens_conv}—can also be placed in the context of propagation of chaos \cite{SznitmanChaos}. The classical example is the Poincaré lemma, which states that the marginals of the spherical measure in high dimensions converge to independent Gaussians. In our setting, we observe that across all temperature regimes, the spin $O(N)$ measures are chaotic: as $N \to \infty$, the projection of the measure to any finite set of coordinates converges to a product measure. In principle, one could extend this analysis to spin models with additional non-quadratic interaction terms—such as quartic terms in the Hamiltonian—though it is not yet clear whether such generalizations would reveal anything beyond the behavior already observed.

\subsection{Outline}
Section \ref{sec:setup} defines the precise setting of the spherical and spin $O(N)$ models, provides the necessary background on various versions of the GFF, and explains their connection to the two models. It also contains new results on the zero-average Green's function. Section \ref{sec:sphr} focuses on the spherical model: in particular, we prove Theorem \ref{thm:spherical} and establish convergence of correlations for any finite number of its spins. Section \ref{sec:O(N)} turns to the spin $O(N)$ model and proves Theorem \ref{thm:spinO(N)}. To keep the article concise, some proofs less central to the main narrative have been moved to Appendix \ref{sec:Appendix}.

\subsection{Acknowledgements}
We would like to thank A. Prévost for interesting discussions.
Both authors are supported by Eccellenza grant 194648 of the Swiss National Science Foundation and are members of NCCR Swissmap.


\section{Setup and preliminaries}
\label{sec:setup}

In this section, we introduce the notation used throughout the article, along with several variants of the Gaussian free field that are closely related to our models of interest, namely the spherical and spin $O(N)$ models. We also recall some properties of these GFFs and their corresponding Green’s functions (their correlation structures), which will be relevant for our proofs.

The only new contribution of this section is Proposition \ref{prop:improved}, which provides refined bounds on the zero-average Green's function on the discrete torus and compares it with the Green's function on $\Z^d$.

\subsection{Notation, definitions, and setup of our two models}

Let $d \geq 2$, $n \in \N$, and $G$ be either the discrete $d$-dimensional torus of side length $n$, i.e., $G = \T^d_n \coloneqq (\Z/n\Z)^d$, or $\Z^d$, both viewed as a graph with the canonical nearest-neighbor structure. For $x,y \in G$, we write $\d(x,y)$ for the graph distance and $\lvert x-y \rvert$ for the Euclidean distance between the vertices. $x\sim y$ denotes two neighboring sites of $G$, and sums over $x\sim y$ run over pairs $(x, y)$ such that $\{x,y\}$ is an edge in $G$. 
Given $A\subset G$, we write $A^c = G\setminus A$ for the complement of $A$ in $G$, $\mathbf{1}_A$ for the vector $(1)_{x \in A}$, and, analogously, for any given $w\in \R^{G}$ we set $w_A \coloneqq (w_x)_{x \in A}$.

If $G = \T^d_n$, to highlight the dependence on $n$, we often write $\Lambda_n$ instead. When considering $A \subset \Z^d$, we also view it as a subset of $\Lambda_n$ for each $n$ (by canonically projecting it onto $\Lambda_n$). Conversely, for $U \subset \Lambda_n$, we denote by $\hat{U}$ the unique preimage of $U$ under the canonical projection contained in $[-n/2, n/2)^d$—in particular, $\hat{\Lambda}_n = [-n/2, n/2)^d \cap \Z^d$. Moreover, let $\langle\cdot, \cdot\rangle$ be the canonical inner product on $\R^{A}$ and $\ntwo{\cdot}$ be the associated norm, in order to distinguish them from the Euclidean inner product and norm on $\R^k$ for $k \in \N$, which we denote by $\cdot$ and $|\cdot|$, respectively.

For the reader's convenience, we recall the definitions of the spherical and spin $O(N)$ models in the setting of interest.
\begin{definition}[Spherical model]
\label{def:spher}
    Let $\beta > 0, n \in \N$. We call $\theta \coloneqq (\theta_x)_{x\in \Lambda_n} \in \sqrt{n^d} \;\Ss^{n^d-1}$ a configuration of the {spherical model} on $\Lambda_n$ at inverse temperature $\beta$ if $\theta \sim \nu_{\Lambda_n,\beta}$ with
    \begin{align*}
        \nu_{\Lambda_n,\beta} [\d \theta] = \frac{1}{Z_{\Lambda_n, \beta}} \exp\bigg( \frac{\beta}{2} \sum_{x\sim y} \theta_x\theta_y \bigg) \mathrm{Unif}_{\sqrt{n^d}\;\Ss^{n^d-1}}[\d \theta],
    \end{align*}
    where $Z_{\Lambda_n, \beta}$ is the normalizing constant.
\end{definition}
\begin{definition}[Spin $O(N)$ model]
\label{def:spinO(N)}
    Let $\beta > 0, n, N \in \N$. We call $S \coloneqq (S_x)_{x\in \Lambda_n} \in (\sqrt{N} \: \Ss^{N-1})^{n^d}$ a configuration of the {spin $O(N)$ model} on $\Lambda_n$ at inverse temperature $\beta$ if $S \sim \mu_{\Lambda_n, N, \beta}$ with
    \begin{align*}
        \mu_{\Lambda_n, N,\beta} [\d S] = \frac{1}{Z_{\Lambda_n, N, \beta}} \exp\bigg( \frac{\beta}{2} \sum_{x\sim y} S_x \cdot S_y \bigg) \prod_{x \in \Lambda} \mathrm{Unif}_{\sqrt{N}\;\Ss^{N-1}}[\d S_x],
    \end{align*}
    where $Z_{\Lambda_n, N, \beta}$ is the normalizing constant.
\end{definition}
In the sequel we will always refer to a configuration of the spherical model as $\theta$ and of the spin $O(N)$ model as $S$.


\subsection{Massive GFF, its vectorial descendant and zero-average GFF}
\label{subsec:GFF}
In the present subsection, we review several variants of the GFF, which, as we will see later, are closely related to our models of interest both in the finite setting and in the distributional limit. We begin by defining the vector-valued (massive) Gaussian free field on a discrete torus with Dirichlet boundary conditions. 
\begin{definition}[$m^2$-massive $N$-vectorial GFF]
\label{def:GFF}
    Let $m \geq 0, d\geq 2$, $N, n \in \N$, and $G = \T^d_n$. Let $U \subset G$ be an arbitrary subset (possibly empty if $m>0$). We call a random function $\phi: G \rightarrow \R^N$ an $m^2$-massive $N$-vectorial Gaussian free field (GFF) on $G$ with Dirichlet boundary condition on $U$ if $\phi_x = 0$ for all $x \in U$ and
    \begin{align} \label{eq:gff_def}
        \P[\d \phi] \propto \exp \bigg(- \frac{1}{4}\sum_{x\sim y} |\phi_x - \phi_y|^2 - \frac{m^2}{2} \sum_{x \in G} |\phi_x|^2\bigg) \prod_{x \in G\setminus U} \d \phi_x.
    \end{align}
    If $m^2 = 0$, we drop the term ``massive'' from the name; if $N=1$, the ``$N$-vectorial'' part. If $U = \emptyset$, we drop the ``Dirichlet'' part altogether. 
\end{definition}
\begin{remark}
    From the expression in \eqref{eq:gff_def}, it is evident that the coordinate processes of $\phi$ are i.i.d., each being an $m^2$-massive GFF on $G$ with Dirichlet boundary condition on $U$. We therefore restrict to the scalar case when studying the covariance structure of the field and when discussing component-wise properties of $\phi$. 
\end{remark}

When $N = 1$, we can rewrite the expression in the exponent of \eqref{eq:gff_def} as follows
\begin{align}
\label{eq:discrGaussform}
    \frac{1}{2}\sum_{x\sim y} |\phi_x - \phi_y|^2 + m^2\sum_{x} |\phi_x|^2 = \big\langle\phi, (-\Delta_{G} + m^2) \phi\big\rangle,
\end{align}
where $\langle \cdot, \cdot\rangle$ stands for the inner product on $\R^{G}$ and 
\begin{align*}
    -\Delta_{G} f (x) = 2d f(x) - \sum_{y \in G: y\sim x} f(y) \quad \text{for any } \,x \in G, \, f: G \rightarrow \R
\end{align*}
is the discrete Laplacian on $G$. When restricted to the set $F_U$ of functions $f: G \rightarrow \R$ with $f(x) = 0$ for all $x \in U$, we denote it by $-\Delta_{G\setminus U}$ (or $-\Delta_{U^c}$) and call it the discrete Laplacian with Dirichlet boundary condition on $U$. As a straightforward consequence of \eqref{eq:discrGaussform}, we see that $-\Delta_{U^c}$ is invertible on $F_U$—note that we assume that $U \neq \emptyset$ or $m^2 >0$—and its inverse $(G_{G\setminus U, m^2}(x,y))_{x,y \in G}$, which is then also the covariance matrix of $\phi$, is the massive (if $m^2>0$) Green's function on $G$ with Dirichlet boundary condition on $U$. If $m^2 = 0$, we often drop the corresponding subscript and write simply $G_{G\setminus U}$.

It is a well-known fact (see for instance \cite[Section 1]{Rod17}) that $G_{U^c, m^2}$ (where $U^c = G\setminus U$) is related to random walks in the following way: for all $x,y \in G$, $\alpha \coloneqq \frac{m^2}{2d+m^2}$
\begin{equation}
\label{eq:Gfct_RW}
\begin{aligned}
    G_{U^c, m^2}(x,y) 
    &\coloneqq \frac{1-\alpha}{2d} \sum_{k \geq 0}\P^x_{\alpha, U}[X_k = y] = \frac{1-\alpha}{2d} \sum_{k \geq 0} (1-\alpha)^k \P^x_{0}[X_k = y, k< H_U], 
\end{aligned}
\end{equation}
where $\P_{\alpha,U}^x$ denotes the law of a random walk $X$ on the graph $G\cup \{\dag\}$ (obtained from $G$ by adding an edge connecting each vertex of $G$ to $\dag$) starting at $x \in G$ with transition probabilities 
\begin{align*}
    p_{u,v} = \frac{1-\alpha}{2d} \ind_{\{u\sim v\}}, \; p_{u,\dag} = \alpha, \; p_{w,w} = 1, \; \text{for all } u \in U^c, v\in G, w\in U \cup \{\dag\},
\end{align*}
$H_U\coloneqq \inf\{k \geq 0: X_k \in U\}$ is the first hitting time of $U$, and $\P^x_{0}$ is the law of a simple random walk (SRW) on $G$ started at $x \in G$. 

Observe that the right-hand side of \eqref{eq:Gfct_RW} remains well-defined (finite) if $G = \Z^d$ for $d\geq 2$—we denote it by $G_{\Z^d\setminus U, m^2}$—under the assumption that $U\neq \emptyset$ or $m^2 > 0$ if $d=2$, since a SRW on $\Z^d$ is transient for $d\geq 3$, while it is recurrent for $d=2$. This also explains our assumption on $m^2, U$ if $G = \T^d_n$, since a SRW on $\T^d_n$ is recurrent. 
Under this assumption on $U$ and $m^2$, we can now define an \emph{$m^2$-massive $N$-vectorial GFF on $\Z^d$ with Dirichlet boundary condition on $U \subset \Z^d$} as a vector of $N$ i.i.d. centered Gaussian processes with covariance matrix $G_{\Z^d\setminus U, m^2}$ each. 

Note furthermore that from \eqref{eq:Gfct_RW}, we get that for any fixed finite $U \subset \Z^d$ and $m^2 \geq 0$ ($m^2 > 0$ if $U=\emptyset$ and $G = \T^d_n$ or $G = \Z^2$), as $G\nearrow \Z^d$, an $m^2$-massive $N$-vectorial GFF $\phi$ on $G$ with Dirichlet boundary condition on $U$ converges in law, uniformly on compact subsets of $\Z^d$, to an $m^2$-massive $N$-vectorial GFF $\psi$ on $\Z^d$ with Dirichlet boundary condition on $U$. Moreover, we can extend the argument to a sequence of $m_n^2$-massive GFFs, each defined on $\Lambda_n = \T^d_n$, under the assumption that $m_n^2\rightarrow m^2 > 0$ as $n\rightarrow \infty$. Indeed, \eqref{eq:Gfct_RW} guarantees the existence of $C>0$, uniform in $n$, such that $\lvert G_{\Lambda_n\setminus U, m_n^2} - G_{\Lambda_n \setminus U, m^2}\rvert \leq C \lvert \alpha - \alpha_n\rvert G_{\Lambda_n \setminus U, m^2}$. The latter, in turn, converges to zero uniformly over any compact subset of $\Z^d$ as $n$ tends to infinity. 

Let us end the discussion of a massive GFF in this subsection by stating one more of its canonical properties that will be frequently used in the sequel. Its proof may be found in \cite[Lemma 1.1]{Rod17} (see also \cite[Lemma 3.1]{Biskup_notes} for the massless case).
\begin{proposition}[Domain Markov property of the (massive) GFF]
\label{prop:domainMP}
    Let $G$ be either $\T^d_n$ or $\Z^d$, $U \subset G$, and $m^2 \geq 0$ ($m^2 > 0$ if $U=\emptyset$ and $G = \T^d_n$ or $G = \Z^2$). Let $\phi$ be an $m^2$-massive GFF on $G$ with Dirichlet boundary condition on $U$. For $K \subset G$, set
     \begin{align*}
        h^{\phi_K}_x = \E[\phi_x |\F_K] \quad \text{for all } x \in G,
    \end{align*}
    where $\F_K \coloneqq \sigma(\phi_x: x \in K)$. Then $h^{\phi_K} = (h^{\phi_K}_x)_{x\in G}$ is discrete massive-harmonic in $G\setminus (K \cup U)$ (i.e., $(-\Delta_{G} + m^2) h^{\phi_K} (y) = 0$ for any $y \in G\setminus (U\cup K)$) with boundary values determined by $h^{\phi_K}_x = \phi_x$ for any $x \in K\cup U$. 
    In terms of the random walk $X \sim \P_{\alpha,U}^x$ as in \eqref{eq:Gfct_RW}, it is given by
    \begin{align*}
        h^{\phi_K}_x = \sum_{y \in K} \P^x_{\alpha, U}[H_K < \infty, X_{H_K} = y] \:\phi_y.
    \end{align*}
    Moreover, the field $\phi^{G\setminus K} \coloneqq (\phi_x - h^{\phi_K}_x)_{x\in G}$ is independent of $\F_K$ and has the law of an $m^2$-massive GFF on $G$ with Dirichlet boundary condition on $K \cup U$. 
\end{proposition}

We conclude this subsection by introducing another version of a (massless) GFF on the discrete torus $\Lambda_n = \T^d_n$ in dimension $d \geq 3$: the so-called zero-average GFF $\gamma$. The name reflects the constraint imposed to reduce the dimensionality of the space on which the GFF measure is defined—namely, that $\sum_x \gamma_x = 0$. 
\begin{definition}[Zero-average GFF on torus]
\label{def:zeroavGFF}
    Let $n \in \N, d \geq 3$. We call a random function $\gamma: \Lambda_n \rightarrow \R$ a {zero-average Gaussian free field on $\Lambda_n$} if it is distributed according to
    \begin{align}\label{eq:0-avf_GFF_def}
        \P[\d \gamma] \propto \Big(- \frac{1}{4}\sum_{x\sim y} (\gamma_x - \gamma_y)^2 \Big) \delta_0 \Big(\sum_x \gamma_x \Big) \prod_{x\in \Lambda_n} \d \gamma_x.
    \end{align}
\end{definition}
To find the covariance matrix of $\gamma$ while avoiding direct computation with the Dirac measure, let $(\eta_w)_{w \in [0,n)^d \cap \Z^d}$ be the eigenvalues of $-\Delta_{\Lambda_n}$ and $(q^w)_w$ be the corresponding orthogonal eigenvectors (cf. Lemma \ref{lemma:massive_EVs}); in particular, $q^0 = \frac{1}{\sqrt{n^d}} \1_{\Lambda_n}$. Set $Q \coloneqq (q^0, (q^w)_{w \neq 0})$ and observe that  
\begin{align*}
    \frac{1}{2} \sum_{x\sim y} (\gamma_x - \gamma_y)^2 =  \big\langle Q^T \gamma, \,\mathrm{diag}\big(0,  (\eta_w)_{w \neq 0}\big) Q^T\gamma \big\rangle; &&
    \frac{1}{\sqrt{n^d}} \sum_{x \in \T^d_n} \gamma_x = \langle q^0, \gamma\rangle = (Q^T \gamma)_0.
\end{align*}
This allows us to conclude that the zero-average GFF has the same law as the orthogonal transformation of a vector of $n^d-1$ independent centered normal random variables $Z_w$ with variance $1/\eta_w$ for $w \in [0,n)^d \cap\Z^d \setminus \{0\}$:
\begin{align}\label{eq:gamma_law_iid}
    \gamma \overset{\text{law}}{=} Q \big(0, (Z_w)_{w\neq 0}\big)^T,
\end{align}
which in turn implies that the covariance matrix of $\gamma$, known as the \emph{zero-average Green's function}, is given by
\begin{align*}
    \big(G_{\Lambda_n}^{0\text{-}\mathrm{avg.}}(x,y)\big)_{x,y \in \Lambda_n} \coloneqq Q \,\mathrm{diag}\bigg(0, \Big(\frac{1}{\eta_k} \Big)_{w \in [0,n)^d \cap\Z^d \setminus \{0\}} \bigg) \,Q^T.
\end{align*} 
 We recall and prove several of its properties in Subsection \ref{subsec:Greens_fcts}.


\subsection{Massive GFF, its vectorial descendant, and the zero-average GFF: relation to the spherical and spin \texorpdfstring{$O(N)$}{O(N)} models}
\label{subsec:rel_GFF_spher_spinO(N)}

Let $d \geq 2$, $n$ and $N \in \N$. Recall that $\Lambda_n = \T^d_n$ is the $d$-dimensional discrete torus. 

\begin{proposition}[Massive and zero-average GFFs: relation to the spherical model]
\label{prop:rel_GFF_spher}
    Let $\nu_{\Lambda_n, \beta}$ be the spherical model on $\Lambda_n$ at inverse temperature $\beta > 0$. Then, for any $m^2>0$, the conditional law of an $m^2$-massive GFF $\phi$ on $\Lambda_n$ given $\norm{\phi} = \sqrt{\beta n^d}$ is exactly $\nu_{\Lambda_n, \beta}$. That is,
    \begin{align*}
        \nu_{\Lambda_n, \beta} = \mathrm{Law} \bigg(\frac{\phi}{\sqrt{\beta}} ~\bigg\vert~ \norm{\phi} = \sqrt{\beta n^d} \bigg), 
    \end{align*}
    Moreover, if $d\geq 3$, $\gamma$ is a zero-average GFF on $\Lambda_n$, and $c$ is an ``independent'' Lebesgue constant, i.e., the ``law'' of $(\gamma, c)$ is $\mathcal{N} (0, G_{\Lambda_n}^{0\text{-}\mathrm{avg.}}) \times \d c$, then
    \begin{align*}
        \nu_{\Lambda_n, \beta} = \mathrm{Law} \bigg(\frac{\gamma + c \mathbf{1}_{\Lambda_n}}{\sqrt{\beta}} ~\bigg\vert~ \norm{\gamma + c \mathbf{1}_{\Lambda_n}} = \sqrt{\beta n^d} \bigg). 
    \end{align*}
\end{proposition}
\begin{proof}
    Using polar coordinates, it is easy to obtain an expression for the conditional density of $\frac{\phi}{\sqrt{\beta}}$ given $\norm{\phi} = \sqrt{\beta n^d}$ at $v \in \R^{\Lambda_n}$ with $\norm{v} = \sqrt{n^d}$:
    \begin{align*}
        f_{\frac{\phi}{\sqrt{\beta}}~\big\vert \norm{\phi} = \sqrt{\beta n^d}} (v) \d v 
        &\propto \exp\bigg(-\frac{\beta}{2} \big\langle v, (-\Delta_{\Lambda_n} + m^2) v \big\rangle\bigg) \mathrm{Unif}_{\sqrt{n^d} \: \Ss^{n^d-1}} (\d v) \\
        &\propto \exp\bigg(\frac{\beta}{2} \sum_{x\sim y} v_x v_y \bigg) \mathrm{Unif}_{\sqrt{n^d} \: \Ss^{n^d-1}} (\d v) \propto \nu_{\Lambda_n, \beta}.
    \end{align*}
    In the last line, we used that $\langle v, v \rangle = n^d$.
    
    As for the second statement, note that upon conditioning, we indeed have a probability measure. To better understand the latter, let us instead consider the orthogonal transformation of the sum—this does not affect the conditioning. More precisely, let $Q$ be the orthogonal matrix diagonalizing $G_{\Lambda_n}^{0\text{-}\mathrm{avg.}}$, that is, $Q^T G_{\Lambda_n}^{0\text{-}\mathrm{avg.}} Q$ is diagonal. It is clear that the ``law'' of $\gamma_{c, Q} \coloneqq Q^T(\gamma + c \mathbf{1}_{\Lambda_n})$ is given by $\frac{1}{\sqrt{n^d}} \d c \times \bigtimes_{k=2}^{n^d} \mathcal{N} (0, \eta_k^{-1})$, where $(\eta_k)_{k \geq 2}$ are the non-zero eigenvalues of $-\Delta_{\Lambda_n}$. Restricted to a finite ball and normalized accordingly, it is a probability measure that we denote by $\rho$. For our purposes we can consider, e.g., $\mathcal{B}^{n^d}(0, 2\sqrt{\beta n^d})$. Then, analogously to the above, for any Borel-measurable $B \subset \R^{\Lambda_n}$,
    \begin{align*}
        \rho\Big[B ~\Big\vert~ \ntwo{\gamma_{c, Q}} = \sqrt{\beta n^d}\Big] &\propto \int \ind_{B} (v) \prod_{k=2}^{n^d} \exp\left(-\frac{\beta}{2} \eta_k v_k^2 \right) \mathrm{Unif}_{\sqrt{n^d} \: \Ss^{n^d-1}} (\d v) \\
        &= \int \ind_{Q B} (\sigma) \exp\bigg(\frac{\beta}{2} \sum_{x\sim y} \sigma_x \sigma_y \bigg) \mathrm{Unif}_{\sqrt{n^d} \: \Ss^{n^d-1}} (\d{\sigma}) \propto \nu_{\Lambda_n, \beta} [Q B]
    \end{align*}
    In the second line we used the substitution $v = Q^T \sigma$ so that $\sum_{k=2}^{n^d} \eta_k v_k^2 = \langle \sigma, -\Delta_{\Lambda_n} \sigma\rangle$, and the fact that the uniform measure on the sphere is invariant under orthogonal transformations. 
\end{proof}

We also have a similar statement for the spin $O(N)$ model, which was implicitly used already in \cite{o(n)expansion} and proven in a slightly different setting in \cite[Proposition 2.3]{AGS22}. We omit its proof as it is fully analogous to the proof of the first statement of the previous proposition and can also be easily adjusted from the latter reference.
\begin{proposition}[Relation between an $N$-vectorial massive GFF and the spin \texorpdfstring{$O(N)$}{O(N)} model]
\label{prop:mGFF-O(N)}
    Let $m^2> 0$ and $\phi$ be an $N$-vectorial $m^2$-massive GFF on $\Lambda_n$, $\beta > 0$. The conditional law of $\phi/\sqrt{\beta}$ given $|\phi_x| = \sqrt{\beta N}$ for all $x\in \Lambda_n$ is that of the spin $O(N)$ model on $\Lambda_n$ at inverse temperature $\beta$. 
\end{proposition}


\subsection{Preliminary estimates on the Green's functions}
\label{subsec:Greens_fcts}
This subsection mainly recalls but also proves some further results about Green's functions and their eigenvalues. We start off with the massive Green's function and then turn to the zero-average Green's function. 

\subsubsection{Massive Green's function}
We begin with some canonical results, whose proofs can be found, for example, in \cite[Section 1]{Rod17} and \cite[Proposition 8.30]{friedli_velenik_2017}.  
\begin{lemma}[Properties of massive Green's function]
\label{lemma:prop:lapl,grfct}
    Let $n \in \N$, $G = \T^d_n$ or $\Z^d$, and $U$ be a subset of $G$, possibly $U = \emptyset$. Let $m^2\geq 0$ if $U \neq \emptyset$ or $G = \Z^d$ for $d\geq 3$ and $m^2 > 0$ otherwise. 
    \begin{enumerate}
        \item For any finite subset $K \subset G\setminus U \eqqcolon U^c$,
        \begin{align}
        \label{eq:Gfct_decomp}
            G_{U^c, m^2}(x,y) = G_{(U\cup K)^c, m^2} + \E_{\alpha, U}^x[ \ind_{\{H_K < \infty\}} G_{U^c, m^2}(X_{H_K},y)]
        \end{align}
        with the notation of \eqref{eq:Gfct_RW}.
        \item If $m^2>0$, there exist two constants $c, C> 0$ depending only on $m^2$ and $d$ such that 
        \begin{align}
        \label{eq:Gfct_expdec}
            G_{U^c, m^2}(x,y) \leq C e^{-c \lvert x-y\rvert}, \quad\text{for all }\, U\subset G, x, y \in G. 
        \end{align}
    \end{enumerate}
\end{lemma}

We will frequently find it helpful to calculate with the basis of the eigenfunctions of the massive Green's function on $\T^d_n$. 
\begin{lemma}[Eigenvalues of massive Green's function]
\label{lemma:massive_EVs}
    Let $n \in \N, m^2 > 0$, and $U$ be a finite subset of $\Z^d$ viewed as a subset of $\Lambda_n$ for each $n$. Set $U^c \coloneqq \Lambda_n \setminus U$ and $u_n \coloneqq \lvert U^c\rvert$.
    \begin{enumerate}
        \item $-\Delta_{\Lambda_n}$ has one zero eigenvalue $\eta_0 = 0$ and $(n^d-1)$ positive eigenvalues given by 
        \begin{align*}
            \eta_{w} = 2 \sum_{i=1}^d \Big(1-\cos\Big( 2\pi \frac{w_i}{n} \Big)\Big) \; \text{ for } w=(w_i)_{i=1}^d \in [0, n)^d \cap \Z^d \setminus \{0\}.
        \end{align*}
        The corresponding eigenvectors $(q^w)_{w\in [0,n)^d \cap \Z^d}$ generating an orthonormal system are given by
        \begin{align*}
            q^w_x = \frac{1}{\sqrt{n^d}} \prod_{i=1}^d \Big(\cos\Big( 2\pi \frac{x_i w_i}{n}\Big) + \sin\Big( 2\pi \frac{x_i w_i}{n}\Big)\Big) \;\text{ for all } x \in [0, n)^d \cap \Z^d.
        \end{align*}
        \item Let $(\eta_k)_{k \geq 2}$ \footnote{Here and in the sequel, we slightly abuse notation and write both $(\eta_k)_{k=1}^{n^d}$ (when ordered non-decreasingly) and $(\eta_x)_{x \in [0, n)^d \cap \Z^d}$ (when indexed by the domain) for the eigenvalues of $-\Delta_{\Lambda_n}$\label{footnote:abuse_EVs}} be the positive eigenvalues of $-\Delta_{\Lambda_n}$ ordered non-decreasingly. Then, 
        \begin{equation} \label{eq:EV_asympt}
            \eta_{k} = \Theta\Big( \frac{k^{2/d}}{n^2} \Big).
        \end{equation}
        \item The eigenvalues of $G_{\Lambda_n, m^2}$ are given by $(1/(m^2 + \eta_x))_{x \in [0,n)^d \cap \Z^d}$, while the eigenvalues $(\tau_i)_{i=1}^{u_n}$ of $G_{U^c, m^2}\vert_{U^c \times U^c}$ ordered non-increasingly satisfy 
        \begin{align} \label{eq:unif_bdsEVs}
            \tau_i \in \bigg[\frac{1}{m^2 + 4d}, \frac{1}{m^2}\bigg] \quad \text{for all } 1\leq i \leq u_n.
        \end{align}
        \end{enumerate}
\end{lemma}
\begin{proof}
    The eigenvalues of $-\Delta_{\Lambda_n}$ can be found, for instance, in \cite[Example 5.17]{aldous-fill-2014} or easily derived from \cite[Appendix A]{SpherModel}; the eigenvectors are also classical.
    
    Using explicit formulae for $(\eta_x)_{x \in [0, n/2]^d \cap \Z^d}$ (any remaining $x$ gives an eigenvalue coinciding with one indexed by some point in $[0, n/2]^d \cap \Z^d$), we obtain $\eta_x = \Theta(\lvert x\rvert^2/n^2)$. This together with symmetries of the torus and properties of $\cos$-function implies that $\eta_{k} = \Theta \big( k^{2/d}/n^2 \big)$. 
    
    The third point follows immediately since $G_{\Lambda_n, m^2}$ is the inverse of $-\Delta_{\Lambda_n} + m^2$, while $G_{U^c, m^2}\vert_{U^c \times U^c}$ is the inverse of $(-\Delta_{U^c} + m^2)\vert_{U^c \times U^c}$—see discussion following \eqref{eq:discrGaussform}. From \eqref{eq:discrGaussform}, one furthermore sees that the smallest eigenvalue of $(-\Delta_{U^c} + m^2)\vert_{U^c \times U^c}$ is at least $m^2$, while the largest is upper bounded by $4d + m^2$.
\end{proof}
\begin{remark}
\label{rem:EV_indexing}
    Note that equivalently the eigenvalues and the corresponding eigenfunctions as well as their domains can be indexed by $[-n/2, n/2)^d\cap \Z^d$ instead of $[0,n)^d\cap \Z^d$ due to the invariance of cosine functions under translation of the argument by $2\pi \Z$. More precisely, for $w,x \in [0,n)^d\cap \Z^d$, $\eta_w = \eta_{\tilde w}$ and $q^w_x = q^{\tilde w}_{\tilde x}$, where $\tilde y_i = y_i$ if $y_i < n/2$ and $y_i - n$ otherwise for $y = x, w$. 
\end{remark}

\subsubsection{The zero-average Green's function}
We now proceed to the results on the zero-average Green's functions corresponding to the covariance structure of our zero-average GFF on $\Lambda_n = \T^d_n$. We start by recalling some known results stemming from \cite{Aba17}:
\begin{proposition}[Properties of the zero-average Green's function]
\label{prop:props:zeroAvGfct}
    Let $d \geq 3$, $n \in \N$. The zero-average Green's function $G_{\Lambda_n}^{0\text{-}\mathrm{avg.}}$ satisfies the following:
    \begin{enumerate}
        \item For any $y \in \Lambda_n$, $\sum_{x \in \Lambda_n} G_{\Lambda_n}^{0\text{-}\mathrm{avg.}}(x,y) = 0$.
        \item Let $\overline{X} = (\overline{X}_t)_{t \geq 0}$ be a continuous-time simple random walk on $\Lambda_n$. We denote its law when started at $x$ by $\P_{\Lambda_n}^x$. Then, for all $x,y \in \Lambda_n$,
        \begin{equation*}
            G_{\Lambda_n}^{0\text{-}\mathrm{avg.}}(x,y) = \frac{1}{2d} \int_0^\infty  \Big( \P^x_{\Lambda_n}[\overline{X}_t = y] - \frac{1}{n^d} \Big) \d t. 
        \end{equation*}
        \item For any $U \subsetneq \Lambda_n$, $x,y \in \Lambda_n$,
        \begin{align}
        \label{eq:0-avGfct_decomp}
            G_{\Lambda_n}^{0\text{-}\mathrm{avg.}}(x,y) = G_{U^c, 0}(x,y) + \E_{0}^x \big[ G_{\Lambda_n}^{0\text{-}\mathrm{avg.}}(X_{H_U},y)\big] - \frac{1}{2d n^d} \E_{0}^x[ H_U],
        \end{align}
        where $H_U= \inf\{n \geq 0: X_n \in U\}$ is the first hitting time of $U$ by a simple random walk $(X_k)_{k \in \N_0} \sim \P^x_{0}$ on $\Lambda_n$ started at $X_0 = x$.
        \item $G_{\Lambda_n}^{0\text{-}\mathrm{avg.}}$ converges to $G_{\Z^d}$ uniformly over compact subsets of $\Z^d$ (viewed also as subsets of $\Lambda_n$).
        \item There exists $C>0$ such that for all $x, y \in \Lambda_n$ and $n \in \N$ sufficiently large,
        \begin{align}
            \label{eq:zero-avGfct_polydec}
            \big\lvert G_{\Lambda_n}^{0\text{-}\mathrm{avg.}}(x,y) \big\rvert \leq C (\log n)^{3d/2} \d(x,y)^{2-d}. 
        \end{align}
        Recall that $\d(\cdot,\cdot)$ stands for the graph-distance on $\Lambda_n$.
       \end{enumerate}
\end{proposition}
For most of the results of this paper, the bound \ref{eq:zero-avGfct_polydec} would have been sufficient, however, not for the critical regime of the spherical model in dimensions $3$ and $4$. Therefore, we establish the improved estimates. 
\begin{proposition}[Finer estimates for the zero-average Green's function]
\label{prop:improved}
Let $d \geq 3, n \in \N$. The zero-average Green's function $G_{\Lambda_n}^{0\text{-}\mathrm{avg.}}$ satisfies the following:
    \begin{enumerate}
     \item There exists $C = C(d)>0$ such that for all $n$ sufficiently large and $y \in [-n/2, n/2)^d \cap \Z^d$, 
        \begin{align}
        \label{eq:0avgGfct_polydec}
            \lvert G_{\Lambda_n}^{0\text{-}\mathrm{avg.}}(0,y) - G_{\Z^d}(0,y) \rvert \leq C n^{2-d}.
        \end{align}
    In particular, for $\lvert y\rvert = \Theta(n)$ as above, $G_{\Lambda_n}^{0\text{-}\mathrm{avg.}}(0,y) = \mathcal{O}(n^{2-d})$. 
    \item  For $d=3$ and any $x \in \Z^d$,
        \begin{align}
            \label{eq:negativity_Greens_diff}
            G^{0\text{-}\mathrm{avg.}}_{\Lambda_n}(x,x) - G_{\Z^d}(x,x) = \Theta(n^{2-d}) < 0 \quad \text{uniformly for large $n$}.
        \end{align}
    \end{enumerate}
\end{proposition}
\begin{proof}
    The proof of the estimate \eqref{eq:0avgGfct_polydec} is rather technical and is given in Appendix \ref{A:zero_avg_polydecay}; it mainly follows the argument verifying \eqref{eq:zero-avGfct_polydec} in \cite[Proposition 1.5]{Aba17}, but with greater precision. 
    
    The inequality \eqref{eq:negativity_Greens_diff} for $d=3$ can be obtained as follows. 
    Let $x = 0$, and recall that $\hat{\Lambda}_n = [-n/2, n/2)^d \cap \Z^d$.
    Consider $\hat U \coloneqq (\partial [-\lfloor n/2\rfloor, \lfloor n/2\rfloor]^d) \cap \Z^d$, and let $U$ be its canonical projection onto $\Lambda_n$. Then, since the random walk on torus killed upon entering $U$ and the one on $\Z^d$ (started in the interior of $U$) killed on $\hat{U}$ have the same law, by \eqref{eq:0-avGfct_decomp} and \eqref{eq:Gfct_decomp}:
    \begin{align*}
        G^{0\text{-}\mathrm{avg.}}_{\Lambda_n} (0,0) &= G_{U^c,0}(0,0) + \E_{0}^0[ G_{\Lambda_n}^{0\text{-}\mathrm{avg.}}(X_{H_U},0)] - \frac{1}{2d n^d} \E_{0}^0 [H_U]\\
        &= G_{\Z^d}(0,0) - \E_{0}^{0,\Z^d} [G_{\Z^d}(X_{H_U},0) \ind_{\{H_U < \infty\}}] +\E_{0}^0[ G_{\Lambda_n}^{0\text{-}\mathrm{avg.}}(X_{H_U},0)] - \frac{1}{2d n^d} \E_{0}^0 [H_U]\\
        &= G_{\Z^d}(0,0) - \E_{0}^{0,\Z^d} [G_{\Z^d}(X_{H_U},0)] +\E_{0}^0[ G_{\Lambda_n}^{0\text{-}\mathrm{avg.}}(X_{H_U},0)] - \frac{1}{2d n^d} \E_{0}^{0,\Z^d} [H_U]
    \end{align*}
    (cf. \cite[(15)]{das2018extremal}). By a standard argument (see e.g. \cite[(1.21)]{lawler2012intersections}) since $\lvert \sum_{k=1}^n X_k\rvert^2 - n$ is a martingale under $\P^{0, \Z^d}_0$, $\E_{0}^{0,\Z^d} [H_U] = \Theta(n^2)$. By \cite[Theorem 1.5.4]{lawler2012intersections} we further know that for $d\geq 3$, 
    \begin{align*}
        G_{\Z^d}(0,x) \sim C(d) \lvert x\rvert^{2-d} \quad \text{ with } \quad C(d) = \frac{1}{2d} \frac{d}{2} \Gamma\Big( \frac{d}{2} - 1\Big) \pi^{-d/2}
    \end{align*}
    Combined with Proposition \ref{A:prop:0avg_bdry_bound}, which states that for all sufficiently large $n$, $G^{0\text{-}\mathrm{avg.}}_{\Lambda_n}(0,y)$ for any $y \in U$ is bounded from above by ${c_*}/{n}$ for some absolute constant $c_* < \frac{3^{2/3}}{2 (4\pi)^{2/3}}$, this implies
    \begin{align*}
        G^{0\text{-}\mathrm{avg.}}_{\Lambda_n}(0,0) - G_{\Z^3}(0,0) &\leq - \inf_{r \in \big[\lfloor \frac{n}{2}\rfloor, \frac{\sqrt{3} n}{2}\big]} \Big( \frac{C(3)}{r} + \frac{1}{6n^3} r^2\Big) + \frac{c_*}{n} = \bigg(c_* - \frac{(3 C(3))^{2/3}}{2}\bigg) \frac{1}{n} \\
        &= -\frac{1}{n} \bigg(\frac{3^{2/3}}{2 (4\pi)^{2/3}} - c_* \bigg)\eqqcolon - \frac{c_{**}}{n}
    \end{align*}
    with an absolute constant $c_{**} > 0$. The infimum is attained at $$r = (3 C(3))^{1/3} n = \Big( \frac{3}{4\pi}\Big)^{1/3} n \in \Big[ \frac{n}{2}, \frac{\sqrt{3} n}{2}\Big].$$ 
\end{proof}

Finally, we discuss the eigenvalues of the submatrices of $G^{0\text{-}\mathrm{avg.}}_{\Lambda_n}$. But first recall that its own eigenvalues are given by $0$ and $(1/\eta_k)_{k= 2, \ldots, n^d}$, where $0 = \eta_1 < \eta_2 \leq \eta_3 \leq \dots \leq \eta_{n^d}$ are the eigenvalues of $-\Delta_{\Lambda_n}$.
\begin{lemma}[Eigenvalues of zero-average Green's function]
    Let $d\geq 3$, $n \in \N$, and $U$ be a finite non-empty subset of $\Z^d$ viewed as a subset of $\Lambda_n$ for each $n$. Set $u_n \coloneqq \lvert U^c\rvert$ and let $(\mu_k)_{k\leq u_n}$ ordered non-increasingly denote the eigenvalues of 
    $$G_{\Lambda_n}^{0\text{-}\mathrm{avg.}}\vert_{U^c\times U^c} - G_{\Lambda_n}^{0\text{-}\mathrm{avg.}}\vert_{U^c\times U} G_{\Lambda_n}^{0\text{-}\mathrm{avg.}} \vert_{U^2}^{-1} G_{\Lambda_n}^{0\text{-}\mathrm{avg.}}\vert_{U\times U^c}.$$ 
    Then, for all $k \leq u_n - \lvert U\rvert$, 
        \begin{equation} \label{eq:auxiliary_estim}
            \begin{aligned}
                \mu_k \in \bigg[\frac{1}{\eta_{n^d-u_n + k + \lvert U\rvert +1}}, \frac{1}{\eta_{k+1}}\bigg] \quad \bigg( =\bigg[0, \frac{1}{\eta_{k+1}}\bigg] \;\text{for } k=u_n - \lvert U\rvert\bigg).
            \end{aligned}
        \end{equation} 
\end{lemma}
\begin{proof}
    Let $(\zeta_k)_{k\leq u_n}$ and $(\upsilon_k)_{k\leq u_n}$ ordered non-increasingly denote the eigenvalues of the matrices $G_{\Lambda_n}^{0\text{-}\mathrm{avg.}}\vert_{U^c\times U^c}$ and $G_{\Lambda_n}^{0\text{-}\mathrm{avg.}}\vert_{U^c\times U} G_{\Lambda_n}^{0\text{-}\mathrm{avg.}} \vert_{U^2}^{-1} G_{\Lambda_n}^{0\text{-}\mathrm{avg.}}\vert_{U\times U^c}$, respectively. Since $G_{\Lambda_n}^{0\text{-}\mathrm{avg.}}\vert_{U^c\times U^c}$ is a principal submatrix of $G_{\Lambda_n}^{0\text{-}\mathrm{avg.}}$, by \cite[Theorem 4.3.28]{horn2012matrix}, for all $1 \leq k \leq u_n - 1$
    \begin{align*}
        \frac{1}{\eta_{n^d-u_n + k + 1}} \leq \zeta_k \leq \frac{1}{\eta_{k+1}}, && 0 \leq \zeta_{u_n} \leq \frac{1}{\eta_{u_n+1}}.
    \end{align*}
    Furthermore, since $G_{\Lambda_n}^{0\text{-}\mathrm{avg.}}\vert_{U^c\times U} G_{\Lambda_n}^{0\text{-}\mathrm{avg.}} \vert_{U^2}^{-1} G_{\Lambda_n}^{0\text{-}\mathrm{avg.}}\vert_{U\times U^c}$ is positive semi-definite and has rank at most $\lvert U\rvert$, $\upsilon_k = 0$ for all $k > \lvert U\rvert$. 
    Thus, by Weyl's inequality \cite[Theorem 4.3.1]{horn2012matrix}, 
    $\mu_k \in [\zeta_{k+ \lvert U\rvert}, \zeta_k]$ for all $k \leq u_n - \lvert U\rvert$.
    \eqref{eq:auxiliary_estim} follows.   
\end{proof}


\section{The infinite volume limit of the Spherical model}
\label{sec:sphr}

The main objective of this section is to characterize the infinite-volume limit of the spherical model. To be exact, we prove the following result, which is a preciser version of Theorem \ref{thm:spherical}. 
\begin{theorem}[Infinite-volume limit of the spherical model]
\label{thm:spher_detailed}
    Let $\Lambda_n = \T^d_n$ and $\theta = (\theta_x)_{x \in \Lambda_n}$ be a configuration of the spherical model on $\Lambda_n$ at inverse temperature $\beta > 0$. Set $\beta_c \coloneqq G_{\Z^d}(0,0)$. For any finite set $U \subset \Z^d$ (considered as a subset of $\Lambda_n$ for each $n$), $\theta_{U} \coloneqq (\theta_x)_{x \in U}$ converges in law as $n\rightarrow \infty$ to:
    \begin{enumerate}
        \item $\beta < \beta_c$: an $m^2$-massive GFF on $\Z^d$ restricted to $U$ scaled by $1/\sqrt{\beta}$ with the mass $m^2$ depending on $\beta$ and $d$ in such a way that
        \begin{align*}
            G_{\Z^d, m^2}(x,x) = \beta \quad \text{for all }\, x\in \Z^d;
        \end{align*}
        \item $\beta = \beta_c$: a GFF on $\Z^d$ restricted to $U$ scaled by $1/\sqrt{\beta}$;
        \item $\beta > \beta_c$: a GFF on $\Z^d$ restricted to $U$ scaled by $1/\sqrt{\beta}$ plus an independent constant random drift of the form $\sqrt{\frac{\beta - \beta_c}{\beta}} X \mathbf{1}_U$ with $X$ being a Rademacher random variable, i.e., $\P[X=1] = \P[X=-1] = \frac{1}{2}$.
    \end{enumerate}
\end{theorem} 
\noindent The proof is divided into three parts, corresponding to the different temperature regimes, which are discussed separately in the following three subsections.

Besides convergence in law, we further establish convergence of local covariance functions of the spherical model. More precisely, in Section \ref{subsec:spher_correlations} we prove the following.
\begin{corollary}
\label{cor:spher_correlations}
    Let $\Lambda_n$, $\theta = (\theta_x)_{x \in \Lambda_n}$, and $U \subset \Z^d$ be as above, $(i_x)_{x \in U} \in \N_0^{U}$ be an arbitrary but fixed vector of non-negative integers.
    Then, the following holds:
    \begin{align*}
        \nu_{\Lambda_n, \beta} \bigg[\prod_{x \in U} \theta_x^{i_x}\bigg] \xrightarrow{n \rightarrow \infty} \E\bigg[ \prod_{x \in U} \alpha_x^{i_x}\bigg],
    \end{align*}
    where $\sqrt{\beta} \alpha \coloneqq \sqrt{\beta}(\alpha_x)_{x \in \Z^d}$ in correspondence with Theorem \ref{thm:spher_detailed} is either an $m^2$-massive GFF on $\Z^d$ if $\beta < \beta_c$, a GFF on $\Z^d$ if $\beta = \beta_c$, or a GFF on $\Z^d$ plus an independent drift $\sqrt{\beta - \beta_c} X \mathbf{1}_{\Z^d}$, where $X$ is a Rademacher random variable, if $\beta > \beta_c$.
\end{corollary}

\subsection{The high-temperature regime}
\label{subsec:sph_Thigh}

This section presents a proof of $\beta < \beta_c$ case of Theorem \ref{thm:spher_detailed}. 
We start by recalling that by Proposition \ref{prop:rel_GFF_spher},
\begin{align*}
    \mathrm{Law}\left(\theta \right) = \mathrm{Law}\bigg( \frac{\phi}{\sqrt{\beta}} ~\bigg\vert~ \norm{\phi} = \sqrt{\beta n^d}\bigg)
\end{align*}
for any $m_n^2$-massive GFF $\phi$ on $\Lambda_n$ with $m_n^2>0$. Thus, for any finite set $U \subset \Z^d$, we have the following equality of Lebesgue densities on $\R^{U}$:
\begin{align}
\label{eq:dens_spher}
    f_{\theta_{U}}(v) 
    &= f_{\frac{\phi_U}{\sqrt{\beta}} \big\vert \norm{\phi} =\sqrt{\beta n^d}} (v) 
    = f_{\frac{\phi_U}{\sqrt{\beta}}} \left( v \right) 
    \frac{f_{\norm{{\phi}^{U^c} + h^{\sqrt{\beta} v}}^2} (\beta n^d)}{f_{\norm{\phi}^2} (\beta n^d)} \eqqcolon f_{\frac{\phi_U}{\sqrt{\beta}}} \left( v \right) {J_n}(v),
\end{align}
where $v \coloneqq (v_x)_{x \in U} \in \R^{U}$. In the second equality, we applied the domain Markov property of $\phi$ (Proposition \ref{prop:domainMP}). Provided that $m_n^2 \to m^2 > 0$, the vector $\phi_U$ converges in distribution to an $m^2$-massive GFF on $\Z^d$ restricted to $U$. This in particular implies pointwise convergence of the corresponding Gaussian densities. To ensure that the ratio $J_n(v)$ in \eqref{eq:dens_spher} converges to $1$, we choose $m_n^2 > 0$ such that $\E[\norm{\phi}^2] = \beta n^d$, or equivalently, $G_{\Lambda_n, m_n^2}(0,0) = \beta$ for all $n$, which is possible since $\beta < \beta_c = G_{\Z^d}(0,0)$.

To conclude the proof in this regime, we show that with the above choice of $(m^2_n)_n$, the ratio $J_n(v)$ converges to one. To this end, let $P \in \R^{U^c\times U^c}, Q \in \R^{\Lambda_n\times \Lambda_n}$ be the orthonormal matrices diagonalizing $G_{U^c, m_n^2}\vert_{U^c \times U^c}$ and $G_{\Lambda_n, m_n^2}$, respectively. In particular,
\begin{align*}
    P^T (G_{U^c, m_n^2}\vert_{U^c \times U^c}) P &= \mathrm{diag}\left( (\tau_i)_{i=1}^{\lvert U^c\rvert} \right), \\
    Q^T G_{\Lambda_n, m_n^2} Q &= \mathrm{diag} \left( (\lambda_i \coloneqq (\eta_i + m_n^2)^{-1})_{i=1}^{n^d}\right)
\end{align*}
where $\tau_1 \geq \ldots \geq \tau_{\lvert U^c\rvert} > 0$ are the eigenvalues of $G_{U^c, m_n^2}\vert_{U^c \times U^c}$, $0 = \eta_1 < \eta_2 \leq \ldots \leq \eta_{n^d}$ are those of $-\Delta_{\Lambda_n}$ (cf. Lemma \ref{lemma:massive_EVs}). Set $u_n \coloneqq \lvert U^c\rvert$, $\tilde h \coloneqq P^T h^{\sqrt{\beta} v}_{U^c}$, and for a given triangle array $(Y_{i, u_n})_{i \leq u_n, u_n \in \N}$ of i.i.d. standard normal random variables, define $Z_{i, u_n} \coloneqq (\tilde h_i + \sqrt{\tau_i} Y_{i, u_n})^2 - \tau_i - \tilde h_i^2$ . In terms of these variables, we can then rewrite $J_n(v)$ as follows:
\begin{align*}
    J_n(v) &= \frac{f_{\ntwo{\tilde h + \mathcal{N}(0, \mathrm{diag}(\tau_j)_{j=1}^{u_n})}^2} (\beta n^d - \beta \ntwo{v}^2)}{f_{\ntwo{\mathcal{N}(0, \mathrm{diag}(\lambda_j)_{j=1}^{n^d})}^2} (\beta n^d)}\\
    &= \frac{f_{\sum_{i=1}^{u_n} Z_{i, u_n}} (\beta n^d - \sum_{i=1}^{u_n} \tau_i - \norm{\tilde h}^2 - \beta \ntwo{v}^2)}{f_{\sum_{i=1}^{n^d} \lambda_i (Y_{i,n^d}^2-1)} (0)}.
\end{align*}
Here, we used that $\beta n^d = \mathrm{Trace}(G_{\Lambda_n, m_n^2}) = \sum_{i=1}^{n^d} \lambda_i$. 

Note that by the canonical properties of the massive Green's function (see Subsection \ref{subsec:Greens_fcts}, particularly \eqref{eq:Gfct_decomp}), we have
\begin{align*}
      \sum_{i=1}^{u_n} \tau_i + \norm{\tilde h}^2
      &= \mathrm{Trace} \Big( G_{\Lambda_n, m_n^2}\vert_{U^c \times U^c} \Big) - \sum_{x \in U^c} \E_{\alpha_n}^x[ \ind_{\{H_{U} < \infty\}} G_{\Lambda_n, m_n^2}(X_{H_{U}},x)] + \big\lVert{h^{\sqrt{\beta} v}_{U^c}\big\rVert}^2\\
      &= u_n \beta - \sum_{y \in U} \sum_{x \in U^c} \P_{\alpha_n}^x[H_{U} < \infty, X_{H_{U}} = y] G_{\Lambda_n, m_n^2}(y,x) + \ntwo{h^{\sqrt{\beta} v}_{U^c}}^2\\
      &\eqqcolon n^d \beta - R_{U, n},
\end{align*}
where $\P_{\alpha_n}^x$ denotes the law of a simple random walk $X$ on $\Lambda_n \cup \{\dagger\}$ started at $x$ and killed uniformly with probability $\alpha_n = \frac{m^2_n}{2d + m^2_n}$ at every step; $H_U$ and $\tilde H_U$ are the first hitting time and the first return time to $U$ by this walk, respectively. Furthermore, since by the weak Markov property of the (massive) random walk $X$ (here $\omega_k$ stands for the shift of the time by $k$) and \eqref{eq:Gfct_expdec},
\begin{align*} 
    \P^{x}_{\alpha_n}[H_{U} < \infty] &= \sum_{k \geq 0} \sum_{z \in U} \P^{x}_{\alpha_n}[X_k = z, \tilde H_{U} \circ \omega_k = \infty] 
    = \sum_{z \in U} \P^{z}_{\alpha_n} [\tilde H_{U}  = \infty] \sum_{k \geq 0} \P^{x}_{\alpha_n}[X_k = z]\\
    &\leq \sum_{z \in U} \frac{2d}{1-\alpha_n} G_{\Lambda_n, m^2}(x,z)
    \leq (2d + m_n^2) |U| \sup_{z \in U} G_{\Lambda_n, m^2}(x,z)\\
    &\leq c(m^2, d, U) e^{-c'(m^2, d) \min_{z \in U}|x-z|}
\end{align*}
for appropriate constants $c, c' > 0$, we obtain
\begin{align*}
    \ntwo{h^{\sqrt{\beta} v}_{U^c}}^2 \leq \beta \ntwo{v}^2 \sum_{x \in U^c} \P^{x}_{\alpha_n}[H_{U} < \infty]^2 \leq c^2 \beta \ntwo{v}^2 \sum_{x \in U^c} e^{-2c'\min_{z \in U}|x-z|}.
\end{align*}
And therefore, $0 \leq |R_{U,n}| \leq \beta |U| + \sum_{y \in U} \sum_{x \in U^c} c e^{-c' |x-y|} + \ntwo{h^{\sqrt{\beta} v}_{U^c}}^2 \leq M$ for an absolute constant $M>0$.

Next, we study the variances of the two sums $\sum_{i=1}^{u_n} Z_{i, u_n}$ and $\sum_{i=1}^{n^d} \lambda_i (Y_{i,n^d}^2-1)$, denoted $s_n^2$ and $\sigma_n^2$, respectively. Using Lemma \ref{lemma:massive_EVs}, we deduce that
\begin{align*}
    \sigma_n^2 = 2 \sum_{k=1}^{n^d} \lambda_k^2 = \Theta(n^d);\qquad
    s_n^2 = 2\sum_{k=1}^{u_n} \tau_k^2 + \underbrace{4\sum_{j=1}^{u_n} \tau_j \tilde{h}_j^2}_{\leq 4\tau_1 M = \Theta(1)} = \sigma_n^2(1 - o(1)).
\end{align*}
This allows us to rewrite the ratio in question as
\begin{align*}
    J_n(u) = (1 + o(1)) \frac{s_n f_{\sum_{i=1}^{u_n} Z_{i, u_n}} \Big(s_n \frac{R_{U, n} - \beta \ntwo{v}^2}{s_n}\Big)}{\sigma_n f_{\sum_{i=1}^{n^d} \lambda_i (Y_{i,n^d}^2-1)} (0)}.
\end{align*}

To conclude the proof, we use the following result (stated in a more general form to also cover the critical case), proven in Appendix \ref{A:proof_densconv_1D}.
\begin{proposition}
\label{prop:dens_conv_1D}
    Let $(X_{i,n})_{i\leq n, n \in \N}$ be a triangular array of independent centered random variables with the probability density functions $f_{i,n}$ (with respect to Lebesgue measure on $\R$). Assume that 
    \begin{enumerate}
        \item $f_{i,n} \in L^r(\R)$ for some $r \in (1,2]$ (independent of $i,n$) such that $\sup_{i, n} \norm{f_{i,n}}_{L^r} \leq M$ for some $M> 0$;
        \item For all $i \leq n, n \in \N$, $\sigma_{i,n}^2 \coloneqq \mathrm{Var}[X_{i,n}] < \infty$ ordered such that $\sigma_{1,n}^2 \geq \ldots \geq \sigma_{n,n}^2$;
        \item Lindeberg's condition is satisfied, that is, for any $\varepsilon > 0$,
        \begin{align*}
            \frac{\sum_{i \geq 1} \E[X^2_{i,n} \ind_{\{ \lvert X_{i,n}\rvert > \varepsilon s_n\}}]}{s_n^2} \xrightarrow{n \rightarrow \infty} 0,
        \end{align*}
        where $s_n^2 \coloneqq \sum_{i=1}^n \sigma_{i,n}^2$.
        \item There exist $\delta > 0$ (uniform), $K(n) \geq 1$, $l_*(n) \geq 1, n \geq l^*(n) \geq 2 \lceil\frac{r}{r-1}\rceil$ such that for all $n$ sufficiently large
        \begin{align*}
            (a) &\; \frac{\sum_{i \geq l_*(n)} \sigma_{i,n}^2}{\sum_{i=1}^n \sigma_{i,n}^2} \geq \delta;\\
            (b) &\; \frac{\sum_{i \geq l_*(n)} \E[X^2_{i,n} \ind_{\{ \lvert X_{i,n}\rvert > K(n)\}}]}{\sum_{i \geq l_*(n)} \sigma_{i,n}^2} \leq \frac{1}{8}; \\
            (c) &\; \frac{n - l^*(n)}{\sigma_{l^*, n}^2 \vee K(n)^2} \gg \log \Big( \sum_{i=1}^n \sigma_{i,n}^2\Big).
        \end{align*}
    \end{enumerate}
    Then the relation 
    \begin{align*}
        s_n f^{(n)} (s_n x) \xrightarrow{n\rightarrow \infty} \frac{1}{\sqrt{2\pi}} \: e^{-\frac{x^2}{2}}
    \end{align*}
    holds uniformly over $\R$. Here, $f^{(n)}$ is the convolution of $f_{1,n}, \ldots, f_{n,n}$, as well as the density function of $\sum_{i=1}^n X_{i,n}$. Note that $s_n f^{(n)} (s_n x)$ is the density of $\frac{\sum_{i=1}^n X_{i,n}}{s_n}$.
\end{proposition}
\noindent Indeed, if the arrays $(Z_{i, u_n})_{i,u_n}$ and $(\lambda_i (Y_{i,n^d}^2-1))_{i,n^d}$ satisfy the assumptions of the proposition, both densities $s_n f_{\sum_{i=1}^{v_n} Z_{i, v_n}} (s_n x)$ and $\sigma_n f_{\sum_{i=1}^{n^d} \lambda_i (Y_{i,n^d}^2-1)} (\sigma_n x)$ converge uniformly to the standard normal density $f_{\mathcal{N}(0,1)}(x)$. In particular, this implies $J_n(v) \to 1$, as desired.

Let us now show that the array $(Z_{i, u_n} = (\tilde h_i + \sqrt{\tau_i} Y_{i, u_n})^2 - \tau_i - \tilde h_i^2)_{i,u_n}$ (up to reordering) indeed satisfies the assumptions of Proposition~\ref{prop:dens_conv_1D} with $l_* = 1$, $l^* = 2 \lceil \frac{r}{r-1}\rceil$, and $K$ sufficiently large but independent of $n$. 

It is clear that the variables $Z_{i, u_n}$ are centered and independent, with variances $\sigma^2_{i,u_n} \coloneqq \mathrm{Var}[Z_{i,u_n}] = 2 \tau_i^2 + 4 \tilde h_i^2 \tau_i \in (0, \infty)$.

Let $f_{i,u_n}$ denote the density of $Z_{i,u_n}$. It is given by
\begin{align*}
    f_{i,u_n}(t) &= \frac{\left( f_Y \Big( \frac{\sqrt{t + \tau_i + \tilde h_i^2} - \tilde h_i}{\sqrt{\tau_i}}\Big) + f_Y \Big( \frac{\sqrt{t + \tau_i + \tilde h_i^2} + \tilde h_i}{\sqrt{\tau_i}} \Big) \right)}{2\sqrt{\tau_i} \sqrt{t + \tau_i + \tilde h_i^2}} \ind_{\{t > - \tau_i - \tilde h_i^2\}},
\end{align*}
where $Y$ is a standard normal random variable. For any $r \in (1,2)$, we estimate (and thus verify condition 1.)
\begin{align*}
    \norm{f_{i,u_n}}_{L^r}^r &= \int_0^\infty \frac{\left( f_Y \Big( \frac{\sqrt{s} - \tilde h_i}{\sqrt{\tau_i}}\Big) + f_Y \Big( \frac{\sqrt{s} + \tilde h_i}{\sqrt{\tau_i}} \Big) \right)^r}{2^r \tau_i^{r/2} s^{r/2}} \d s 
    \leq \int_0^\infty \frac{f^r_Y \Big( \frac{\sqrt{s} - \lvert \tilde h_i\rvert}{\sqrt{\tau_i}} \Big)}{\tau_i^{r/2} s^{r/2}} \d s\\
    &= \int_0^\infty \frac{f_Y^r \Big( \sqrt{s} - \lvert\frac{\tilde h_i}{\sqrt{\tau_i}}\rvert\Big)}{\tau_i^{r-1} s^{r/2}} \d s 
    = 2\int_0^\infty \frac{f_Y^r \Big( t - \lvert\frac{\tilde h_i}{\sqrt{\tau_i}}\rvert\Big)}{\tau_i^{r-1} t^{r-1}} \d t\\ 
    &\leq \frac{2}{\tau_i^{r-1}} \Bigg(\int_{0}^1 t^{1-r} \d t + \int_1^{\infty} f_Y \bigg( t - \Big\lvert\frac{\tilde h_i}{\sqrt{\tau_i}} \Big\rvert \bigg) \d t\Bigg)\\ 
    &\leq \frac{2}{\tau_i^{r-1}} \left(\frac{1}{2-r} + 1 \right) \eqqcolon c(r) \frac{1}{\tau_i^{r-1}}.
\end{align*}
Hence, $\sup_{i, n} \norm{f_{i,u_n}}_{L^r}^r \leq c(r) (4d + 2 m^2)^{r-1}$, by uniform bounds on $(\tau_i)_i$—$\tau_i \in [L_*, L^*] \coloneqq [(2m^2+4d)^{-1}, 2/m^2]$—implied by \eqref{eq:unif_bdsEVs} and the convergence $m_n^2 \xrightarrow{n\rightarrow \infty} m^2>0$. 
Since $\tilde h_i^2 \leq \ntwo{h^{\sqrt{\beta} v}_{U^c}}^2 \leq c(m^2, d, U) \beta \ntwo{v}^2$, this further yields 
$$0 < 2 L_*^2 \leq \sigma^2_{i,u_n} \leq 2(L^*)^2 + 4c \beta\ntwo{v}^2 L^* < \infty$$ 
uniformly over $(i, n)$—condition 2. up to reordering. Consequently, $s_n^2 = \Theta(u_n)$. 

We now check condition $4.$ (and $3.$): $(a)$ and $(c)$ are immediate with our choice of $l_*, l^*$, and large $K$ independent of $n$. For all $K$ sufficiently large (e.g., $K > 4(L^* + c \beta\ntwo{v}^2)$), we have
\begin{align*}
    \big\{ \lvert Z_{i,u_n}\rvert > K \big\} &\subset \Big\{\big(\tilde h_i + \sqrt{\tau_i} Y_{i,u_n}\big)^2 > K\Big\} \\
    &\subset \Big\{|\sqrt{\tau_i} Y_{i,u_n}| > \sqrt{K} - |\tilde h_i| \Big\} 
    \subset \bigg\{ \lvert Y_{i,u_n} \rvert > \frac{\sqrt{K}}{2 \sqrt{L^*}} \bigg\}.
\end{align*}
Furthermore, $Z^2_{i,u_n} \leq \tilde c(d, m^2, U, L*, v) (Y_{i, u_n}^4 + \lvert Y_{i, u_n}\rvert^3 + Y_{i, u_n}^2 + \rvert Y_{i, u_n}\rvert + 1)$ for an appropriate $\tilde c > 0$, so
\begin{align*}
    \frac{\E[Z^2_{i,u_n} \ind_{\{ \lvert Z_{i,u_n}\rvert > K\}}]}{ \sigma_{i,u_n}^2}
    &\leq \frac{5\tilde c}{2L^2_*}  \E\bigg[Y_{i, u_n}^4 \ind_{\big\{ \lvert Y \rvert > \frac{\sqrt{K}}{2 \sqrt{L^*}} \big\}}\bigg] = \hat{c} \E\bigg[Y^4 \ind_{\big\{ \lvert Y \rvert > \frac{\sqrt{K}}{2 \sqrt{L^*}} \big\}}\bigg].
\end{align*}
This in turn yields,
\begin{align*}
    \frac{\sum_{i=1}^{u_n} \E[Z^2_{i,u_n} \ind_{\{ \lvert Z_{i,u_n}\rvert > K\}}]}{s_n^2} \leq \hat{c}\E\bigg[Y^4 \ind_{\big\{ \lvert Y \rvert > \frac{\sqrt{K}}{2 \sqrt{L^*}} \big\}}\bigg],
\end{align*}
which can be made arbitrarily small by choosing $K$ large. The Lindeberg condition follows with $\varepsilon s_n$ for $\varepsilon > 0$ in place of $K$.

\begin{remark}
In the high-temperature regime, one could alternatively deduce the uniform convergence of the numerator and denominator in $J_n$ to the standard normal density using results from Malliavin calculus \cite[Corollary 6.6, Theorem 7.3, Lemma 7.1]{HU2014814}. However, such results are not applicable in the supercritical regime and offer limited help for the critical case (only for $d \geq 4$, see Remark~\ref{rem:Malliavin}). For these reasons, we have chosen to follow a more direct route.
\end{remark}

\subsection{The low-temperature regime}
\label{subsec:sph_Tlow}

This section proves the $\beta > \beta_c$ case of Theorem \ref{thm:spher_detailed} and lays the groundwork for the critical regime. Throughout, we assume $d \geq 3$ and $\beta \geq \beta_c$, explicitly indicating when the analysis is restricted to $\beta > \beta_c$.
Recall that by Proposition \ref{prop:rel_GFF_spher}, 
\begin{align*}
    \mathrm{Law}\left(\theta \right) = \mathrm{Law}\bigg( \frac{\gamma + c \mathbf{1}_{\Lambda_n}}{\sqrt{\beta}} ~\bigg\vert~ \norm{\gamma + c \mathbf{1}_{\Lambda_n}} = \sqrt{\beta n^d}\bigg),
\end{align*}
where $\gamma$ is a zero-average GFF on $\Lambda_n$ and $c$ is an independent Lebesgue constant. Our strategy is to analyze the conditional joint law of $(\gamma_U, c)$ given $\norm{\gamma + c \mathbf{1}_{\Lambda_n}} = \sqrt{\beta n^d}$, for a finite set $U \subset \Z^d$ (viewed as a subset of $\Lambda_n$). More precisely, we show that it converges to the product measure of the GFF $\psi$ on $\Z^d$ restricted to $U$ and a Rademacher random variable $X$ scaled by $\sqrt{\beta - \beta_c}$. This immediately yields the desired local convergence in distribution of $\theta$ towards $\frac{1}{\sqrt{\beta}} \psi + \sqrt{\frac{\beta - \beta_c}{\beta}} X \mathbf{1}_{\Z^d}$.

Let $B \subset \R^{U}$ be a Borel set and $I \subset [0,\sqrt{\beta}]$ be a non-empty interval with the closure $\overline{I} = [a,b]$ and such that $a,b \notin \{\pm \sqrt{\beta - \beta_c}\}$. Observe that since $\gamma$ is a zero-average GFF, $\langle\gamma, \mathbf{1}_{\Lambda_n}\rangle = 0$ almost surely. Therefore,
\begin{align*}
    \P &\Big[ \gamma_U \in B, c \in I ~\Big\vert~ \norm{\gamma + c \mathbf{1}_{\Lambda_n}}  = \sqrt{\beta n^d}\Big] = \P \Big[ \gamma_U \in B, c \in I ~\Big\vert \norm{\gamma} ^2 + c^2 n^d = \beta n^d\Big]\\
    &= \int_B \frac{\int_I f_{\gamma_U, \ntwo{\gamma}^2}(y, n^d (\beta - c_0^2)) \d c_0 }{\int_{-\sqrt{\beta}}^{\sqrt{\beta}} f_{\ntwo{\gamma}^2} (n^d(\beta - \tilde c^2)) \d \tilde c} \d y 
    = \int_B f_{\gamma_U}(y) \frac{\int_I f_{\ntwo{\hat{\gamma}^y}^2}(n^d (\beta - c_0^2)) d c_0 }{2\int_{0}^{\sqrt{\beta}} f_{\ntwo{\gamma}^2} (n^d(\beta - \tilde c^2)) \d \tilde c} \d y  \\
    &= \int_B f_{\gamma_U}(y) \frac{\int_{a^2}^{b^2} \frac{1}{2 \sqrt{u}} n^d f_{\ntwo{\hat{\gamma}^y}^2 }(n^d (\beta - u)) \d u}{2 \int_{0}^{\beta} \frac{1}{2 \sqrt{\tilde u}} n^d f_{\ntwo{\gamma}^2} (n^d(\beta - \tilde u)) \d \tilde u} \d y 
    = \int_B f_{\gamma_U}(y) \frac{\int_{a^2}^{b^2} \frac{1}{\sqrt{u}} f_{\frac{\ntwo{\hat{\gamma}^y}^2}{n^d}}(\beta - u) \d u}{2 \int_{0}^{\beta} \frac{1}{\sqrt{\tilde u}} f_{\frac{ \ntwo{\gamma}^2}{n^d}} (\beta - \tilde u) \d \tilde u} \d y \\
    &= \int_B f_{\gamma_U}(y) \frac{\E \Big[ \frac{1}{\sqrt{\beta - \ntwo{\hat{\gamma}^y}^2/ n^d}} \ind_{(\beta - b^2, \beta - a^2)} \Big(\frac{\ntwo{\hat{\gamma}^y}^2}{n^d} \Big)\Big]}{2\E \Big[ \frac{1}{\sqrt{\beta - \ntwo{\gamma}^2/n^d}} \ind_{[0, \beta)} \Big( \frac{\ntwo{\gamma}^2}{n^d} \Big)\Big]}\d y \eqqcolon \int_B f_{\gamma_U}(y) R_n(y) \d y,
\end{align*}
where $\hat{\gamma}^y$ is defined by $\hat{\gamma}^y_U \equiv y$ and 
\begin{equation}
\label{eq:cond_GFFlaw}
    \begin{aligned}
        \hat{\gamma}^y_{U^c} &\sim \mathcal{N} \Big( G_{\Lambda_n}^{0\text{-}\mathrm{avg.}}\vert_{U^c\times U} G_{\Lambda_n}^{0\text{-}\mathrm{avg.}} \vert_{U^2}^{-1} y, 
        \; G_{\Lambda_n}^{0\text{-}\mathrm{avg.}}\vert_{U^c\times U^c} - G_{\Lambda_n}^{0\text{-}\mathrm{avg.}}\vert_{U^c\times U} G_{\Lambda_n}^{0\text{-}\mathrm{avg.}} \vert_{U^2}^{-1} G_{\Lambda_n}^{0\text{-}\mathrm{avg.}}\vert_{U\times U^c}\Big)\\
        &\eqqcolon \mathcal{N} (\nu(y), \Xi) = \mathrm{Law} (\gamma_{U^c} ~\vert~ \gamma_U = y).
    \end{aligned}
\end{equation}
Note that we may assume $B$ is bounded. Indeed, by the symmetries of the torus, 
\begin{align*}
    \E\big[\ntwo{\gamma_U}^2 ~\big\vert~ \ntwo{\gamma}^2 + c^2 n^d = \beta n^d \big] = \lvert U\rvert \:\E\big[\gamma_0^2 ~\big\vert~ \ntwo{\gamma}^2 + c^2 n^d = \beta n^d \big] \leq \lvert U\rvert \beta.
\end{align*}
Therefore, by choosing $t > 0$ large enough, the probability $\P \big[ \gamma_U \notin B^{\lvert U\rvert}(0,t) ~\big\vert~ \norm{\gamma + c \mathbf{1}_{\Lambda_n}} = \sqrt{\beta n^d} \big]$ can be made arbitrarily small.

Our goal is to prove the following result, which would conclude the proof of the low-temperature regime.
\begin{claim} \label{claim:ratio_conv}
   Uniformly over any fixed bounded $B \ni y$, 
    \begin{align}
    \label{eq:ratio_convergence}
        R_n(y) = \frac{\E \Big[ \frac{1}{\sqrt{\beta - \ntwo{\hat{\gamma}^y}^2/ n^d}} \ind_{(\beta - b^2, \beta - a^2)} \Big(\frac{\ntwo{\hat{\gamma}^y}^2}{n^d} \Big)\Big]}{2\E \Big[ \frac{1}{\sqrt{\beta - \ntwo{\gamma}^2/n^d}} \ind_{[0, \beta)} \Big( \frac{\ntwo{\gamma}^2}{n^d} \Big)\Big]} \xrightarrow{n \rightarrow \infty} \frac{1}{2} \ind_{I}(\sqrt{\beta - \beta_c}).
    \end{align} 
\end{claim}
\noindent Indeed, let $s \in \R \setminus\{\pm \sqrt{\beta - \beta_c}\}$ (the continuity points of the cumulative distribution function of $\sqrt{\beta - \beta_c} X$) and set $I_+ = (-\infty, s] \cap (0,\sqrt{\beta}]$, $I_- = (-\infty, s] \cap [-\sqrt{\beta}, 0]$. By symmetry and the fact that $\gamma_U$ converges in law towards $\psi_U$, the above implies that
\begin{align*}
    \P \Big[ &\gamma_U \in B, c \leq s ~\Big\vert~ \norm{\gamma + c \mathbf{1}_{\Lambda_n}}  = \sqrt{\beta n^d}\Big] 
    = \P \Big[ \gamma_U \in B, c \in I_+ ~\Big\vert~ \norm{\gamma + c \mathbf{1}_{\Lambda_n}}  = \sqrt{\beta n^d}\Big] \ind_{\{I_+ \neq \emptyset\}} \\
    &\hspace{2.2cm}+ \P \Big[ \gamma_U \in B, c \in -I_- ~\Big\vert~ \norm{\gamma + c \mathbf{1}_{\Lambda_n}}  = \sqrt{\beta n^d}\Big]\ind_{\{I_- \neq \emptyset\}}\\
    &\xrightarrow{n \rightarrow \infty} \frac{1}{2}\Big( \ind_{I_+}(\sqrt{\beta - \beta_c}) + \ind_{-I_-}(\sqrt{\beta - \beta_c})\Big) \P[\psi_U \in B]
    = \P\big[\sqrt{\beta - \beta_c} X \leq s\big] \P[\psi_U \in B],
\end{align*}
where $\ind_\emptyset \equiv 0$.

To prove Claim \ref{claim:ratio_conv} for $\beta >\beta_c$, we proceed in three steps (two of which also hold for $\beta = \beta_c$):
\begin{itemize}
    \item \textbf{Claim 1:} $\frac{\ntwo{\gamma}^2}{n^d}$ and $\frac{\ntwo{\hat{\gamma}^y}^2}{n^d}$ converge almost surely to $\beta_c$.
    \item \textbf{Claim 2}: If $\beta > \beta_c$, there exist $c = c(\beta, d), \hat{c} = \hat{c}(\beta, d) > 0$ such that for any $0 < \varepsilon < (\beta - \beta_c)/4$, 
    \begin{align*}
        \P\bigg[\frac{\ntwo{\gamma}^2}{n^d} \geq \beta - \varepsilon\bigg] \leq e^{-c n^{d-2}}, \qquad
        \P\bigg[\frac{\ntwo{\hat{\gamma}^y}^2}{n^d} \geq \beta - \varepsilon \bigg] \leq e^{-\hat{c} n^{d-2}}.
    \end{align*}
    \item \textbf{Claim 3:} The density functions of $\ntwo{\gamma}^2$ and $\ntwo{\hat{\gamma}^y}^2$ are uniformly bounded. The latter also holds uniformly in $y$ in a bounded Borel set $B$.
\end{itemize}
In view of \eqref{eq:ratio_convergence}, the convergence of $R_n(y)$ crucially depends on the behavior of the squared norms $\ntwo{\gamma}^2/n^d$ and $\ntwo{\hat{\gamma}^y}^2/n^d$. Claim 1 captures the expected concentration of both quantities around $\beta_c$, thereby providing a heuristic justification for the limiting behavior of the ratio. To make this precise, however, we must control the contributions from the upper end of the integration domain, specifically the interval $(\beta - \varepsilon, \beta)$, where the integrand may become singular due to the factor $1/\sqrt{\beta - x}$. Claims 2 and 3 are tailored to this task: Claim 2 shows that the probability mass in this region is exponentially suppressed, while Claim 3 ensures that the density functions remain uniformly bounded. A more direct route—namely, proving uniform boundedness of the density of $\ntwo{\hat{\gamma}^y}^2/n^d$ on this interval—would be desirable but is technically involved. Instead, we split $(\beta - \varepsilon, \beta)$ into two subintervals: one where the contribution is negligible due to exponential concentration away from $\beta$, and another, infinitesimally small neighborhood of $\beta$, where we can appeal to the integrability of $x \mapsto 1/\sqrt{\beta - x}$.

\begin{proof}[Proof of Claim \ref{claim:ratio_conv}]
Combined, Claims 1, 2, and 3 yield \eqref{eq:ratio_convergence} in the low-temperature regime. Indeed, for $\varepsilon > 0$ small enough (as in Claim 2), since $\R_+ \ni x \mapsto 1/\sqrt{x}$ away from zero is bounded, by dominated convergence theorem and Claim 1 we get
\begin{align*}
    \E \bigg[ \frac{1}{\sqrt{\beta - \ntwo{\hat{\gamma}^y}^2/ n^d}} \ind_{(\beta - b^2, (\beta - a^2) \wedge (\beta - \varepsilon))} \bigg(\frac{\ntwo{\hat{\gamma}^y}^2}{n^d} \bigg)\bigg] &\xrightarrow{n \rightarrow \infty} \frac{1}{\sqrt{\beta - \beta_c}} \ind_{(a,b)}(\sqrt{\beta - \beta_c}); \\
    \E \bigg[ \frac{1}{\sqrt{\beta - \ntwo{\gamma}^2/n^d}} \ind_{[0, \beta - \varepsilon)} \bigg( \frac{\ntwo{\gamma}^2}{n^d} \bigg)\bigg] &\xrightarrow{n \rightarrow \infty} \frac{1}{\sqrt{\beta - \beta_c}}.
\end{align*}
Let us consider the remaining bits and show that they are vanishing in the limit. Let $\delta > 0$ be small,
\begin{align*}
    0 &\leq \E \bigg[ \frac{1}{\sqrt{\beta - \ntwo{\gamma}^2/n^d}} \ind_{[\beta - \varepsilon, \beta)} \bigg( \frac{\ntwo{\gamma}^2}{n^d} \bigg)\bigg] \\
    &\leq \frac{n^d}{\sqrt{\delta}} \underbrace{\P \bigg[\frac{\ntwo{\gamma}^2}{n^d} \in [\beta - \varepsilon, \beta - \delta/n^{2d}) \bigg]}_{\overset{\text{Claim $2$}}{\leq} e^{-c n^{d-2}}} + \sqrt{n^d} \underbrace{\E \bigg[ \frac{1}{\sqrt{\beta n^d - \ntwo{\gamma}^2}} \ind_{\beta n^d + [- \delta/n^d, 0)} ( \ntwo{\gamma}^2)\bigg]}_{\overset{\text{Claim $3$}}{\leq} \norm{f_{\ntwo{\gamma}^2}}_{L^\infty} \int_0^{\delta/n^d} \frac{\d t}{\sqrt{t}} = 2\sqrt{\frac{\delta}{n^d}} \norm{f_{\ntwo{\gamma}^2}}_{L^\infty}}.
\end{align*}
By taking limit $n \rightarrow \infty$, we eliminate the first term, the second one can be made arbitrarily small by adjusting $\delta$. By replacing $\gamma$ with $\hat{\gamma}^y$, we recover the same result for the latter.
The claim concerning $R_n(y)$ follows immediately from the above two observations.
\end{proof}

We now prove the remaining claims.
\begin{proof}[Proof of Claim 1 and Claim 2]
We will show that there exists $c>0$ such that for any $t > 0$ and for all $n$ sufficiently large,
\begin{align}
\label{eq:concentration_of_measure}
    \P[S_n > t] \vee \P[S_n < -t] \leq e^{-c n^{d-2} (\sqrt{\beta_c + t/2} - \sqrt{\beta_c})^2},
\end{align}
where $S_n \in \Big\{\frac{\ntwo{\gamma}^2 - \E[\ntwo{\gamma}^2]}{n^d}, \frac{\ntwo{\hat{\gamma}^y}^2 - \E[\ntwo{\hat{\gamma}^y}^2]}{n^d}\Big\}$. Together with the observations that
\begin{align}\label{eq:observation}
    G^{0\text{-}\mathrm{avg.}}_{\Lambda_n}(0,0) \xrightarrow{n \rightarrow \infty} \beta_c && \text{and} &&\frac{1}{n^d} \E[\ntwo{\hat{\gamma}^y_{U^c}}^2] = \frac{1}{n^d}\mathrm{Trace}(C) + \frac{1}{n^d} \ntwo{\nu(y)}^2 \xrightarrow{n \rightarrow \infty} \beta_c + 0,
\end{align}
\eqref{eq:concentration_of_measure} immediately implies Claim 1 by Borel-Cantelli lemma (as the estimates are clearly summable in $n$) and Claim 2 by setting $t = \beta - \beta_c - \varepsilon \pm o(1)$. 
As for \eqref{eq:observation}, note that by Proposition \ref{prop:props:zeroAvGfct}, $G_{\Lambda_n}^{0\text{-}\mathrm{avg.}}(x,y)$ converges uniformly over $U^2$ to $G_{\Z^d}(x,y)$, and so do the corresponding eigenvalues. The latter matrix $G_{\Z^d}\vert_{U \times U}$ is positive-definite, thus, there exists $c>0$ (depending only on $U$) such that 
$$0<c\leq \lambda_{\min}\big(G_{\Lambda_n}^{0\text{-}\mathrm{avg.}}\vert_{U\times U}\big) \leq \lambda_{\max} \big(G_{\Lambda_n}^{0\text{-}\mathrm{avg.}}\vert_{U\times U}\big) \leq \mathrm{Trace} \big(G_{\Lambda_n}^{0\text{-}\mathrm{avg.}}\vert_{U\times U}\big) \leq 2\lvert U \rvert \beta_c.$$ 
Combined with \eqref{eq:0avgGfct_polydec}, this yields
\begin{equation} 
\label{eq:bound_shift}
    \begin{aligned}
        \frac{1}{n^d} \ntwo{\nu(y)}^2 &= \frac{1}{n^d} \big\lVert G_{\Lambda_n}^{0\text{-}\mathrm{avg.}}\vert_{U^c\times U} G_{\Lambda_n}^{0\text{-}\mathrm{avg.}} \vert_{U^2}^{-1} y\big\rVert^2 
        \leq \frac{\ntwo{y}^2}{c^2 n^d} \sum_{x \in U^c, z \in U} G_{\Lambda_n}^{0\text{-}\mathrm{avg.}}(x,z)^2\\
        &\leq \frac{\tilde C(U, d)}{n^d} \big( n \ind_{\{d=3\}} + \log n \ind_{\{d=4\}} + \ind_{\{d\geq 5\}} \big) \frac{\ntwo{y}^2}{c^2},
    \end{aligned}
\end{equation}
and fully analogously that $\frac{1}{n^d}\mathrm{Trace}(\Xi) = \beta_c - o(1)$ with the error uniform in $y \in B$ (bounded).

To prove \eqref{eq:concentration_of_measure}, we use the following concentration inequality for Lipschitz functions of a standard normal vector $\overline{Z}$, \cite[Theorem 3.8, (3.10)]{Massart}: if $f: \R^N \rightarrow \R$ is Lipschitz, $x>0$, then 
$$\P[f(\overline{Z}) \geq \E[f(\overline{Z})] + x] \leq 2 \frac{1}{\sqrt{2\pi}} \int_{x/\mathrm{Lip}(f)}^\infty e^{-u^2/2} \d u \leq e^{-\frac{1}{2} ({x}/{\mathrm{Lip}(f)})^2}.$$

Let us define $f: \R^{n^d} \rightarrow \R$, $x \mapsto \big\lVert \sqrt{G^{0\text{-}\mathrm{avg.}}_{\Lambda_n}} x\big\rVert$ and $\hat{f}^y: \R^{U^c} \rightarrow \R$, $x \mapsto \ntwo{\sqrt{\Xi} x + \nu(y)}$. Both these functions are Lipschitz with Lipschitz constants bounded by the largest eigenvalues of $\sqrt{G^{0\text{-}\mathrm{avg.}}_{\Lambda_n}}$ and $\sqrt{\Xi}$, respectively. These are $\eta_{2}^{-1/2} = \Theta(n)$ and $\sqrt{\mu_1} = \Theta(n)$, respectively (cf. \eqref{eq:EV_asympt} and \eqref{eq:auxiliary_estim}).
Let $S_n = \frac{\ntwo{\gamma}^2 - \E[\ntwo{\gamma}^2]}{n^d}$. We observe that for $\overline{Z} \sim \mathcal{N}(0, \mathrm{Id}_{n^d})$, an arbitrarily small $0 < \delta < t/4$, and all $n$ sufficiently large:
\begin{align*}
    \P[S_n \geq t] &\leq \P[\ntwo{\gamma} \geq \sqrt{n^d} \sqrt{\beta_c + t -\delta}] \\
    &\leq \P\big[f(\overline{Z}) \geq \E[f(\overline{Z})] + \sqrt{n^d} \sqrt{\beta_c + t -\delta} - \E[f^2(\overline{Z})]]^{1/2}\big] \\
    &\leq \exp \bigg( -\frac{1}{2} \Theta(n^{d-2}) \big(\sqrt{\beta_c + t -\delta} - \sqrt{\beta_c - o(1)}\big)^2\bigg) \leq e^{-cn^{d-2} (\sqrt{\beta_c + t/2} - \sqrt{\beta_c})^2}. 
\end{align*}
With $\hat{f}^y$ in place of $f$, we obtain an analogous bound for $S_n = \frac{\ntwo{\hat{\gamma}^y}^2 - \E[\ntwo{\hat{\gamma}^y}^2]}{n^d}$. To estimate $\P[S_n \leq -t]$, consider $-f$ and $-\hat{f}^y$.
\end{proof}

\begin{proof}[Proof of Claim 3]
Recall that $\Xi$ is positive semi-definite as a covariance matrix, with eigenvalues $\mu_1 \geq \mu  \geq\ldots \geq \mu_{\lvert U^c\rvert} \geq 0$. We set $u_n \coloneqq \lvert U^c\rvert$ and observe that $\mu_{u_n} = 0$ with the unit eigenvector $p^{u_n} = \frac{1}{\sqrt{u_n}} \mathbf{1}_{U^c}$. 
Let $P = (p^1, \ldots, p^{u_n})$ and $Q$ be the orthonormal matrices diagonalizing $\Xi$ and $G_{\Lambda_n}^{0\text{-}\mathrm{avg.}}$, respectively, i.e., $P^T \Xi P = \mathrm{diag} \big( (\mu_k)_{k=1}^{u_n}\big)$ and $Q^T G_{\Lambda_n}^{0\text{-}\mathrm{avg.}} Q = \mathrm{diag} \big( 0, (1/\eta_k)_{k=2}^{n^d}\big)$, where $0=\eta_1 < \eta  \leq \ldots\leq \eta_{n^d}$ are the eigenvalues of $-\Delta_{\Lambda_n}$. Let $(Y_i)_i$ be i.i.d. standard normal random variables. Then, if we set $h(y) \coloneqq P^T \nu(y)$,
\begin{equation}
\label{eq:EV_representation}
    \begin{aligned}
        \norm{\gamma} ^2 \overset{\text{law}}{=} \sum_{k=2}^{n^d} \frac{1}{\eta_k} Y_k^2;\qquad
        \norm{\hat{\gamma}^y_{U^c}} ^2 \overset{\text{law}}{=} \sum_{k=1}^{u_n-1} (h_k(y) + \sqrt{\mu_k} Y_k)^2 + h_{u_n}(y)^2.
    \end{aligned}
\end{equation}
By Young's convolution inequality since $\norm{f}_{L^1} = 1$ for a probability density function $f$, 
\begin{align*}
    \norm{f_{\norm{\hat{\gamma}^y}}}_{L^\infty} \leq \norm{f_{(h_l(y) + \sqrt{\mu_l} Y_l)^2 + (h_i(y) + \sqrt{\mu_i} Y_i)^2}}_{L^\infty}, && \norm{f_{\norm{\gamma} ^2}}_{L^\infty} \leq \norm{f_{Y_k^2/\eta_k + Y_j^2/\eta_j}}_{L^\infty}
\end{align*}
for any $k\neq j, l \neq i$ (the choice will be specified later). 

To complete the proof, it now suffices to show that there exists $K > 0$ uniform in $y$ (here, over all $\R^U$) such that $$\norm{f_{(h_l(y) + \sqrt{\mu_l} Y_l)^2 + (h_i(y) + \sqrt{\mu_i} Y_i)^2}}_{L^\infty} \leq K,$$ as the remaining case can be easily deduced from this one. 
To this end, fix $x \in \R_+$ and, for simplicity, write $(h_l(y) + \sqrt{\mu_l} Y_l)^2 + (h_i(y) + \sqrt{\mu_i} Y_i)^2 = (a + c Y)^2 + (b + d Z)^2$ (with $Y, Z$ independent standard normal random variables). The density of $(a/c + Y)^2$ at $x$ is given by $$\frac{1}{2 \sqrt{x}} \Big(f_Y(\sqrt{x} - \lvert a/c\rvert) + f_Y(\sqrt{x} + \lvert a/c\rvert)\Big) \leq \frac{1}{\sqrt{x}}.$$ Therefore,
\begin{align*}
    f_{(a + c Y)^2 + (b + d Z)^2}(x) = \int_0^x f_{c^2(a/c + Y)^2} (z) f_{d^2(b/d + Z)^2} (x-z) \d z \leq \frac{1}{c d} \int_0^1 \frac{1}{\sqrt{u (1-u)}} \d u = \frac{\pi}{c d}.
\end{align*}
If $c, d > 0$ are uniformly bounded away from zero, the claim follows. In our setting we can, for example, pick $(l,i) = (\lceil u_{n}/2\rceil, \lceil u_{n}/2\rceil + 1) = (k, j)$, then by \eqref{eq:EV_asympt} and \eqref{eq:auxiliary_estim}: $\mu_{l}, \mu_i, 1/ \eta_k, 1/\eta_j = \Theta(1)$.
\end{proof}

\subsection{At the criticality}
\label{subsec:spher_crit}

In this section, we analyze the model at the critical temperature $\beta = \beta_c$. Unlike in the other two regimes, the case $d=3$ requires a slightly different treatment compared to higher dimensions. Our proof for $d \geq 4$ combines arguments from both non-critical regimes and makes use of Proposition \ref{prop:dens_conv_1D} along with \eqref{eq:0avgGfct_polydec}. For $d=3$, the argument relies heavily on the fact that $\beta_c - G_{\Lambda_n}^{0\text{-}\mathrm{avg.}}(0,0) = \Theta(1/n) > 0$ (see \eqref{eq:negativity_Greens_diff}), and on the observation that the variance of $\ntwo{\gamma}^2$ is of the same order as the square of the largest eigenvalue of $G_{\Lambda_n}^{0\text{-}\mathrm{avg.}}(0,0)$.

We start by recalling that, except for the concluding argument and Claim 2, all steps from the low-temperature regime remain valid at $\beta = \beta_c$. Our goal is to prove \eqref{eq:ratio_convergence} in this setting, which amounts to showing that, uniformly in $y$ bounded,
\begin{align*}
    R_n(y) = \frac{\E \Big[ \frac{1}{\sqrt{\beta_c - \ntwo{\hat{\gamma}^y}^2/ n^d}} \ind_{(\beta_c - b^2, \beta_c - a^2)} \Big(\frac{\ntwo{\hat{\gamma}^y}^2}{n^d} \Big)\Big]}{2\E \Big[ \frac{1}{\sqrt{\beta_c - \ntwo{\gamma}^2/n^d}} \ind_{[0, \beta_c)} \Big( \frac{\ntwo{\gamma}^2}{n^d} \Big)\Big]},
\end{align*}
converges to $1/2$ if $a= 0$, and to zero otherwise. This would imply that, as desired, the conditional law of $(\gamma_U, c)$, given $\ntwo{\gamma + c \mathbf{1}_{\Lambda_n}}^2 = \beta_c n^d$, converges to $(\phi_U, 0)$, where $\phi$ is the GFF on $\Z^d$. 
Note that if $a>0$, by Fatou's lemma applied to the denominator, dominated convergence theorem to the numerator, Claim 1 in the previous section, and since $\P[\ntwo{\gamma}^2 < \beta_c n^d]$ is uniformly bounded away from zero (as we will see below in the proof of the case $a=0$), we obtain that $\limsup_n R_n(y) = 0$. This also yields that without loss of generality we can assume that $b^2 = \beta_c$. So, hereinafter let $a=0, b^2 = \beta_c$. Also recall that $u_n = |U^c|$.

We define the following variables:
\begin{align*}
    X_n &\coloneqq \frac{\ntwo{\gamma}^2 - \E[\ntwo{\gamma}^2]}{\sqrt{\mathrm{Var}[\ntwo{\gamma}^2]}},
    \qquad\hat{X}_n\coloneqq \frac{\ntwo{\hat{\gamma}^y}^2 - \E[\ntwo{\hat{\gamma}^y}^2]}{\sqrt{\mathrm{Var}[\ntwo{\hat{\gamma}^y}^2]}};\\
    T_n &\coloneqq \frac{\beta_c n^d - \E[\ntwo{\gamma}^2]}{\sqrt{\mathrm{Var}[\ntwo{\gamma}^2]}}, 
    \qquad\:\hat{T}_n\coloneqq \frac{\beta_c n^d - \E[\ntwo{\hat{\gamma}^y}^2]}{\sqrt{\mathrm{Var}[\ntwo{{\gamma}^y}^2]}}.
\end{align*}
Then, using the fact (proved below) that 
\begin{equation}
    \label{eq:var_error}
    \begin{aligned}
        \mathrm{Var}[\ntwo{\hat{\gamma}^y}^2] &= \mathrm{Var}[\ntwo{{\gamma}}^2] + \mathcal{O} \Big( n^2 \big(n \ind_{\{d=3\}} + \log n \ind_{\{d= 4\}} + \ind_{\{d \geq 5\}}\big) \Big)\\
        &= (1+o(1)) \mathrm{Var}[\ntwo{{\gamma}}^2] = \Theta \big(n^4 \ind_{\{d=3\}} + n^d \log n \ind_{\{d=4\}} + n^d \ind_{\{d \geq 5\}}\big),
    \end{aligned}
\end{equation}
we can rewrite $R_n(y)$ as 
\begin{align*}
    R_n(y) = \frac{1 + o(1)}{2} \frac{\E \Big[ \frac{1}{\sqrt{\hat{T}_n - \hat{X}_n}} \ind_{\{\hat{X}_n < \hat{T}_n\}}\Big]}{\E \Big[ \frac{1}{\sqrt{T_n - X_n}} \ind_{\{X_n < T_n\}}\Big]}.
\end{align*}
Since $\E[\ntwo{\hat{\gamma}^y}^2] = G^{0\text{-}\mathrm{avg.}}_{\Lambda_n}(0,0) u_n + o(n^{d/2})$ (follows analogously to \eqref{eq:bound_shift}), and by combining \eqref{eq:0avgGfct_polydec} at $y = 0$ with \eqref{eq:negativity_Greens_diff}, we see that both $T_n$ and $\hat{T}_n$ converge to the same finite limit $T$ (depending on $d$): for $d \geq 4$, $T = 0$; for $d = 3$, $T > 0$. As for $\hat{X}_n$ and $X_n$, we have the following.
\begin{itemize}
    \item \textbf{Claim A:} $X_n$ and $\hat{X}_n$ converge in law to the same non-trivial (by symmetry, equivalent to finite non-zero) real random variable $X$. For $d \geq 4$, $X$ is a standard normal random variable.
    \item \textbf{Claim B:} The densities of $\hat{X}_n$ and $X_n$ are uniformly bounded.
\end{itemize}

Using these two observations and applying Portmanteau theorem to the upper/lower semicontinuous functions $x \mapsto \frac{1}{\sqrt{x}} \ind_{\{x \geq \varepsilon\}}$ and $x \mapsto \frac{1}{\sqrt{x}} \ind_{\{x > 0\}}$, respectively, for $\varepsilon > 0$ arbitrary small such that $\varepsilon < T(d=3)/2$, we obtain that 
\begin{align*}
    \limsup_{n \rightarrow \infty} \E \bigg[ \frac{1}{\sqrt{\hat{T}_n - \hat{X}_n}} \ind_{\{\hat{X}_n < \hat{T}_n\}}\bigg] &\leq 2M \sqrt{\varepsilon} + \E\bigg[ \frac{1}{\sqrt{T - X}} \ind_{\{X < T - \varepsilon\}}\bigg];\\
    \liminf_{n \rightarrow \infty} \E \bigg[ \frac{1}{\sqrt{T_n - X_n}} \ind_{\{X_n < T_n\}}\bigg] &\geq \E\bigg[ \frac{1}{\sqrt{T - X}} \ind_{\{X < T\}} \bigg] \geq \E\bigg[ \frac{1}{\sqrt{T - X}} \ind_{\{X < T - \varepsilon\}}\bigg]. 
\end{align*}
Here $M>0$ is the uniform bound on the densities of $\hat{X}_n$ and $X_n$.
Since $X$ is non-trivial and of mean zero, the probability of it being non-positive is strictly positive. Hence, if $d=3$, since $T - \varepsilon > 0$ in this case, there exists an absolute (independent of $\varepsilon < T/2$) constant $p>0$ such that $\E\Big[ \frac{1}{\sqrt{T - X}} \ind_{\{X < T - \varepsilon\}}\Big] > p$. Existence of such a constant for $d\geq 4$ even though $T(d\geq 4) = 0$ is clear since $X$ is standard Gaussian. Thus, $R_n(y) \leq \frac{1+o(1)}{2} \Big( 1 + \frac{2M\sqrt{\varepsilon}}{p}\Big)$.
By taking $\varepsilon>0$ much smaller than $p$, we conclude that $R_n$ is less or equal to $1/2$ in the limit. By reversing the roles of the numerator and denominator, we derive that $R_n(y)$ is bounded from below by $1/2$. The proof of the critical-temperature regime is thus complete, provided \eqref{eq:var_error} and Claims A and B are verified.

We now proceed to discuss the remaining proofs of \eqref{eq:var_error} and Claims A and B. Our arguments for the latter will significantly differ for $d \geq 4$ and $d=3$. 
\begin{proof}[Proof of \eqref{eq:var_error}]
   Using the eigenvalue representation \eqref{eq:EV_representation}, we easily observe that $\mathrm{Var}[\ntwo{\hat{\gamma}^y}^2] = 2 \sum_{k=1}^{u_n-1} (\mu_k^2 + 2 h_k(y)^2 \mu_k)$. By \eqref{eq:EV_asympt}, \eqref{eq:auxiliary_estim}, and  \eqref{eq:bound_shift}, $$\mu_1 \ntwo{h(y)}^2 \ll \Theta \big(n^4 \ind_{\{d=3\}} + n^d \log n \ind_{\{d=4\}} + n^d \ind_{\{d \geq 5\}}\big) = 2 \sum_{k=1}^{u_n-1} \mu_k^2$$ uniformly over $y$ in a fixed bounded set. Combined with \eqref{eq:0avgGfct_polydec}, this further implies that
    \begin{align*} 
        \frac{1}{2} \mathrm{Var}[\ntwo{\hat{\gamma}^y}^2] &= \sum_{k=1}^{u_n-1} \mu_k^2 + \mathcal{O} \big( n^2 \big(n \ind_{\{d=3\}} + \log n \ind_{\{d= 4\}} + \ind_{\{d \geq 5\}}\big) \big);\\  
        \sum_{k=1}^{u_n-1} \mu_k^2 &= \mathrm{Trace}(\Xi^2) = \frac{1}{2} \mathrm{Var}[\ntwo{{\gamma}}^2] + A_1 + A_2 + A_3, 
    \end{align*}
    where $\hat{\Lambda}_n = [-n/2, n/2)^d \cap \Z^d$ and 
    \begin{align*}
        A_1 &\coloneqq \sum_{x, y \in U^c}  G_{\Lambda_n}^{0\text{-}\mathrm{avg.}}(x,y)^2 - \frac{1}{2} \mathrm{Var}[\ntwo{{\gamma}}^2] = - 2 \sum_{\substack{x \in U^c, \\ z \in U}} G_{\Lambda_n}^{0\text{-}\mathrm{avg.}}(x,z)^2 
        \\&\phantom{:}= \mathcal{O} \big( \big(n \ind_{\{d=3\}} + \log n \ind_{\{d= 4\}} + \ind_{\{d \geq 5\}}\big) \big); \\
        A_2 &\coloneqq \sum_{\substack{z, z',\\ w, w' \in U\\ x,y \in U^c}} G_{\Lambda_n}^{0\text{-}\mathrm{avg.}}(x,z) G_{\Lambda_n}^{0\text{-}\mathrm{avg.}} \vert_{U^2}^{-1}(z, z') G_{\Lambda_n}^{0\text{-}\mathrm{avg.}}(z',y) G_{\Lambda_n}^{0\text{-}\mathrm{avg.}}(y,w) G_{\Lambda_n}^{0\text{-}\mathrm{avg.}} \vert_{U^2}^{-1}(w, w') G_{\Lambda_n}^{0\text{-}\mathrm{avg.}}(w',x)\\
        &\phantom{:}=\mathcal{O}\Big( \big( \int_{1}^n r^{4-2d} r^{d-1} \d r\big)^2\Big) = \mathcal{O} \big( n^2 \ind_{\{d=3\}} + \log^2 n \ind_{\{d= 4\}} + \ind_{\{d \geq 5\}}\big);\\ 
        A_3 &\coloneqq - 2\sum_{\substack{z, z' \in U\\ x,y \in U^c}} G_{\Lambda_n}^{0\text{-}\mathrm{avg.}} \vert_{U^2}^{-1}(z, z') G_{\Lambda_n}^{0\text{-}\mathrm{avg.}}(x,y) G_{\Lambda_n}^{0\text{-}\mathrm{avg.}}(y,z) G_{\Lambda_n}^{0\text{-}\mathrm{avg.}}(z',x) \\
        &\phantom{:}= \mathcal{O} \Big( \sum_{\substack{x\neq y \in \hat{\Lambda}_n \setminus\{0\}}} \lvert x\rvert^{2-d} \lvert y \rvert^{2-d} \d(x,y)^{2-d} \Big)\\ 
        &\phantom{:}= \mathcal{O} \bigg( 2 \sum_{x \in \hat{\Lambda}_n \setminus \{0\}} \lvert x\rvert^{4-2d} 
        \sum_{\substack{y \in \hat{\Lambda}_n \setminus \{0,x\}\\ \lvert y\rvert \geq \lvert x\rvert}} \d(x,y)^{2-d} \bigg) \\
        &\phantom{:}= \mathcal{O} \Big( n^2 \sum_{x \in \hat{\Lambda}_n \setminus \{0\}} \lvert x\rvert^{4-2d} \Big) = \mathcal{O} \Big( n^2 \big(n \ind_{\{d=3\}} + \log n \ind_{\{d= 4\}} + \ind_{\{d \geq 5\}}\big) \Big). 
    \end{align*}
    In conclusion, as desired
    \begin{align*}
        \mathrm{Var}[\ntwo{\hat{\gamma}^y}^2] &= \mathrm{Var}[\ntwo{{\gamma}}^2] + \mathcal{O} \Big( n^2 \big(n \ind_{\{d=3\}} + \log n \ind_{\{d= 4\}} + \ind_{\{d \geq 5\}}\big) \Big)\\
        &= (1+o(1)) \mathrm{Var}[\ntwo{{\gamma}}^2].
    \end{align*}
\end{proof}

\subsubsection*{Proof of Claims A and B for $d \geq 4$}
The key observation in this case is the following.

\textbf{Claim C:}
    The densities of $X_n$ and $\hat{X}_n$\footnote{up to removing the last $\lvert U\rvert$ terms in the eigenvalue decomposition \eqref{eq:EV_representation} of $\hat{X}_n$ containing $\mu_{u_n - q}$ for $1 \leq q \leq \lvert U\rvert$} converge uniformly towards the density function of a standard normal random variable.
  
The proof of this result is based on Proposition \ref{prop:dens_conv_1D} (note that for $d=3$ even Lindeberg's condition fails) and the eigenvalue representation \eqref{eq:EV_representation}. Before proving Claim C, we explain how it yields Claims A and B. Note that by the eigenvalue decomposition we can write $\hat{X}_n$ as a sum of two independent random variables $P_n + L_n$, where $P_n$ is the ``good'' part treated in Claim C and $L_n$ is the sum of the remaining $\lvert U\rvert$ terms. Then, by Claim C, $P_n$ converges in law to a standard normal random variable and its density is uniformly bounded. Note that $L_n = \frac{\sum_{u_n - \lvert U\rvert}^{u_n} \mu_k (Y_k^2 - 1) + 2h_k(y) \sqrt{\mu_k} Y_k}{\sqrt{\mathrm{Var}[\ntwo{\hat{\gamma}^y}^2]}}$. Since these $\mu_k$'s are uniformly bounded and $\ntwo{h(y)}$ by \eqref{eq:bound_shift} is growing slower than any polynomial, it is straightforward to verify that $L_n$ admits a density function and converges almost surely to zero. This implies that $L_n + P_n$ converges in law to a standard Gaussian variable (Claim A) and that the density of $\hat{X}_n$ is uniformly bounded as a convolution of a bounded function and an integrable non-negative (with unit $L^1$-norm) function (Claim B). 

\begin{proof}[Proof of Claim C]
We check that the required assumptions of Proposition \ref{prop:dens_conv_1D} are satisfied by $\mathbf{Y}_n \coloneqq ((h_k(y) + \sqrt{\mu_k} Y_k)^2 - \mu_k - h_k(y)^2)_{k=1}^{u_n - \lvert U\rvert - 1}$ contributing to $\hat X_n$, and thus, also by $(\frac{1}{\eta_k}(Y_k^2 - 1))_{k=2}^{n^d}$ contributing to $X_n$, since the latter can be recovered from the former by setting $h(y) \equiv 0$ and $U = \emptyset$. More precisely, we consider independent families $(\mathbf{Y}_n)_n$ that form a triangular array. For better readability, we do not use double-indexing for the elements of each $\mathbf{Y}_n$.

The first assumption of the proposition can be directly deduced from the corresponding verification in the high-temperature regime since the eigenvalues $(\mu_k)_{k=1}^{u_n-\lvert U\rvert - 1}$ of $\Xi$ and $(1/\eta_k)_{k=2}^{n^d}$ of $G_{\Lambda_n}^{0\text{-}\mathrm{avg.}}$ are uniformly bounded away from zero. The second condition (up to reordering which will be done for the verification of the fourth criterion) is clear since $\sigma^2_{k} = 2 \mu_k^2 + 4 h_k(y)^2 \mu_k < \infty$. 

From this point on, the index $k$ refers to the original ordering, while $\tilde k$ denotes the ordering by decreasing variance.
Recall that Lindeberg's condition (the third condition of the proposition) follows from Lyapunov's condition, which in our setting for $\delta = 2$ appears as
\begin{align*}
    \frac{1}{s_n^4} \sum_{k=1}^{u_n-\lvert U\rvert - 1} \E\big[((h_k(y) + \sqrt{\mu_k} Y_k)^2 - \mu_k - h_k(y)^2)^4 \big] \xrightarrow{n \rightarrow \infty} 0,
\end{align*}
where $s_n^2 = \sum_{k=1}^{u_n-\lvert U\rvert - 1} (2 \mu_k^2 + 4 h_k(y)^2 \mu_k) = 2(1+o(1)) \sum_{k=1}^{u_n-\lvert U\rvert - 1} \mu_k^2$ by \eqref{eq:bound_shift}. We observe that this convergence indeed takes place since 
\begin{align*}
    \sum_{k=1}^{u_n - \lvert U\rvert - 1}\E\big[((h_k(y) + \sqrt{\mu_k} Y_k)^2 - \mu_k - h_k(y)^2)^4 \big] &\leq c \sum_{k=1}^{u_n - \lvert U\rvert - 1} (\mu_k^4 + \mu_k^2 h_k(y)^4 + \mu_k^3 h_k(y)^2) \\
    &\leq c \bigg(\sum_{k=1}^{u_n - \lvert U\rvert - 1} \mu_k^4 + \mu_1^2 \ntwo{h(y)}^4 + \mu_1^3 \ntwo{h}^2\bigg),
\end{align*}
and by \eqref{eq:EV_asympt}, \eqref{eq:auxiliary_estim}, and \eqref{eq:bound_shift}, for $d \geq 4$, this is of order $\mathcal{O} \big(n^8\ind_{\{d \leq 7\}} + n^d \log n \ind_{\{d = 8\}} + n^d \ind_{\{d \geq 9\}}\big)$, while $s_n^4 = \Theta\big(n^8 (\log n)^2 \ind_{\{d = 4\}} + n^{2d} \ind_{\{d \geq 5\}}\big)$.

For the fourth criterion, consider the set $J = \{\sigma_k^2: k \leq u_n - \lvert U\rvert -1 \text{ such that } \lvert h_k(y)\rvert > 1/\log n\}$—there are at most $\log n \ntwo{h(y)}^2$ such indices $k$. 
Notice that since $(\mu_k)_{k \leq u_n - \lvert U\rvert - 1}$ are uniformly bounded away from zero and $\sigma_{k}^2 = 2 \mu_k^2 + 4\mu_k h_k(y)^2$, the order of the variances on the complement of $J$ will be determined by the order of $\mu_k$'s. Furthermore, elements of $J$ after reordering can only be sent forward or remain at their place. Hence, for any $L \geq 1$,
\begin{align*}
    \{\sigma_k^2: k \geq L\} \setminus J \quad\subset\; \{\sigma_{\tilde k}^2: \tilde k \geq L\} \setminus J \;\subset\quad \{\sigma_k^2: k \geq L - 2\lvert J\rvert\} \setminus J.
\end{align*}
Let $\varepsilon \in (0,1)$ be fixed but arbitrarily small. We choose the constants as follows:
\begin{itemize}
    \item for $d \geq 5$, let $l^* = l_* \coloneqq \inf\{l \geq \varepsilon n^d: l - 2\lvert J\rvert \notin J, l \in \Z\}$, $K = K(\varepsilon)$ a large constant depending only on $\varepsilon$;
    \item for $d=4$, let $l^*\coloneqq \inf\{l \geq \varepsilon n^d: l - 2\lvert J\rvert \notin J, l \in \Z\}$, $l_*\coloneqq \inf\{l \geq n^{d-\varepsilon}: l - 2\lvert J\rvert \notin J, l \in \Z\}$, $K(n)= n^\varepsilon$.  
\end{itemize}
Note that $l^*, l_*$ can at most be equal to the reference value (either $\varepsilon n^d$ or $n^{d-\varepsilon}$) plus $\lvert J\rvert = o(n^\delta)$ for any $\delta > 0$. Moreover, for $L \in \{l_*, l^*\}$,
\begin{align*}
    \{\sigma_k^2: k \geq L\} \setminus J \quad\subset\; \{\sigma_{\tilde k}^2: \tilde k \geq L\} \;\subset\quad \{\sigma_k^2: k \geq L - 2\lvert J\rvert\}.
\end{align*}
This implies that the $l^*$-th element of the reordered sequence $(\sigma^2_{\tilde k})_{\tilde k}$ is smaller or equal to $\sigma^2_{l^* - 2\lvert J\rvert}$ (of the original order). 

We first verify the assumptions of the fourth criterion for $d\geq 5$. By \eqref{eq:EV_asympt} and \eqref{eq:auxiliary_estim}, for $u_n-\lvert U\rvert - 1 \geq k \geq l_* - 2\lvert J\rvert, k \notin J$, $0 < c_* \leq \mu_k = \Theta(n^2/k^{2/d}) = \Theta(1) \leq \mu_{l_* - 2\lvert J\rvert} \leq c^* < \infty$, $\sigma_{k}^2 = 2\mu_k^2 + 4 \mu_k h_k(y)^2 = \Theta(1)$ uniformly. Thus, for $Z_{k} = (h_k(y) + \sqrt{\mu_k} Y_k)^2 - \mu_k - h_k(y)^2$, $Z \sim \mathcal{N}(0,1)$,
\begin{align*}
    \frac{\sum_{\tilde k \geq l_*}^{u_n-\lvert U\rvert - 1} \E[Z_{\tilde k}^2 \ind_{\{\lvert Z_{\tilde k}\rvert > K\}}]}{\sum_{\tilde k \geq l_*}^{u_n-\lvert U\rvert - 1} \sigma_{\tilde k}^2} &\leq \frac{\sum_{k \geq l_* - 2\lvert J\rvert, k \notin J}^{u_n-\lvert U\rvert - 1} \E[Z_k^2 \ind_{\{\lvert Z_k\rvert > K\}}] + \lvert J\rvert \sigma^2_{l_* - 2\lvert J\rvert}}{\sum_{k \geq l_*, k \notin J}^{u_n-\lvert U\rvert - 1} \sigma^2_{k}}\\
    &\leq \frac{\sum_{k \geq l_*, k \notin J}^{u_n-\lvert U\rvert - 1} \E[Z_k^2 \ind_{\{\lvert Z_k\rvert > K\}}] + 3\lvert J\rvert \sigma^2_{l_* - 2\lvert J\rvert}}{\sum_{k \geq l_*, k \notin J}^{u_n-\lvert U\rvert - 1} \sigma^2_{k}}\\
    &\leq \frac{\E\Big[(Z^2 - 1)^2 \ind_{\big\{\mu_{l_* - 2\lvert J\rvert} Z^2 + \frac{2 \sqrt{\mu_{l_* - 2\lvert J\rvert}}}{\log n} \lvert Z\rvert > K \big\}}\Big] \sum_{k \geq l_*, k \notin J}^{u_n-\lvert U\rvert - 1} \mu_k^2}{2 \sum_{k \geq l_*, k \notin J}^{u_n-\lvert U\rvert - 1} \mu_k^2} \\
    &\qquad\qquad +\frac{4\mu_1 \ntwo{h(y)}^2 + \frac{1}{\log n} 4 \sqrt{3} \sum_{k \geq l_*, k \notin J}^{u_n-\lvert U\rvert - 1} \mu_k^{3/2} + o(n)}{2 \sum_{k \geq l_*, k \notin J}^{u_n-\lvert U\rvert - 1} \mu_k^2}\\
    &\leq \frac{1}{2} \E\big[(Z^2 - 1)^2 \ind_{\{ Z^2 + \lvert Z\rvert > K/(2 c^*)\}}\big] + o(n^{3-d}) + \Theta \Big( \frac{1}{\log n}\Big)
\end{align*}
can be made arbitrarily small by choosing $K > c^*$ sufficiently large. Furthermore,  
\begin{align*}
    \frac{\sum_{\tilde k \geq l_*}^{u_n-\lvert U\rvert - 1} \sigma_{\tilde k}^2}{s_n^2} \geq \frac{\sum_{k \geq l_*, k \notin J}^{u_n-\lvert U\rvert - 1} \sigma_{k}^2}{s_n^2} &\geq \frac{\sum_{k = \varepsilon n^d + n}^{u_n-\lvert U\rvert - 1} 2 \mu_k^2}{\Theta (n^d)} \\
    &\geq \frac{\Theta \big(n^4 \sum_{k = \varepsilon n^d + n^{d/2}}^{u_n-\lvert U\rvert - 1} k^{-\frac{4}{d}} \big)}{\Theta (n^d)} = (1 - \varepsilon^{\frac{d-4}{d}}) \Theta(1); \\
    \frac{u_n-\lvert U\rvert - 1 - l^*}{\sigma^2_{l^*, \text{reorder.}} \vee K^2} &\geq \frac{u_n-\lvert U\rvert - 1 - l^*}{\sigma^2_{l^* - 2\lvert J\rvert} \vee K^2} = \Theta(n^d) \gg \log(n^d).
\end{align*}
This completes the verification of all required conditions for $d \geq 5$. 

For $d= 4$, the argument proceeds analogously. Using \eqref{eq:EV_asympt} and \eqref{eq:auxiliary_estim}, we obtain:
\begin{align*}
    \frac{\sum_{\tilde k \geq l_*}^{u_n-\lvert U\rvert - 1} \sigma_{\tilde k}^2}{s_n^2} \geq \frac{\sum_{k \geq l_*, k \notin J}^{u_n-\lvert U\rvert - 1} \sigma_{k}^2}{s_n^2} &\geq \frac{\Theta\Big( \sum_{k = 2n^{d-\varepsilon}}^{u_n-\lvert U\rvert - 1} n^4/k \Big)}{\Theta (n^4 \log n)} = \frac{\Theta \big(\log \big( \frac{u_n}{n^{d-\varepsilon}}\big) \big)}{\Theta (\log n)} = \varepsilon \Theta(1);\\
    \frac{u_n-\lvert U\rvert - 1 - l^*}{\sigma^2_{l^*, \text{reorder.}} \vee K(n)^2} &\geq \frac{u_n-\lvert U\rvert - 1 - l^*}{\sigma^2_{l^* - 2\lvert J\rvert} \vee K(n)^2} = \Theta \Big( \frac{n^d}{(n^4/(\varepsilon n^d)^{4/d}) \vee n^{2\varepsilon}}\Big)\\ 
    &= \Theta(n^{d-2 \varepsilon}) \gg \log(n^4 \log n); \\
    \frac{\sum_{\tilde k \geq l_*}^{u_n-\lvert U\rvert - 1} \E[Z_{\tilde k}^2 \ind_{\{\lvert Z_{\tilde k}\rvert > K(n)\}}]}{\sum_{\tilde k \geq l_*}^{u_n-\lvert U\rvert - 1} \sigma_{\tilde k}^2} &\leq \frac{\sum_{k \geq l_*, k \notin J}^{u_n-\lvert U\rvert - 1} \E[Z_k^2 \ind_{\{\lvert Z_k\rvert > K(n)\}}] + 3\lvert J\rvert \sigma^2_{l_*-2\lvert J\rvert}}{\sum_{k \geq l_*, k \notin J}^{u_n-\lvert U\rvert - 1} \sigma_{k}^2} \\
    &\leq \frac{1}{2} \E\Big[(Z^2 - 1)^2 \ind_{\big\{ Z^2 + \lvert Z\rvert > \frac{K(n)}{2 \mu_{l_* - 2\lvert J\rvert}}\big\}}\Big] \\&\qquad\qquad+\frac{\Theta \Big( n^2 \ntwo{h(y)}^2 + \frac{\sum_{k \geq l_*}^{u_n} \mu_k^{3/2}}{\log n} + \lvert J\rvert \sigma^2_{l_* - 2\lvert J\rvert}\Big)}{\varepsilon \Theta (n^4 \log n)}\\
    &= \frac{1}{2} \E\Big[(Z^2 - 1)^2 \ind_{\big\{ Z^2 + \lvert Z\rvert > \frac{n^\varepsilon}{\Theta(n^{\varepsilon/2})}\big\}}\Big]  + o(n^{- 1}) + \Theta \Big( \frac{1}{\log^2 n}\Big).
\end{align*}
The latter is arbitrarily small.
\end{proof}

\subsubsection*{Proof of Claims A and B for $d=3$} 
\begin{proof}[Proof of Claim B]
   The statement can be verified fully analogously to Claim 3 in the low-temperature regime (see Section \ref{subsec:sph_Tlow}) using eigenvalue representation \eqref{eq:EV_representation} once we notice that for $\mu \in \{\mu_1, \mu_2, \eta_2^{-1}, \eta_3^{-1}\}$ and $v \in \{\mathrm{Var}[\ntwo{\hat{\gamma}^y}^2], \mathrm{Var}[\ntwo{{\gamma}}^2]\}$:
   $\mu/\sqrt{v} = \Theta(1)$ uniformly in $n$. The latter is due to the above observations that $\mu, \eta^{-1} = \Theta(n^2)$ and the variances are of order $\Theta(n^4)$.
\end{proof}

\begin{proof}[Proof of Claim A]
    Recall that if a sequence of random variables $(X_n)_n$ possesses moment generating functions (MGF) $(M_n)$ that are finite on a common open interval around zero, $(-a,a)$ for $a>0$, i.e., $M_n(u) < \infty$ for all $n \in \N, u \in (-a,a)$, and there exists $0<b<a$ and a finite-valued function $M$ defined on $[-b,b]$ such that for any $u \in [-b,b]$, $M(b) = \lim_n M_n(b)$, then there exists a random variable $X$, which is the distributional limit of $(X_n)_n$, its MGF is finite on $[-b,b]$ and coincides with $M$ on this interval (cf. \cite[Theorem 3]{aoms1177731541}). Using this fact it suffices to prove that the MGFs $(M_n)_n$ and $(\hat{M}_n)_n$ of $(X_n)_n$ and $(\hat{X}_n)$, respectively, converge pointwise for all $\lvert x\rvert \leq b < \min \big( \frac{1}{2 \sqrt{2}}, \frac{1}{4} \eta_2 \sqrt{\mathrm{Var}[\ntwo{{\gamma}}^2]} \big)$ (for all $n$ sufficiently large) to some finite (on $[-b,b]$) function $M$, unequal to a constant function $1$ (that is the limiting law is not that of a constant $0$). 
    
    First, note that by \eqref{eq:EV_asympt}, we have $\eta_2^2 \mathrm{Var}[\ntwo{{\gamma}}^2] = 2 \sum_{k=2}^{n^d} \eta_2^2/\eta_k^2 \sim c \sum_k k^{-4/3} \xrightarrow{} c' > 0$ as $n\rightarrow \infty$. Thus, a suitable $b > 0$ can be chosen as required. 
    
    Next, we find explicit expressions for the MGFs of $\hat{X}_n$ and $X_n$. To this end, let $t \in [-b,b]$ and recall that the MGF of $(Z+\mu)^2$ for a standard normal variable $Z$ is given by $\exp \big( \frac{\mu^2 x}{1-2x} - \frac{1}{2} \log(1-2x)\big)$ for all $x< 1/2$. Combined with the eigenvalue representation \eqref{eq:EV_representation} of $\hat{X}_n$ and the fact that $\mathrm{Var}[\ntwo{\hat{\gamma}^y}^2] = 2(1+o(n^{\varepsilon - 1})) \sum_{k=1}^{u_n-1} \mu_k^2$, this yields
    \begin{align*}
        \hat{M}_n(t) &= \exp \Bigg( -t \sum_{k=1}^{u_n-1}\frac{\mu_k}{\sqrt{\mathrm{Var}[\ntwo{\hat{\gamma}^y}^2]}} - \frac{1}{2} \sum_{k=1}^{u_n-1} \log\bigg( 1 - \frac{2t\mu_k}{\sqrt{\mathrm{Var}[\ntwo{\hat{\gamma}^y}^2]}}\bigg) \Bigg) \\
        &\qquad\qquad\times \exp \Bigg( \underbrace{t \sum_{k=1}^{u_n-1}\frac{h_k(y)^2}{\sqrt{\mathrm{Var}[\ntwo{\hat{\gamma}^y}^2]}} \bigg(\underbrace{ \bigg(1 - \frac{2t \mu_k}{\sqrt{\mathrm{Var}[\ntwo{\hat{\gamma}^y}^2]}} \bigg)^{-1} - 1}_{\lvert \cdot\rvert \in [0,3)}\bigg)}_{\lvert \cdot\rvert \in \big[0, 3b \ntwo{h(y)}^2 / \sqrt{\mathrm{Var}[\ntwo{\hat{\gamma}^y}^2]} \big) \overset{\eqref{eq:bound_shift}}{\subset} (0, o(n^{\varepsilon - 1}))} \Bigg) \\
        &= (1 + o(n^{\varepsilon - 1})) \exp \Bigg( \sum_{l=2}^\infty \underbrace{\frac{1}{2l} \bigg( \frac{\sqrt{2} t}{1+o(n^{\varepsilon - 1})}\bigg)^l \sum_{k=1}^{u_n-1} \bigg(\frac{\mu_k^2}{ \sum_{p=1}^{u_n-1} \mu_p^2}\bigg)^{l/2}}_{\eqqcolon f_{n}(l)} \Bigg)
    \end{align*}
    for any $\varepsilon \in (0,1)$. Here, we additionally used that 
    $$0 < \frac{2\lvert t\rvert\mu_k}{\sqrt{\mathrm{Var}[\ntwo{\hat{\gamma}^y}^2]}} \leq \frac{(1+o(1)) 2 b\mu_1}{\sqrt{\mathrm{Var}[\ntwo{{\gamma}}^2]}} < \frac{1+o(1)}{2} \eta_2 \mu_1 \leq \frac{1+o(1)}{2} < \frac{3}{4}$$ 
    for all $n$ sufficiently large, which follows by \eqref{eq:auxiliary_estim} and \eqref{eq:EV_asympt}. 
    
    In order to show that the limits of the MGFs $\hat{M}_n$ and $M_n$ exist and coincide, we first observe that the sum in $f_n(l)$ equals $1$ when $l=2$. Since moreover $\mu_k^2/ \sum_{p=1}^{u_n-1} \mu_p^2 < 1$ for all $k$, it remains bounded by $1$ for all $l > 2$. Therefore, uniformly in $t \in [-b,b]$, $$\lvert f_n(l)\rvert \leq \frac{1}{2l} \Big(\frac{3}{4}\Big)^l.$$ Since the latter bound is summable, by dominated convergence theorem, we obtain
    \begin{align*}
        \lim_{n\rightarrow\infty} \hat{M}_n(t) = \exp \Bigg( \sum_{l=2}^\infty \frac{1}{2l} (\sqrt{2} t)^l \lim_{n \rightarrow \infty} \sum_{k=1}^{u_n-1} \bigg(\frac{\mu_k^2}{ \sum_{p=1}^{u_n-1} \mu_p^2}\bigg)^{l/2} \Bigg).
    \end{align*}
    And analogously,
    \begin{align*}
        \lim_{n\rightarrow\infty} M_n(t) = \exp \Bigg( \sum_{l=2}^\infty \frac{1}{2l} (\sqrt{2} t)^l \lim_{n \rightarrow \infty} \sum_{k=2}^{n^d} \bigg(\frac{1/\eta_k^2}{ \sum_{p=2}^{n^d} 1/\eta_p^2}\bigg)^{l/2} \Bigg).
    \end{align*}
    It thus remains to show that for each fixed $l \geq 3$, $$\lim_{n} \sum_k \bigg(\frac{\mu_k^2}{ \sum_p \mu_p^2}\bigg)^{l/2} = \lim_{n} \sum_k \bigg(\frac{1/\eta_k^2}{ \sum_p 1/\eta_p^2}\bigg)^{l/2},$$
    and that these limits are indeed well-defined. But before turning to proving this, observe that for $l=2$, both sides of the latter expression are equal to $1$, which then suffices to conclude that the limiting distribution of $(\hat{X}_n)$ and $(X_n)_n$ is non-trivial. 
    
    We start by verifying the well-definedness of $\lim_{n} \sum_k \big(\frac{1/\eta_k^2}{ \sum_p 1/\eta_p^2}\big)^{l/2}$. First note that
    \begin{align*}
        \sum_{k=2}^{n^d} \bigg(\frac{1/\eta_k^2}{ \sum_{p=2}^{n^d} 1/\eta_p^2}\bigg)^{l/2} =\hspace{-0.25cm}\sum_{x \in [0,n) \cap \Z^d \setminus \{0\}}\hspace{-0.2cm} \bigg(\frac{1/\eta_x^2}{ \sum_{y \in [0,n) \cap \Z^d \setminus \{0\}} 1/\eta_y^2}\bigg)^{l/2} = \hspace{-0.1cm} \sum_{x \in \hat{\Lambda}_n \setminus \{0\}} \hspace{-0.15cm} \bigg(\frac{1/\eta_x^2}{ \sum_{y \in \hat{\Lambda}_n \setminus \{0\}} 1/\eta_y^2}\bigg)^{l/2},
    \end{align*}
    where $\eta_x = 2 \sum_{i=1}^3 (1-\cos(2\pi x_i/n))$ and $\hat{\Lambda}_n = [-n/2, n/2)^3 \cap \Z^3$. Let $f(n)$ be an arbitrary function slowly growing to infinity such that $f(n) = o(n)$, e.g., $f(n) = \log(n)$. Then, using explicit form of $\eta_x$ and Taylor expanding it, we easily get that
    \begin{align*}
        \sum_{y \in \hat{\Lambda}_n \setminus \{0\}} \frac{1}{n^4 \eta_y^2} = \big(1 + \mathcal{O}\big(f(n)^{-1}\big)\big) \sum_{\substack{y \in \hat{\Lambda}_n \setminus \{0\}\\ \lvert y\rvert \leq f(n)}} \frac{1}{(2\pi)^4 \lvert y\rvert^4} + o(1) \overset{n \rightarrow \infty}{\sim} \sum_{y \in \Z^3\setminus \{0\}} \frac{1}{(2\pi)^4 \lvert y\rvert^4}.
    \end{align*}
    Fully analogously,
    \begin{align*}
        \sum_{x \in \hat{\Lambda}_n \setminus \{0\}} \Big(\frac{1}{n^2 \eta_x}\Big)^l \overset{n \rightarrow \infty}{\sim} \sum_{x \in \Z^3\setminus \{0\}} \frac{1}{(2\pi)^{2l} \lvert x\rvert^{2l}},
    \end{align*}
    which then yields
    \begin{align*}
        \lim_{n\rightarrow \infty} \sum_k \bigg(\frac{1/\eta_k^2}{ \sum_p 1/\eta_p^2}\bigg)^{l/2} = \sum_{x \in \Z^3\setminus \{0\}} \frac{1}{\lvert x\rvert^{2l}} ~\Big/~ \bigg( \sum_{y \in \Z^3\setminus \{0\}} \frac{1}{\lvert y\rvert^4} \bigg)^{l/2} < \infty.
    \end{align*}
    
    To conclude the proof of the claim, it only remains to show that the ratio $$\sum_k \bigg(\frac{\mu_k^2}{ \sum_p \mu_p^2}\bigg)^{l/2}~\bigg/~\sum_k \bigg(\frac{1/\eta_k^2}{ \sum_p 1/\eta_p^2}\bigg)^{l/2} \xrightarrow{n\to \infty} 1.$$ 
    The latter can further be reduced to showing $T_n \coloneqq \sum_k \mu_k^l ~\big/\sum_k 1/\eta_k^l \nearrow 1$ as $n$ tends to infinity, which is due to the fact that $\sum_p \mu_p^2 = \frac{1}{2} \mathrm{Var}[\ntwo{\gamma}^2] + o(1) = \sum_p 1/\eta_p^2 +o(1)$ and $\sum_k \mu_k^l ~\big/~\sum_k 1/\eta_k^l \leq 1$. The former observation follows from the proof of \eqref{eq:var_error}, while the latter is a consequence of \eqref{eq:EV_asympt} and \eqref{eq:auxiliary_estim} (since the contribution of the smallest $\lvert U\rvert$ eigenvalues $\mu_k$ to the sum is irrelevant). 
    
    Our objective now is to find an alternative expression for $T_n$ in terms of traces of the zero-average Green's function and its appropriate principal submatrix. To this end, note that for $k \ll n^d$, by the proof of \eqref{eq:auxiliary_estim}, we get that $\mu_k \geq \max(\zeta_k -\upsilon_1, \zeta_{k+\lvert U\rvert})$, where $\zeta_k$ is the $k$-th largest eigenvalue of $G' \coloneqq G^{0\text{-}\mathrm{avg.}}_{\Lambda_n}\vert_{U^c\times U^c}$ and $\upsilon_1$ is the largest eigenvalue of $A \coloneqq G^{0\text{-}\mathrm{avg.}}_{\Lambda_n}\vert_{U^c\times U}G^{0\text{-}\mathrm{avg.}}_{\Lambda_n}\vert_{U\times U}^{-1} G^{0\text{-}\mathrm{avg.}}_{\Lambda_n}\vert_{U\times U^c}$. We now show that for $k \leq \log^3 n$: $\mu_k = \zeta_k(1-o(1))$ with $o(1)$ uniform in this range of $k$. First observe that by \eqref{eq:EV_asympt}, $\zeta_k \geq 1/\eta_{k+\lvert U\rvert + 1} = \Theta(n^2/k^{2/3})$. A sufficiently good upper bound for $\upsilon_1 > 0$ follows from \eqref{eq:0avgGfct_polydec}: $$\upsilon_1 \leq \mathrm{Trace}(A) \leq c(U) \sum_{x \in U^c, y \in U} G^{0\text{-}\mathrm{avg.}}_{\Lambda_n}(x,y)^2 = \mathcal{O}(n).$$ Analogously to the above, this implies that
    \begin{align*}
        \frac{\sum_{k=1}^{u_n-1} (\mu_k/n^2)^l}{\sum_{k=2}^{n^d} 1/(n^2\eta_k)^l} {\sim} \frac{\sum_{k=1}^{\log^3 n} (\mu_k/n^2)^l}{\sum_{k = 2}^{\log^3 n} 1/(n^2\eta_k)^l} {\sim} \frac{\sum_{k=1}^{\log^3 n} (\zeta_k/n^2)^l}{\sum_{k = 2}^{\log^3 n} 1/(n^2\eta_k)^l} {\sim} \frac{\sum_{k=1}^{u_n} (\zeta_k/n^2)^l}{\sum_{k=2}^{n^d} 1/(n^2\eta_k)^l} = \frac{\mathrm{Trace}((G')^l)}{\mathrm{Trace}((G^{0\text{-}\mathrm{avg.}}_{\Lambda_n})^l)}
    \end{align*}
    as $n \rightarrow \infty$, and that $\mathrm{Trace}((G')^l) = \Theta(n^{2l})$. 
    
    To show that the latter ratio converges to $1$, observe first that \eqref{eq:0avgGfct_polydec} yields that there is an absolute constant $K>0$ such that for all $x,y \in \hat{\Lambda}_n$, $|G^{0\text{-}\mathrm{avg.}}_{\Lambda_n}(x,y)| \leq K G_{\Z^d}(x,y)$. The switch to $G_{\Z^d}$ instead of $G^{0\text{-}\mathrm{avg.}}_{\Lambda_n}$ guarantees that all the terms in the expressions below are non-negative and, in particular, allows us to estimate as in the third line below. With the convention that $x_{l+1} = x_1$, we have
    \begin{align*}
        0 \leq \mathrm{Trace}((G^{0\text{-}\mathrm{avg.}}_{\Lambda_n})^l) &- \mathrm{Trace}((G')^l) = \sum_{k=1}^l \sum_{\substack{I \subset \{1, \ldots, l\}\\ \lvert I\rvert = k}} \sum_{\substack{x_i \in U: i \in I;\\ x_j \in U^c: j \notin I}} \prod_{i=1}^l G^{0\text{-}\mathrm{avg.}}_{\Lambda_n} (x_i, x_{i+1})\\
        &\leq \sum_{k=1}^l K^l \sum_{\substack{I \subset \{1, \ldots, l\}\\ \lvert I\rvert = k}} \sum_{\substack{x_i \in U: i \in I;\\ x_j \in U^c: j \notin I}} \prod_{i=1}^l G_{\Z^d} (x_i, x_{i+1})\\ 
        &\leq \sum_{k=1}^l K^l \binom{l}{k} \sum_{u \in U} \sum_{\substack{x_i \in \hat{\Lambda}_n: i \neq 1}} G_{\Z^d}(u, x_2) \bigg(\prod_{i=2}^{l-1} G_{\Z^d} (x_i, x_{i+1})\bigg) G_{\Z^d}(x_l, u)\\
        &\leq (2^l - 1) K^l \sum_{u \in U} \lambda_{\max} \Big( G_{\Z^d}\vert_{\hat{\Lambda}_n \times \hat{\Lambda}_n}^{l-2}\Big) \sum_{\substack{x\in \hat{\Lambda}_n}} G_{\Z^d}(u, x)^2\\
        &\leq c n \lambda_{\max} \big( G_{\Z^d}\vert_{\hat{\Lambda}_n \times \hat{\Lambda}_n}\big)^{l-2}
    \end{align*}
    for an appropriate $c = c(l, U)>0$. Furthermore, by Cauchy-Schwarz inequality,
    \begin{align*}
        \lambda_{\max} \big( G_{\Z^d}\vert_{\hat{\Lambda}_n \times \hat{\Lambda}_n}\big) \leq \sup_{h \in \R^{\hat{\Lambda}_n}: \norm{h} = 1} \sum_{x, y \in \hat{\Lambda}_n} h_x G_{\Z^d}(x,y) h_y \leq \sup_{x \in \hat{\Lambda}_n} \sum_{y \in \hat{\Lambda}_n} G_{\Z^d}(x,y) \leq c' n^2.
    \end{align*}
    Thus, the ratio of interest $1 \geq \mathrm{Trace}((G')^l)/ \mathrm{Trace}((G^{0\text{-}\mathrm{avg.}}_{\Lambda_n})^l) \geq 1 - \mathcal{O}(n^{-3})$ converges to $1$ as $n$ tends to infinity.
\end{proof}

\begin{remark}
\label{rem:Malliavin}
    As mentioned above, instead of relying on Proposition \ref{prop:dens_conv_1D} and the estimate \eqref{eq:0avgGfct_polydec}, which together yield uniform convergence of the probability density functions of the variables $X_n$ and $\hat{X}_n$ to the standard normal density for $d \geq 4$, one could consider an alternative route via Malliavin calculus \cite{HU2014814}. More precisely, by slightly modifying the proof of Lemma 7.1 in that reference, one can verify the assumptions of their Corollary 6.6 (cf. Theorem 7.3) for $\hat{X}_n$ and $X_n$ in dimensions $d \geq 5$. For $d = 4$, however, checking the fifth assumption would require an additional argument to replace Lemma 7.1, which we expect would resemble the proof of our Proposition \ref{prop:dens_conv_1D}.\\
    Finally, this alternative approach does not extend to $d = 3$, as the variables $\hat{X}_n$ and $X_n$ are not Gaussian in the limit.
\end{remark}

\subsection{Local correlation functions of the Spherical model}
\label{subsec:spher_correlations}

The key objective of this section is to prove Corollary \ref{cor:spher_correlations}, namely to show that
\begin{align}
\label{eq:cov_converg}
    \nu_{\Lambda_n, \beta} \bigg[\prod_{x \in U} \theta_x^{i_x}\bigg] \xrightarrow{n \rightarrow \infty} \E\bigg[ \prod_{x \in U} \alpha_x^{i_x}\bigg],
\end{align}
where $\sqrt{\beta} \alpha$, as described in Theorem \ref{thm:spher_detailed}, is either an $m^2$-massive GFF on $\Z^d$ if $\beta < \beta_c$, a GFF on $\Z^d$ if $\beta = \beta_c$ or a GFF on $\Z^d$ plus an independent drift $\sqrt{\beta - \beta_c} X \mathbf{1}_{\Z^d}$ with a Rademacher random variable $X$ if $\beta > \beta_c$.
Since $\theta_U$ already converges in law to $\alpha_U$, it suffices to verify the uniform integrability of the variables $\prod_{x \in U} \theta_x^{i_x}$ in order to conclude \eqref{eq:cov_converg} (see, e.g., \cite[Theorem 25.12]{billingsley1995probability}).
The latter follows instantaneously from the following:
\begin{lemma}[Moments of spins of the spherical model]
    \label{lemma:spher_moments}
    Let $\theta$ be a configuration of the spherical model on $\Lambda_n$ at inverse temperature $\beta > 0$. For any $p \in \N_0$, there exists a constant $C_p > 0$, independent of $n$, such that for all $n$,
    \begin{align*}
        \nu_{\Lambda_n, \beta}\big[\lvert\theta_0\rvert^{2p}\big] \leq C_p.
    \end{align*}
\end{lemma}
Indeed, since for a sequence of random variables $(X_n)$ to be uniformly integrable, it suffices to satisfy the condition $\sup_n \E[\lvert X_n\rvert^2] < \infty$; and by generalized H\"older's inequality, we have
\begin{align*}
    \nu_{\Lambda_n, \beta} \bigg[\prod_{x \in U} \lvert\theta_x\rvert^{2i_x}\bigg] \leq \prod_{x \in U} \nu_{\Lambda_n, \beta}\Big[\lvert\theta_x\rvert^{2i_x \lvert U\rvert}\Big]^{1/\lvert U\rvert} \leq \nu_{\Lambda_n, \beta}\Big[\lvert\theta_0\rvert^{2 \lvert U\rvert \max_x(i_x)}\Big] \leq C_{|U| \max_x (i_x)}.
\end{align*}
The second to last inequality follows from the fact that spins of the spherical model on torus are identically distributed.

We now proceed to the proof of the lemma, which is motivated by \cite[Lemma 3]{Lukkarinen2020}.
\begin{proof}
    Recall that by Proposition \ref{prop:rel_GFF_spher}, $\sqrt{\beta} \theta$ has the same law as any massive ($m^2>0$) GFF $\phi$ on $\Lambda_n$ conditioned on $\ntwo{\phi}^2 = \beta n^d$. Therefore,
    \begin{align*}
        \beta^p \nu_{\Lambda_n, \beta}\big[\theta_0^{2p}\big] = \E\big[\phi_0^{2p} ~\big\vert~ \ntwo{\phi}^2 = \beta n^d\big].
    \end{align*}
    Let $Q = (q_{x}^w)_{x, w \in [0, n)^d \cap \Z^d}$ be the orthonormal matrix diagonalizing $G{\Lambda_n, m^2}$, whose columns are the eigenvectors $(q^w)_w$; in particular, we have $Q^T G_{\Lambda_n, m^2} Q = \mathrm{diag} \big( ((m^2+\eta_w)^{-1})_{w}\big)$, where $(\eta_w)_{w \in [0, n)^d \cap \Z^d}$ are the eigenvalues of $-\Delta_{\Lambda_n}$. Let $(Z_w)_w$ be i.i.d. standard normal random variables. Then, 
    \begin{align*}
        \phi = (\phi_x)_{x \in \Lambda_n} &\overset{\text{law}}{=} \bigg( \sum_{w \in [0, n)^d \cap \Z^d} q_{x}^w \frac{1}{\sqrt{m^2+\eta_w}} Z_w\bigg)_{x \in \Lambda_n}; \\
        \ntwo{\phi}^2 &\overset{\text{law}}{=} \sum_{w \in [0, n)^d \cap \Z^d} \frac{1}{\eta_w + m^2} Z_w^2.
    \end{align*}
    Using this, we obtain on the one hand that
    \begin{align*}
        \beta^{p} = \frac{1}{n^{dp}} \E\big[\ntwo{\phi}^{2p} ~\big\vert~ \ntwo{\phi}^2 = \beta n^d\big] = \frac{1}{n^{dp}} \sum_{\substack{w_1, \ldots, w_{p}\\ \in[0, n)^d \cap \Z^d }} \E\bigg[ \prod_{j=1}^{p} \frac{Z_{w_j}^2}{m^2+ \eta_{w_j}} ~\bigg\vert~ \sum_{w} \frac{Z_w^2}{\eta_w + m^2} = \beta n^d\bigg].
    \end{align*}
    And on the other hand, by Lemma \ref{lemma:massive_EVs},
    \begin{align*}
        \E\big[\phi_0^{2p} ~\big\vert~ \ntwo{\phi}^2 = \beta n^d\big] &= \sum_{\substack{w_1, \ldots, w_{2p}\\ \in[0, n)^d \cap \Z^d }} \bigg(\prod_{j=1}^{2p} q_{0}^{w_j}\bigg) \E\bigg[ \prod_{j=1}^{2p} \frac{Z_{w_j}}{\sqrt{m^2+ \eta_{w_j}}} ~\bigg\vert~ \sum_{w} \frac{Z_w^2}{\eta_w + m^2} = \beta n^d\bigg]\\
        &= \frac{1}{n^{dp}} \sum_{\substack{w_1, \ldots, w_{2p}\\ \in[0, n)^d \cap \Z^d }}\E\bigg[ \prod_{j=1}^{2p} \frac{Z_{w_j}}{\sqrt{m^2+ \eta_{w_j}}} ~\bigg\vert~ \sum_{w} \frac{Z_w^2}{\eta_w + m^2} = \beta n^d\bigg].
    \end{align*}
    Since the conditional law of $(Z_{w}/\sqrt{m^2+ \eta_{w}})w$, given $\sum_{w \in [0, n)^d \cap \Z^d} \frac{1}{\eta_w + m^2} Z_w^2 = \beta n^d$, is symmetric in each coordinate, it is invariant under sign flips. Therefore, any product involving an odd power of some $Z_w$ vanishes under the conditional expectation, and hence,
    \begin{align*}
        \E\big[\phi_0^{2p} ~\big\vert~ \ntwo{\phi}^2 = \beta n^d\big] \leq \frac{(2p)!}{n^{dp}} \sum_{\substack{w_1, \ldots, w_{p}\\ \in[0, n)^d \cap \Z^d }}\E\bigg[ \prod_{j=1}^{p} \frac{Z^2_{w_j}}{m^2+ \eta_{w_j}} ~\bigg\vert~ \sum_{w} \frac{Z_w^2}{\eta_w + m^2} = \beta n^d\bigg] = (2p)! \beta^p.
    \end{align*}
    This concludes the proof, establishing the uniform moment bound $\nu_{\Lambda_n, \beta}\big[\theta_0^{2p}\big] \leq (2p)! = C_p$.
\end{proof}


\section{The infinite spin-dimensionality distributional limit of the spin \texorpdfstring{$O(N)$}{O(N)} model}
\label{sec:O(N)}
This section studies the ``local" infinite spin-dimensionality limit ($N \rightarrow \infty$) of the spin $O(N)$ model both on the finite domain—the discrete $d$-dimensional torus $\Lambda_n$ for $d \geq 2$—and in the subsequent infinite volume limit ($n \rightarrow \infty$). More precisely, we prove the following two theorems corresponding to these respective cases that combined with Scheff\'e's lemma\footnote{Scheff\'e's lemma states: if a sequence of integrable functions $(f_n)_n$ on a measure space $(X, \F, \mu)$ converges almost everywhere to an integrable function $f$, then $(f_n)$ converges to $f$ in $L^1(\mu)$ if and only if $\norm{f_n}_{L^1(\mu)} \rightarrow \norm{f}_{L^1(\mu)}$.} yield Theorem \ref{thm:spinO(N)} in the introduction. 
\begin{theorem}[$N\rightarrow \infty$ limit of the spin $O(N)$ model on a finite torus]
\label{thm:spinO(N)_torus}
    Let $\Lambda_n = \T^d_n$ and $S = (S_x)_{x \in \Lambda_n}$ be a configuration of the spin $O(N)$ model at inverse temperature $\beta > 0$.
    For any $M \in \N$, the density function of $\ovlineM{S} \coloneqq (S^i_x)_{x \in \Lambda_n}^{i = 1, \ldots, M}$ converges as $N \rightarrow \infty$ to that of an $M$-vectorial massive GFF on $\Lambda_n$ scaled by $1/\sqrt{\beta}$ with the mass $m_n^2$ depending on $\beta, d$, and $n$ in such a way that
    \begin{align} \label{eq:choice_of_mass}
        G_{\Lambda_n, m_n^2}(x,x) = \beta \quad \text{for all } \, x\in \Lambda_n.
    \end{align}
    or more explicitly\footnote{This follows from the fact that the trace of a matrix equals the sum of its eigenvalues and Lemma \ref{lemma:massive_EVs}.}, $m^2$ is the unique solution to  
    \begin{align}
    \label{eq:mass_spinO(N)}
        \frac{1}{m_n^2 n^d} + \frac{1}{n^d} \sum_{x \in [0,n)^d \cap \Z^d \setminus\{0\}} \frac{1}{m_n^2 + 2\sum_{k=1}^d (1 - \cos(2\pi x_k/n))} = \beta.
    \end{align}
\end{theorem}

\begin{theorem}[Infinite-volume $N\rightarrow \infty$ limit of the spin $O(N)$ model]
\label{thm:spinO(N)_infVol}
    For each $n, N \in \N$, let $S = (S_x)_{x \in \Lambda_n}$ be a configuration of the spin $O(N)$ model on $\Lambda_n$ at inverse temperature $\beta > 0$.  Set $\beta_c \coloneqq G_{\Z^d}(0,0)$. Then, for any $M \in \N$ and a finite set $U \subset \Z^d$ (considered as a subset of $\Lambda_n$ for each $n$), $\ovlineM{S}_{U} \coloneqq (S^i_x)_{x \in U}^{i = 1, \ldots, M}$ converges in law as $N \rightarrow \infty$ followed by $n \rightarrow \infty$ towards: 
    \begin{enumerate}
        \item $\beta < \beta_c$: an $M$-vectorial massive GFF on $\Z^d$ restricted to $U$ scaled by $1/\sqrt{\beta}$ with the mass $m^2$ depending on $\beta$ and $d$ in such a way that
        \begin{align*}
            G_{\Z^d, m^2}(x,x) = \beta \quad \text{for all }\, x\in \Z^d;
        \end{align*}
        \item $\beta = \beta_c$: an $M$-vectorial GFF on $\Z^d$ restricted to $U$ scaled by $1/\sqrt{\beta}$;
        \item $\beta > \beta_c$: an $M$-vectorial GFF on $\Z^d$ restricted to $U$ scaled by $1/\sqrt{\beta}$ plus an independent constant random drift of the form $\sqrt{\frac{\beta - \beta_c}{\beta}} \ovlineM{Z} \mathbf{1}_U$ with $\ovlineM{Z}$ being an $M$-dimensional standard normal vector.
    \end{enumerate}
\end{theorem}

\subsection{The infinite spin-dimensionality limit on the torus}
\label{subsec:spinO(N)_torus}

This subsection is concerned with the infinite spin-dimensionality limit of the spin $O(N)$ model on the discrete torus, and in particular proves Theorem \ref{thm:spinO(N)_torus}. Since $n \in \N$ is fixed throughout this section, we write $\Lambda \coloneqq \T^d_n$ for brevity.

Before proceeding to the proof, let us quickly discuss the strategy. First, by Proposition \ref{prop:mGFF-O(N)}, we know that for any $m^2$-massive $N$-vectorial GFF $\phi$ with $m^2>0$,
\begin{align*}
    \mathrm{Law} \big(\ovlineM{S}\big) = \mathrm{Law} \bigg(\frac{\ovlineM{\phi}}{\sqrt{\beta}} ~\bigg\vert~ |\phi_x| = \sqrt{\beta N}, \,\forall x \in \Lambda \bigg),
\end{align*}
where $\ovlineM{\phi}$ denotes the first $M$ components and $\rest{\phi}$ the remaining $N-M$ components of $\phi$. Thus, for any $\xi \coloneqq (\xi_x)_{x\in \Lambda} \in (\R^M)^{|\Lambda|}$, we have the following equality of densities with respect to Lebesgue measure on $\R^{M |\Lambda|}$:
\begin{align}
\label{eq:dens_O(N)}
    f_{\ovlineM{S}}(\xi) = f_{\frac{\ovlineM{\phi}}{\sqrt{\beta}} ~\big\vert (|\phi_x|)_x = (\sqrt{\beta N})_x}(\xi) 
    = f_{\frac{\ovlineM{\phi}}{\sqrt{\beta}}} \left(\xi \right) \frac{f_{(|\rest{\phi_x}|^2)_x} (({\beta N} - \beta |\xi_x|^2)_x)}{f_{(|\phi_x|^2)_x} (({\beta N})_x)}.
\end{align}
Since by the central limit theorem, 
\begin{equation*}
    \bigg(\frac{|{\phi_x}|^2- \E[|{\phi_x}|^2]}{\sqrt{N}}\bigg)_{x \in \Lambda} = \bigg(\frac{|{\phi_x}|^2- N G_{\Lambda, m^2}(0,0)}{\sqrt{N}}\bigg)_{x \in \Lambda}
\end{equation*} 
(and similarly for $\rest{\phi}$) converges to a centered $n^d$-dimensional Gaussian vector, we expect that, by choosing $m^2 = m^2_n$ as in \eqref{eq:choice_of_mass}, the above ratio converges to one, which would complete the proof. 
Motivated by this discussion, we state and prove (see Appendix \ref{A:proof_densconv}) the following result, which can be seen as a multidimensional generalization of \cite[Chp.8 $\S$46 Theorem 1]{gnedenko1968limit}.
\begin{proposition}
\label{prop:dens_conv}
    Let $(X^i)_i$ be a sequence of i.i.d. centered random vectors with the probability density function (with respect to Lebesgue measure on $\R^k$) $f: \R^k \rightarrow \R$. Assume that 
    \begin{enumerate}
        \item $f \in L^r(\R^k)$ for some $r \in (1,2]$;
        \item All the entries of the covariance matrix $C$ of $X^1 = (X^1_j)_{j=1}^k$ are well-defined and finite, or equivalently, for all $1\leq j\leq k$, $X_j^1 \in L^2(\P)$;
        \item $C$ is positive definite.
    \end{enumerate}
    Then the relation 
    \begin{align*}
        n^{k/2} f^{(n)} (\sqrt{n} v) \xrightarrow[]{n\rightarrow \infty} \frac{1}{(2\pi)^{k/2}\sqrt{\det C}} \: e^{-\frac{1}{2} (v, C^{-1}v)}
    \end{align*}
    holds uniformly over $v \in \R^k$. Here, $f^{(n)}$ is the $n$-fold convolution of $f$, as well as the density function of $\sum_{i=1}^n {X}^i$.
\end{proposition}
\begin{proof}[Proof of Theorem \ref{thm:spinO(N)_torus}]
    Choose $m^2> 0$ as in \eqref{eq:mass_spinO(N)}. Then, $\E[(\phi^i_x)^2] = G_{\Lambda, m_n^2}(x,x) = \beta$. 
    We start by checking that the i.i.d. vectors $\big((\phi^i)^2 - \beta \coloneqq ((\phi^i_x)^2 - \beta)_{x\in \Lambda} \big)_{i}$ satisfy assumptions of Proposition \ref{prop:dens_conv}. Property $2.$ is obvious since we are working with the squares of Gaussians. As for $3.$, $C$ is positive semi-definite as a covariance matrix. Furthermore, $v^T Cv = 0$ if and only if $\langle v, (\phi^i)^2-\beta\rangle = \sum_{x\in \Lambda} v_x ((\phi_x^i)^2-\beta) = 0$ almost surely, which in turn is true only for $v = 0 \in \R^{\Lambda}$. To verify $1.$, we derive a formula for the density function $f$ of $((\phi^1_x)^2 - \beta)_{x\in \Lambda}$ and show that it belongs to $L^r(\R^{n^d})$ for any $r \in [1,2)$. Let $f_{\phi}$ be the density function of an $m_n^2$-massive GFF on $\Lambda$, then for any $t = (t_x)_{x \in \Lambda} \in \R^\Lambda$
    \begin{align*}
        f(t) = \bigg( \prod_{x \in \Lambda} \frac{1}{2\sqrt{t_x + \beta}} \bigg) \sum_{(k_x)_x \in \{0,1\}^{n^d}} f_{\phi} \Big( \big((-1)^{k_x} \sqrt{t_x + \beta}\big)_x\Big) \ind_{\{\forall x: \: t_x + \beta > 0\}}.
    \end{align*}
    Recall that the eigenvalues of $(-\Delta_\Lambda + m_n^2)$ belong to $[m_n^2, m_n^2 + 4d]$ (see Lemma \ref{lemma:massive_EVs}). Therefore, 
    \begin{align*}
        f(t) \leq \frac{2^{n^d}}{(2\pi)^{n^d/2} \sqrt{\det G_{\Lambda, m_n^2}}} \prod_{x\in \Lambda} \left( \frac{1}{2\sqrt{t_x + \beta}} e^{-\frac{1}{2} m_n^2 (t_x + \beta)} \ind_{\{t_x + \beta > 0\}}\right).
    \end{align*}
    From this explicit formula we directly see that $f \in L^r(\R^{n^d})$ for any $r \in [1, 2)$. 
    
    Having verified the assumptions of Proposition \ref{prop:dens_conv}, we now apply it to the ratio in \eqref{eq:dens_O(N)}.
    \begin{align*}
        &\frac{f_{(|\rest{\phi_x}|^2)_x} \left(({\beta N} - \beta |\xi_x|^2)_x \right)}{f_{(|\phi_x|^2)_x} (({\beta N})_x)} = \Big({\frac{N}{N-M}}\Big)^{{n^d}/{2}} \times\\
        &\hspace{1.45cm}\times \frac{(N-M)^{{n^d}/{2}} f_{(|\rest{\phi_x}|^2 - (N-M)\beta)_x} \Big(\sqrt{N-M}\big(\frac{\beta M - \beta |\xi_x|^2}{\sqrt{N-M}}\big)_x \Big)}{N^{n^d/2} f_{(|\phi_x|^2 - N\beta)_x} (\sqrt{N}({0})_x)}
        \xrightarrow{N \rightarrow \infty} 1.
    \end{align*}
\end{proof}

\subsection{The infinite-volume limit: \texorpdfstring{$N\rightarrow \infty$}{N goes to infty} followed by \texorpdfstring{$n\rightarrow \infty$}{n goes to infty}}
\label{subsubsec:infdim,infvol_limit}
In this subsection we discuss  the subsequent infinite-volume limit of the model and prove Theorem \ref{thm:spinO(N)_infVol}. 

Taking the limit $N \to \infty$ first significantly simplifies the problem. Indeed, by Theorem \ref{thm:spinO(N)_torus}, we know that for any fixed sufficiently large $n \in \N$ and finite $U \subset \Z^d$ (viewed as a subset of $\Lambda_n$), the marginal $\ovlineM{S}_U \coloneqq (S^{i}_x)_{x \in U}^{i = 1, \dots, M}$ converges in law as $N\rightarrow \infty$ to $\frac{1}{\sqrt{\beta}} \ovlineM{\phi}_U \coloneqq \frac{1}{\sqrt{\beta}}(\phi^i_x)_{x\in U}^{i= 1,\ldots, M}$, where $\ovlineM{\phi}_U$ is an $m^2_n$-massive $M$-vectorial GFF on $\Lambda_n$ restricted to $U$, with $m_n^2$ determined by \eqref{eq:mass_spinO(N)}. 
Since the fields $(\phi^i)_i$ are i.i.d. centered Gaussian vectors, it suffices to verify that the covariance matrix of a single coordinate process converges to that of the limiting field. 

\begin{proof}[Proof of Theorem \ref{thm:spinO(N)_infVol} (Convergence of the covariance structure)]
    Recall that the critical inverse temperature $\beta_c \in (0, \infty]$ corresponds to $G_{\Z^d}(0,0)$. In particular, $\beta_c = \infty$ if $d=2$, $\beta_c \in (0,\infty)$ for $d\geq 3$. Let us have a closer look at \eqref{eq:mass_spinO(N)}: using Riemann sum approximation, for $m^2> 0$ if $d=2$, or $m^2 \geq 0$ if $d\geq 3$, we get
    \begin{equation}
    \label{eq:asympt_Gdiag}   
        \begin{aligned}
           \frac{1}{n^d} \sum_{x \in [0,n)^d \cap \Z^d \setminus\{0\}} &\frac{1}{m^2 + 2\sum_{k=1}^d (1 - \cos(2\pi x_k/n))} \\
           \overset{n\rightarrow \infty}{\sim} &\int_{(0,1)^d} \frac{1}{m^2 + 2\sum_{k=1}^d (1 - \cos(2\pi x_k))} \d x = G_{\Z^d, m^2}(0,0).
        \end{aligned}
    \end{equation}
    
    If $\beta < \beta_c$, this implies that $(m^2_n)_n$—solutions to \eqref{eq:mass_spinO(N)}—converge to $m^2> 0$, the unique solution to $G_{\Z^d, m^2}(0,0) = \beta$. Combined with the fact that $G_{\Lambda_n, m^2}(x,y) \to G_{\Z^d, m^2}(x,y)$ for all $x,y \in U$, this completes the proof in this case.

    Suppose that $\beta > \beta_c$ ($d\geq 3$). By \eqref{eq:asympt_Gdiag}, we see that in the limit the sum in \eqref{eq:mass_spinO(N)} can be at most $\beta_c$. Hence, $m_n^2$ must be of order $\Theta(1/n^d)$. Since all the non-zero eigenvalues of $-\Delta_{\Lambda_n}$ are at least of order $\Omega(1/n^2)$,
    \begin{align*}
        \frac{1}{n^d} \sum_{x \in [0,n)^d \cap \Z^d \setminus\{0\}} &\frac{1}{m_n^2 + \eta_x} = \frac{1-\mathcal{O}(n^{2-d})}{n^d} \sum_{x \in [0,n)^d \cap \Z^d \setminus\{0\}} \frac{1}{\eta_x} \overset{n \rightarrow \infty}{\sim} \beta_c.
    \end{align*}
    Therefore, $m_n^2 = \frac{1 + o(1)}{(\beta - \beta_c) n^d}$, where the $o(1)$ term is uniform in large $n$. Let $(q^w)_{w \in [0,n)^d \cap \Z^d}$ be the orthonormal eigenvectors of $-\Delta_{\Lambda_n}$ corresponding to the eigenvalues $(\eta_w)_{w \in [0,n)^d \cap \Z^d}$—in particular, $q^0_x = \frac{1}{\sqrt{n^d}}$ for any $x \in [0,n)^d \cap \Z^d$. For any $x, y \in U$, we have
    \begin{equation}
    \label{eq:Gfct_zeroavGfct}
        \begin{aligned}
            G_{\Lambda_n, m_n^2}(x,y) &= \sum_{w \in [0,n)^d \cap \Z^d \setminus\{0\}} \frac{1}{m_n^2+\eta_w} q^w_x q^w_y + q^0_x q^0_y \frac{1}{m^2_n} \\
            &= (1-o(1)) \bigg(\sum_{w \in [0,n)^d \cap \Z^d \setminus\{0\}} \frac{1}{\eta_w} q^w_x q^w_y + (\beta - \beta_c)\bigg)\\
            &= (1-o(1)) \Big( G_{\Lambda_n}^{0\text{-avg.}}(x,y) + (\beta - \beta_c) \Big) \\
            &\xrightarrow[\text{unif. in $x,y\in U$}]{n \rightarrow \infty} G_{\Z^d}(x,y) + (\beta - \beta_c),
        \end{aligned}
    \end{equation}
    where $G_{\Lambda_n}^{0\text{-avg.}}$ is the zero-average Green's function on $\Lambda_n$. We conclude by observing that $(G_{\Z^d}(x,y) + (\beta - \beta_c))_{x,y \in U}$ is the covariance matrix of the sum of the GFF on $\Z^d$ restricted to $U$ and a Gaussian field of the form $(\sqrt{\beta - \beta_c} Z)_{x \in U}$, where $Z \sim \mathcal{N}(0,1)$ is independent of the GFF. 

    Let $\beta = \beta_c$. By \eqref{eq:asympt_Gdiag}, we see that for any $\varepsilon > 0$, there exists $n_0$ such that for all $n \geq n_0$, $0 < m_n^2 < \varepsilon$. Combined with \eqref{eq:Gfct_RW}, this implies that for any $\delta \in (0,1)$,
    \begin{align*}
        G_{\Z^d, \varepsilon}(x,y) \xleftarrow[\text{unif. over $U$}]{n \rightarrow \infty} G_{\Lambda_n, \varepsilon}(x,y) \leq G_{\Lambda_n, m_n^2}(x,y) \leq G_{\Lambda_n, \min(m_n^2, n^{-2 - \delta})}(x,y).
    \end{align*}
        As $\varepsilon \rightarrow 0$, $G_{\Z^d, \varepsilon}(x,y)$ converges to $G_{\Z^d}(x,y)$ uniformly in $x,y \in U$. Moreover, analogously to \eqref{eq:Gfct_zeroavGfct}, by \eqref{eq:mass_spinO(N)}
    \begin{align*}
        G_{\Lambda_n, \min(m_n^2, n^{-2 - \delta})}(x,y) &= (1 - o(1)) G_{\Lambda_n}^{0\text{-avg.}}(x,y) + \frac{1}{\min(m_n^2, n^{-2 - \delta}) n^d} \\
        &= (1-o(1)) G_{\Lambda_n}^{0\text{-avg.}}(x,y) + \mathcal{O}(n^{2+ \delta - d}) \\
        &\hspace{1.6cm}+ \Big(\beta_c - \frac{1-\mathcal{O}(m_n^2 n^2)}{n^d} G_{\Lambda_n}^{0\text{-avg.}}(0,0)\Big) \ind_{\{m_n^2 \leq n^{-2 - \delta}\}} \\
        &\xrightarrow[\text{unif. in $x,y\in U$}]{n \rightarrow \infty} G_{\Z^d}(x,y)
    \end{align*}
    Hence, $G_{\Lambda_n, m_n^2}$ converges uniformly over $U\times U$ to $G_{\Z^d}$ as desired. 
\end{proof}


\appendix
\section{Appendix}
\label{sec:Appendix}

\subsection{Polynomial decay of zero-average Green's function}
\label{A:zero_avg_polydecay}

This subsection proves \eqref{eq:0avgGfct_polydec}, which we restate as the following proposition.
\begin{proposition}
    Let $d \geq 3$, $n \in \N$, and $\Lambda \coloneqq \T^d_n$ be the discrete $d$-dimensional torus of side length $n$. Then, there exists $C = C(d)>0$ such that for all $n$ sufficiently large and $y \in \hat \Lambda \coloneqq [-n/2, n/2)^d \cap \Z^d$, 
    \begin{align*}
        \lvert G_{\Lambda}^{0\text{-}\mathrm{avg.}}(0,y) - G_{\Z^d}(0,y) \rvert \leq C n^{2-d}.
    \end{align*}
\end{proposition}

\begin{proof}
Recall that (see Proposition \ref{prop:props:zeroAvGfct}), 
\begin{align*}
    2d G_{\Lambda}^{0\text{-}\mathrm{avg.}}(0,y) = \int_0^\infty \Big( \P^0_{\Lambda}[\overline{X}_t = y] - \frac{1}{n^d} \Big) \d t,
\end{align*}
where $\overline{X} = (\overline{X}_t)_{t \geq 0}$ is the continuous-time simple random walk on $\Lambda$ viewed as a graph. To obtain the desired estimate we follow the strategy of the proof of \cite[Proposition 1.5]{Aba17} corresponding to our \eqref{eq:zero-avGfct_polydec} and split the above integral approximately at the time $\overline{{X}}$ reaches equilibrium (the uniform distribution). More precisely, as in the reference (with $N = n$ in our setting), for $\lambda_* \coloneqq \frac{1}{d}(1-\cos(2\pi/n)) \overset{n \rightarrow \infty}{\sim} \frac{2 \pi^2}{d n^2}$ (the spectral gap of $\overline{{X}}$) and $t_* \coloneqq \frac{\log(n^d)}{\lambda_*}$ by \cite[Theorem 20.6]{levin2017markov},
\begin{align*}
    \int_{t_*}^\infty \Big\lvert \P^0_{\Lambda}[\overline{X}_t = y] - \frac{1}{n^d} \Big\rvert \d t \leq \int_{t_*}^\infty e^{-\lambda_* t} \d t = \frac{1}{\lambda_* n^d} = \Theta(n^{2-d}). 
\end{align*}

The remaining part of the integral will be treated similarly to the referenced proof but with higher precision tailored to our setting. Note that 
\begin{align*}
    \int_0^{t_*} &\Big( \P^0_{\Lambda}[\overline{X}_t = y] - \frac{1}{n^d} \Big) \d t 
    = \int_0^{t_*} \sum_{k=0}^\infty \P^0_{\Lambda} [N_t = k] \P^0_0 [X_k = y] \d t - \frac{t_*}{n^d} \\
    &\in \sum_{k=0}^{t_* + \sqrt{4 t_* \log(n^d)}} \P^0_0 [X_k = y] \int_0^{t_*} \frac{t^k e^{-t}}{k!} \d t - \frac{t_*}{n^d} + \Bigg[0, \int_0^{t_*} \P^0_{\Lambda} [N_t > t_* + \sqrt{4 t_* \log(n^d)}] \d t\Bigg].
\end{align*}
Here $N_t$ is the number of jumps of $(\overline{X}_s)_s$ up to time $t$.
We now show that $\int_0^{t_*} \P^0_{\Lambda} [N_t > t_* + \sqrt{4 t_* \log(n^d)}] \d t = o(n^{2-d})$ as $n \rightarrow \infty$. For $t \in (0, t_*)$ by exponential Markov's inequality,
\begin{equation}
   \label{eq:Pois_tails} 
   \begin{aligned}
        \P^0_{\Lambda} [N_t > t_* + \sqrt{4 t_* \log(n^d)}] &\leq \inf_{\lambda > 0} \exp \Big( t(e^\lambda - 1) - \lambda \big(t_* + \sqrt{4 t_* \log(n^d)}\big) \Big)\\
        &\leq \exp \Big( t_* \inf_{\lambda > 0} \big[(e^\lambda - 1) - \lambda \big(1 + \sqrt{4 \lambda_*}\big) \big] \Big) \\
        &= \exp \Big( t_* \big[\sqrt{4 \lambda_*} - \underbrace{\log \big(1 + \sqrt{4 \lambda_*}\big)}_{\geq \sqrt{4 \lambda_*} - 4\lambda_*/2} \big(1 + \sqrt{4 \lambda_*}\big) \big] \Big)\\
        &\leq \exp \big( - 2 t_* \lambda_*(1 - o(1))\big) \leq (n^d)^{-3/2}.
    \end{aligned}
\end{equation}
And thus, $\int_0^{t_*} \P^0_{\Lambda} [N_t > t_* + \sqrt{4 t_* \log(n^d)}] \d t \leq t_* n^{-3d/2} = \Theta(n^{2-3d/2} \log(n)) = o(n^{2-d})$ as desired. Furthermore, for $0 \leq k \leq t_* + \sqrt{4 t_* \log(n^d)} \eqqcolon t_*'$,
\begin{align*}
    \P^0_0 [X_k = y] = \sum_{\substack{v \in \Z^d:\\ \d_{\Z^d}(0, y+nv) \leq t_*'}} \P^{0, \Z^d}_0 [X_k = y + nv].
\end{align*}
We split the latter sum into two parts: $\sqrt{t_*'} C_d \sqrt{\log(n^d)} \leq \d_{\Z^d}(0, y+nv) \leq t_*'$ and $\d_{\Z^d}(0, y+nv) \leq \sqrt{t_*'} C_d \sqrt{\log(n^d)}$ for some $C_d > 0$ to be determined later. We show that the contribution of the former fragment of the sum to $\sum_{k=0}^{t_*'} \P^0_0 [X_k = y] \int_0^{t_*} \frac{t^k e^{-t}}{k!} \d t$ can be made as small as $o(n^{2-d})$ if $C_d$ is chosen appropriately (the lower bound is clearly zero). For some appropriate absolute (depending only on the dimension) constants $c, c', r_0 > 0, R_0 < \infty$,
\begin{align*}
    \underbrace{\sum_{k=0}^{t_*'} \int_0^{t_*} \frac{t^k e^{-t}}{k!} \d t}_{\leq t_*} &\underbrace{\sum_{\substack{v \in \Z^d:\sqrt{t_*'} C_d \sqrt{\log(n^d)} \leq \\ \d_{\Z^d}(0, y+nv) \leq t_*'}}}_{\leq c (t_*'/n)^d} \P^{0, \Z^d}_0 [X_k = y + nv] \\
    &\leq \P^{0, \Z^d}_0 \Big[\max_{k=0}^{\lfloor t_*'\rfloor} \d_{\Z^d}(0, X_k) \geq \sqrt{t_*'} C_d \sqrt{\log(n^d)}\Big] \underbrace{t_* c(t_*'/n)^d}_{= \Theta(n^{2+d} \log^{d+1}(n))}\\
    &\leq c'n^{2+d} \log^{d+1}(n) \P^{0, \Z^d}_0 \Big[\max_{k=0}^{\lfloor t_*'\rfloor} \lvert X_k\rvert \geq (C_d/d) \sqrt{t_*' \log(n^d)}\Big]\\
    &\leq c'n^{2+d} \log^{d+1}(n) R_0 e^{-r_0(C_d/d)^2 \log(n^d)} = \Theta(n^{2+d - r_0C_d^2/d} \log^{d+1}(n)).
\end{align*}
The last inequality (and in particular, existence of $r_0, R_0$ as above) follows from \cite[Proposition 2.1.2 (b)]{lawler_limic_2010}. By choosing $C_d$ such that $C_d > d\sqrt{2/r_0}$, say, $C_d = (2d/\sqrt{r_0}) \vee 1$, we get that the latter expression is $o(n^{2-d})$. 

It remains to bound 
\begin{align*}
    I_n \coloneqq \sum_{k=0}^{t_*'} \sum_{\substack{v \in \Z^d: \d_{\Z^d}(0, y+nv) \\\qquad \leq C_d \sqrt{t_*'\log(n^d)}}} \P^{0, \Z^d}_0 [X_k = y + nv] \int_0^{t_*} \frac{t^k e^{-t}}{k!} \d t - \frac{t_*}{n^d}.
\end{align*}
First observe that $\int_0^{t_*} \frac{t^k e^{-t}}{k!} \d t = \P[\mathrm{Pois}(t_*) \geq k+1]$ (by definition of the upper incomplete gamma function and explicit expression of Poisson cumulative distribution function in terms of it). Therefore,
\begin{align*}
    I_n \leq \sum_{k=0}^{t_*'} \sum_{\substack{v \in \Z^d: \d_{\Z^d}(0, y+nv) \\\qquad\leq C_d \sqrt{t_*'\log(n^d)}}} \P^{0, \Z^d}_0 [X_k = y + nv] - \frac{t_*}{n^d},
\end{align*}
but also 
\begin{align*}
    I_n &\geq \P[\mathrm{Pois}(t_*) \geq t_*-\sqrt{4 t_* \log(n^d)}] \\ &\quad\times \sum_{k=0}^{t_*-\sqrt{4 t_* \log(n^d)}} \sum_{\substack{v \in \Z^d: \d_{\Z^d}(0, y+nv) \\\qquad\leq C_d \sqrt{t_*'\log(n^d)}}} \P^{0, \Z^d}_0 [X_k = y + nv] - \frac{t_*}{n^d}\\
    &\geq (1- n^{-3d/2}) \sum_{k=0}^{t_*-\sqrt{4 t_* \log(n^d)}} \sum_{\substack{v \in \Z^d: \d_{\Z^d}(0, y+nv) \\\qquad\leq C_d \sqrt{t_*'\log(n^d)}}} \P^{0, \Z^d}_0 [X_k = y + nv] - \frac{t_*}{n^d}.
\end{align*}
In the latter inequality we have used that fully analogously to \eqref{eq:Pois_tails}, $\P[\mathrm{Pois}(t_*) < t_*-\sqrt{4 t_* \log(n^d)}] \leq n^{-3d/2}$. 
Define $p_k: \R^d \rightarrow \R$ by $p_k(x) \coloneqq 2 \big(\frac{d}{2\pi k}\big)^{d/2} e^{-\frac{d\lvert x\rvert^2}{2k}}$ and write $x \leftrightarrow k$ if $k$ and $x$ have the same parity, i.e., if $k + x_1 + \ldots + x_d$ is even. Set further $E_k(x)\coloneqq \big( \P^{0, \Z^d}_0[X_k = x] - p_k(x)\big)\ind_{\{k \leftrightarrow x\}}$. By \cite[Lemma 1.5.2]{lawler2012intersections}, $\lvert x\rvert^\alpha \sum_{k \geq 0} \lvert E_k(x)\rvert \xrightarrow{\lvert x\rvert \rightarrow \infty} 0$ for every $\alpha < d$. Note that for $y \in \hat{\Lambda}$ and $v \in \Z^d \setminus\{0\}$, $\lvert y+nv\rvert = \Omega(n) \geq n/2$. Altogether, this implies that for $t \in \{t_*', t_* - \sqrt{4 t_* \log(n^d)}\}$ and $\alpha \in (d-2, d)$, 
\begin{align*}
    \sum_{k=0}^{t} &\sum_{\substack{v \in \Z^d: \d_{\Z^d}(0, y+nv) \\\qquad\leq C_d \sqrt{t_*'\log(n^d)}}} \P^{0, \Z^d}_0 [X_k = y + nv] \\
    &= \underbrace{\sum_{k=1}^{t} \sum_{\substack{v \in \Z^d\setminus \{0\}: \d_{\Z^d}(0, y+nv) \\\qquad\leq C_d \sqrt{t_*'\log(n^d)}}} p_k(y + nv) \ind_{\{k \leftrightarrow y+nv\}}}_{\eqqcolon J_n(t)} \\
    &\qquad \qquad+ \sum_{k=0}^t \P^{0, \Z^d}_0 [X_k = y] \pm \underbrace{\Theta((\log n)^d) o(n^{-\alpha})}_{= o(n^{2-d})}.
\end{align*}

\begin{itemize}
    \item \textbf{Claim 1:} $\quad0 \leq 2d G_{\Z^d}(0,y) - \sum_{k=0}^t \P^{0, \Z^d}_0 [X_k = y] \leq o(n^{2-d})$.
    \item \textbf{Claim 2:} $\quad\lvert J_n(t) - \frac{t_*}{n^d}\rvert = \mathcal{O}(n^{2-d})$.
\end{itemize}

Note that combined with the above arguments these claims would imply the desired result. Indeed, by the above and since $G_{\Z^d}(0,y)$ is finite for all $y \in \Z^d$,
\begin{align*}
    -\mathcal{O}(n^{2-d}) + \underbrace{(1 - n^{-3d/2})\Big(J_n + \sum_{k=0}^t \P^{0, \Z^d}_0 [X_k = y]\Big) - \frac{t_*}{n^d}}_{\geq 2d G_{\Z^d}(0,y) - \mathcal{O}(n^{2-d})} \leq 2d G^{0\text{-}\mathrm{avg.}}(0,y) \\
    \leq \mathcal{O}(n^{2-d}) + \underbrace{J_n - \frac{t_*}{n^d} + \sum_{k=0}^t \P^{0, \Z^d}_0 [X_k = y]}_{\leq 2d G_{\Z^d}(0,y) + \mathcal{O}(n^{2-d})}.
\end{align*}
Hence, $\lvert G^{0\text{-}\mathrm{avg.}}(0,y) - G_{\Z^d}(0,y)\rvert = \mathcal{O}(n^{2-d})$.

Let us first prove Claim 1. Recall that $2d G_{\Z^d}(0,y) = \sum_{k\geq 0} \P^{0, \Z^d}_0 [X_k = y]$. Therefore, it only remains to check that the tail ($k\geq t$) of this sum decays at rate $o(n^{2-d})$. This follows almost immediately from the following three observations: $p_k(y) \leq C(d) r^{-\frac{d}{2}}$, $\lvert E_k(y)\rvert \leq \mathcal{O}(k^{-\frac{d}{2}-1})$ by \cite[Theorem 1.2.1 (1.10)]{lawler2012intersections} and $t = \Omega(n^2 \log n)$, which together imply that the tail is bounded by $\int_{C n^2 \log n}^\infty (r^{-\frac{d}{2}-1} + r^{-\frac{d}{2}}) \leq \mathcal{O}(n^{2-d}/(\log n)^{\frac{d}{2} - 1}) = o(n^{2-d})$ as $d\geq 3$. 

Next, we address Claim 2. We first show that the contribution of $k < \varepsilon n^2$ terms (for any $0 < \varepsilon < \frac{1}{4}$) is $\mathcal{O}(n^{2-d})$. Indeed, since $\R_+ \ni r \mapsto p_r(z)$ is increasing on $(0, \lvert z\rvert^2)$ for any $z\in \R^d$ and $\lvert y + n v\rvert^2 \geq \frac{n^2}{4}$ for $v \in \Z^d \setminus\{0\}$,
\begin{align*}
    \sum_{k=1}^{\varepsilon n^2} \sum_{\substack{v \in \Z^d\setminus \{0\}: \d_{\Z^d}(0, y+nv) \\\qquad\leq C_d \sqrt{t_*'\log(n^d)}}} \hspace{-0.6cm}& p_k(y + nv) \leq \hspace{-0.6cm} \sum_{\substack{v \in \Z^d\setminus \{0\}: \d_{\Z^d}(0, y+nv) \\\qquad\leq C_d \sqrt{t_*'\log(n^d)}}} \int_{1}^{2\varepsilon n^2} 2 \Big(\frac{d}{2\pi r}\Big)^{d/2} e^{-\frac{d\lvert y + n v\rvert^2}{2r}} \d r\\
    &\leq \sum_{\substack{v \in \Z^d\setminus \{0\}: \d_{\Z^d}(0, y+nv) \\\qquad\leq C_d \sqrt{t_*'\log(n^d)}}} \hspace{-0.6cm} c_d \lvert y + n v\rvert^{2-d} \int_{\frac{d \lvert y + n v\rvert^2}{4 \varepsilon n^2}}^\infty s^{\frac{d}{2} - 2} e^{-s} \d s\\
    &\leq \tilde c_d n^{2-d} \varepsilon^{2 - \frac{d}{2}} \hspace{-0.6cm} \sum_{\substack{v \in \Z^d\setminus\{0\}: \d_{\Z^d}(0, y+nv) \\\qquad\leq C_d \sqrt{t_*'\log(n^d)}}} \Big\lvert \frac{y}{n} + v\Big\rvert^{-2} e^{- \frac{d}{4\varepsilon} \lvert \frac{y}{n} + v\rvert^2}\\
    &\leq c_d' n^{2-d} \varepsilon^{2 - \frac{d}{2}} \prod_{i=1}^d \Bigg(\sum_{v_i = - 2\tilde C_d \log n}^{2 \tilde C_d \log n} e^{- \frac{d}{4\varepsilon} (\frac{y_i}{n} + v_i)^2} \Bigg)\\
    &\leq c_d'' n^{2-d} \varepsilon^{2 - \frac{d}{2}} \Bigg( \int_{\R} e^{- \frac{d}{4\varepsilon} x^2} \d x + 1\Bigg)^d = \hat{c}_d n^{2-d} \varepsilon^{2 - \frac{d}{2}}.
\end{align*}
In the third line we used that for all $x > 1$, 
\begin{align*}
    \int_{x}^\infty s^{d/2} e^{-s}\d s \leq x^{d/2-2} e^{-x} \bigg(1 + \sum_{k=2}^{\lfloor d/2\rfloor} \prod_{j=2}^k \Big(\frac{d}{2} -j\Big)\bigg).
\end{align*}
Note also that there exists an absolute constant $c > 0$ such that for all $\lvert z\rvert \leq 2 C_d \sqrt{t_*' \log(n^d)}$, $t_*' \geq s \geq \varepsilon n^2$: $p_{\lfloor s\rfloor}(z) / p_{s}(z) \in \big(1 - c \frac{\log^2 n}{\varepsilon^2 n^2}, 1 + c \frac{\log^2 n}{\varepsilon^2 n^2}\big)$. 
Since $\lvert y+nv \rvert \leq \d_{\Z^d}(0, y+nv) \leq d \lvert y+nv \rvert$, and $\frac{1}{2} p_k$ is the density function of $\mathcal{N}\big(0, \frac{k}{d} \mathrm{Id}_d\big)$, this yields that 
\begin{align*}
    J_n(t) &\lessgtr_a \Big(1 \pm c \frac{\log^2 n}{\varepsilon^2 n^2}\Big) \int_{\varepsilon n^2}^{t} \d s \int_{\mathcal{B}^d\big( \frac{-y}{n}, a C_d \sqrt{\frac{t_*' \log(n^d)}{n^2}} \big) \setminus [-\frac{1}{2}, \frac{1}{2})^d} \frac{p_s\big(n \big(\frac{y}{n} + \lfloor v\rfloor \big)\big)}{2} \d v + \mathcal{O} (n^{2-d})\\
    &= \frac{n^2}{n^d} \Big(1 \pm c \frac{\log^2 n}{\varepsilon^2 n^2}\Big) \int_{\varepsilon}^{\frac{t}{n^2}} \d r \int_{\mathcal{B}^d\big( \frac{-y}{n}, aC_d \sqrt{\frac{t_*' \log(n^d)}{n^2}} \big) \setminus [-\frac{1}{2}, \frac{1}{2})^d} \frac{p_{r}\big(\frac{y}{n} + \lfloor v\rfloor \big)}{2} \d v + \mathcal{O}(n^{2-d}),
\end{align*}
where $\lfloor v\rfloor \in \Z^d$ stands for the point on the grid $\Z^d$ such that $v \in \lfloor v\rfloor + [-1/2,1/2)^d$ and $\lessgtr_a$ here means that $J_n(t) \leq$ r.h.s. with $a= 2$, non-negative $\mathcal{O}(n^{2-d})$ and $+$ in $1 \pm c \frac{\log^2 n}{\varepsilon^2 n^2}$ and $J_n(t) \geq $ r.h.s. for $a= 1/(2d)$ with $\mathcal{O}(n^{2-d}) = 0$ and $-$ in $1 \pm c \frac{\log^2 n}{\varepsilon^2 n^2}$. Observe that 
\begin{align*}
    0 \leq \frac{n^2}{n^d} \int_{\varepsilon}^{\frac{t}{n^2}} \d r \frac{p_{r}\big(\frac{y}{n}\big)}{2} \leq C_d \frac{n^2}{n^d} \int_{\varepsilon}^{\frac{t}{n^2}} r^{-\frac{d}{2}} \d r \leq \mathcal{O}(n^{2-d}).
\end{align*}
Therefore, we can add $\big[-\frac{1}{2}, \frac{1}{2}\big)^d$ back to the integral by adding $-\mathcal{O}(n^{2-d})$ to the lower bound of $J_n(t)$, for the upper bound other than that nothing changes.
Using first less precise bounds, we further show that 
\begin{align}
\label{eq:J_n_term}
    J_n(t) = \frac{n^2}{n^d} \int_{\varepsilon}^{\frac{t_*}{n^2}} \d r \int_{\R^d} \frac{p_{r}\big(\frac{y}{n} + \lfloor v\rfloor \big)}{2} \d v + \mathcal{O}(n^{2-d}).
\end{align}
By the mean-value theorem, for any $v \in \mathcal{B}^d\big( \frac{-y}{n}, 2C_d \sqrt{\frac{t_*' \log(n^d)}{n^2}} \big)$, there exists $c_v \in [0,1]$ such that
\begin{align*}
    \Big\lvert p_{r} \Big(\frac{y}{n} + \lfloor v\rfloor \Big) &- p_{r} \Big(\frac{y}{n} + v \Big)\Big\rvert \\
    &\leq p_{r} \Big(\frac{y}{n} + (1-c_v) v + c_v \lfloor v\rfloor \Big) \frac{d \big\lvert\frac{y}{n} + (1-c_v) v + c_v \lfloor v\rfloor \big\rvert}{r} \lvert v - \lfloor v\rfloor\rvert \\
    &\leq e^{\frac{d}{2r}c_v^2 \lvert v - \lfloor v\rfloor\rvert^2} p_{r} \Big(\frac{1}{\sqrt{2}} \Big(\frac{y}{n} + v\Big) \Big) d^{3/2} \frac{\big\lvert\frac{y}{n} + v \big\rvert + \sqrt{d}}{r} \\
    &\leq c_d(\varepsilon) p_{r}\Big(\frac{1}{\sqrt{2}} \Big(\frac{y}{n} + v\Big) \Big) \frac{\big\lvert\frac{y}{n} + v \big\rvert + 1}{r}.
\end{align*}
Furthermore, for any $C > 0$,  $0 < s \leq t_*'$,
\begin{align*}
    \P\big[ \norm{\mathcal{N} (0, (s/d) \mathrm{Id}_d)} &\geq C \sqrt{t_*' \log(n^d)}\big] \leq \P\big[ \norm{\mathcal{N}(0, \mathrm{Id}_d)} \geq C \sqrt{d\log(n^d)}\big]\\
    &\leq 1 - \P\big[\lvert \mathcal{N}(0,1)\rvert \leq C \sqrt{\log(n^d)}\big]^d \leq 1 - (1 - n^{-d C^2})^d = \mathcal{O}(n^{-d C^2}).
\end{align*}
Hence, for $R \coloneqq a C_d \sqrt{\frac{t_*' \log(n^d)}{n^2}}$,
\begin{align*}
    \int_{\mathcal{B}^d\big( \frac{-y}{n}, R\big)} \frac{p_{r}\big(\frac{y}{n} + \lfloor v\rfloor \big)}{2} \d v
    &= \int_{\mathcal{B}^d\big( 0, R \big)} \frac{p_{r} (u)}{2} \d u 
    \pm \frac{c_d(\varepsilon)}{r} \mathcal{O} \bigg(\int_{\mathcal{B}^d\big( 0, R \big)} \frac{ p_{r} (u/\sqrt{2}) (\lvert u\rvert + 1)}{2} \d u \bigg)\\
    &= \P[ \norm{\mathcal{N}(0, (r/d) \mathrm{Id}_d)} \leq R] 
    \pm \frac{c_d(\varepsilon)}{r} \mathcal{O} \Big(1 + \E[\norm{\mathcal{N}(0, (r/d) \mathrm{Id}_d)}]\Big) \\
    &= 1 - \mathcal{O}(n^{-d C_d^2}) \pm \frac{1}{\sqrt{r}} \mathcal{O}(1)
\end{align*}
and analogously,
\begin{align*}
    \int_{\mathcal{B}^d\big( \frac{-y}{n}, R\big)^c} \frac{p_{r}\big(\frac{y}{n} + \lfloor v\rfloor \big)}{2} \d v
    &= \P[ \norm{\mathcal{N}(0, (r/d) \mathrm{Id}_d)} \geq R] ( 1+ \mathcal{O}(1/r))  \\
    &\qquad\pm \frac{1}{r} \mathcal{O} \Big(\E[\norm{\mathcal{N}(0, (r/d) \mathrm{Id}_d)}^2]^{1/2} \P[ \norm{\mathcal{N}(0, (r/d) \mathrm{Id}_d)} \geq R]^{1/2}\Big) \\
    &= \mathcal{O}(n^{-d C_d^2/2}) \Big(1 \pm \frac{1}{\sqrt{r}}\Big)
\end{align*}
where we additionally used Cauchy inequality. This in turn implies that 
\begin{align*}
    \frac{\log^2 n}{n^d} \int_{\varepsilon}^{\frac{t}{n^2}} \d r \int_{\mathcal{B}^d\big( \frac{-y}{n}, R\big)} \frac{p_{r}\big(\frac{y}{n} + \lfloor v\rfloor \big)}{2} \d v &= \mathcal{O}\Big(\frac{t}{n^2} \frac{\log^2 n}{n^d} \Big) = o(n^{2-d}); \\
    \frac{n^2}{n^d} \int_{\varepsilon}^{\frac{t}{n^2}} \d r \int_{\mathcal{B}^d\big( \frac{-y}{n}, R\big)^c} \frac{p_{r}\big(\frac{y}{n} + \lfloor v\rfloor \big)}{2} \d v &= \mathcal{O}\Big(n^{-d C_d^2/2} \frac{t}{n^d}\Big) = o(n^{2-d}); \\
    \frac{n^2}{n^d} \int_{\frac{t_*}{n^2}}^{\frac{t}{n^2}} \d r \int_{\mathcal{B}^d\big( \frac{-y}{n}, R\big)} \frac{p_{r}\big(\frac{y}{n} + \lfloor v\rfloor \big)}{2} \d v &=  \mathcal{O}\Big(\frac{t - t_*}{n^d}\Big) = o(n^{2-d}).
\end{align*}
Hence, \eqref{eq:J_n_term}. Note that the above arguments as well imply that $\lvert J_n(t) - \frac{t_*}{n^d}\rvert = \mathcal{O} \big(n^{2-d} \sqrt{\frac{t_*}{n^2}}\big) = \mathcal{O}(n^{2-d} \sqrt{\log n})$. 

We now prove that $\lvert J_n(t) - \frac{t_*}{n^d}\rvert = \mathcal{O}(n^{2-d})$.
By Taylor's theorem for multivariate functions, for $k \in \Z^d, x \coloneqq \frac{y}{n} + k$ and $u, w \in B \coloneqq [-1/2, 1/2)^d$:
\begin{align*}
    p_{r} (x) = p_{r} (x + u) - u^T\nabla p_r (x) - u^T \int_0^1 (1-t) D^2 p_r (x + tu) \d t \: u.
\end{align*}
Note that the second term on the right-hand side is symmetric w.r.t. the origin as a function in $u$. Hence, if we integrate both sides w.r.t. $u$ and $w$ over $B$,
\begin{align*}
    p_r(x) - \int_{B} p_{r} (x + u) \d u &+ \int_B \int_{B}\int_0^1 (1-t) u^T D^2 p_r (x + tu + w) \:u \:\d t \d u \d w  \\
    &= \int_B \int_{B} \int_0^1  (1-t) u^T \big(D^2 p_r (x + tu + w) - D^2 p_r (x + tu)\big) u\: \d t \d u \d w.
\end{align*}
since $\lvert B\rvert = 1$. Observe that $(D^2 p_r (z))_{i,j} = \big( \frac{d}{r}\big)^2 p_r(z) \big(z_i z_j - \frac{r}{d} \delta_{i,j}\big)$ with the maximal eigenvalue $\big( \frac{d}{r}\big)^2 p_r(z) \big(\lvert z\rvert^2 - \frac{r}{d}\big)$ and the remaining ones equal to $- \frac{d}{r} p_r(z)$. Thus, since $r > \varepsilon > 0$,
\begin{align*}
    \Big\lvert \int_{B}\int_0^1 (1-t) u^T D^2 p_r (z + tu) \:u \:\d t \d u \Big\rvert &\leq \int_{B}\int_0^1 \lvert u\rvert^2 \Big( \frac{d}{r}\Big)^2 p_r(z+ tu) \Big(2\lvert z\rvert^2 + 2\lvert u\rvert^2 +  \frac{r}{d}\Big) \d t \d u\\
    &\leq \Big( \frac{d}{r}\Big)^2 \frac{d}{4} \Big(2\lvert z\rvert^2 + \frac{d}{2} +  \frac{r}{d}\Big) \int_{B}\int_0^1 e^{\frac{d}{2r} t^2\lvert u\rvert^2} p_r\Big(\frac{z}{\sqrt{2}}\Big) \d t \d u\\
    &\leq \Big( \frac{d}{r}\Big)^2 \Big(2\lvert z\rvert^2 + \frac{d}{2} +  \frac{r}{d}\Big) c_d(\varepsilon) p_r\Big(\frac{z}{\sqrt{2}}\Big). 
\end{align*}
The latter as a function in $z$ is clearly integrable, hence, by Fubini's theorem
\begin{align*}
    \sum_{x \in \frac{y}{n} + \Z^d} \bigg(\int_{B} p_{r} (x + u) \d u &- \int_B \int_{B}\int_0^1 (1-t) u^T D^2 p_r (x + tu + w) \:u \:\d t \d u \d w \bigg) \\ = \int_{\R^d} p_r(z) \d z &- \int_{B}\int_0^1 (1-t) u^T \int_{\R^d} D^2 p_r (z+ tu) \d z\:u \:\d t \d u =2. 
\end{align*}
Here we also used that $\frac{1}{2} p_r(z)$ is a density function of $\mathcal{N}(0, \frac{r}{d} \mathrm{Id}_d)$ and so, $\int_{\R^d} p_r(z) \big(z_i z_j - \frac{r}{d} \delta_{i,j}\big) = 2\big(\frac{r}{d} \delta_{i,j} - \frac{r}{d} \delta_{i,j}\big) = 0$. Moreover, by the mean-value theorem, for any $i, j \leq d$, $z \in \R^d$,
\begin{align*}
    \lvert \big(D^2 p_r (z &+ w) - D^2 p_r (z)\big)_{i,j}\rvert \leq \sup_{c \in [0,1]} \lvert \nabla (D^2 p_r)_{i,j} (z+ cw)\rvert \lvert w\rvert \\
    &= \Big(\frac{d}{r}\Big)^3 \sup_{c \in [0,1]} p_r(z+ cw) \bigg\lvert \bigg( -\prod_{l \in \{i,j,k\}}(z+ cw)_l + \frac{r}{d}\sum_{l \in \{i,j,k\}} (z + cw)_l \delta_{\{i,j,k\} \setminus \{l\}} \bigg)_k \bigg\rvert \lvert w\rvert \\
    &\leq \Big(\frac{d}{r}\Big)^3 p_r \Big(\frac{z+ w}{\sqrt{2}}\Big) \frac{\sqrt{d}}{2} \sup_{c \in [0,1]} e^{\frac{d}{2r} (1-c)^2 \lvert w\rvert^2} \Big ( \lvert z+ cw\rvert^3 + 3 \sqrt{d} \frac{r}{d}\lvert z+ cw\rvert \Big) \\
    &\leq C(d, \varepsilon) \Big(\frac{d}{r}\Big)^3 p_r \Big(\frac{z+ w}{\sqrt{2}}\Big) \Big(\lvert z + w\rvert^3 + \lvert z + w\rvert^2 + \frac{r}{d} \lvert z + w\rvert + \lvert z + w\rvert + \frac{r}{d} + 1 \Big)\\
    &\leq \tilde C(d, \varepsilon) \Big(\frac{d}{r}\Big)^3 p_r \Big(\frac{z+ w}{\sqrt{2}}\Big) \Big(\lvert z + w\rvert^3 + \frac{r}{d} \lvert z + w\rvert + \frac{r}{d}\Big).
\end{align*}
For us, of interest is $z = x + tu$. Note that by analogous estimation procedures we can further conclude that 
\begin{align*}
    p_r \Big(\frac{z+ w}{\sqrt{2}}\Big) \Big(\lvert z + w\rvert^3 + \frac{r}{d} \lvert z + w\rvert + \frac{r}{d}\Big) \leq C(d, \varepsilon) p_r \Big(\frac{x+ w}{2}\Big) \Big(\lvert x + w\rvert^3 + \frac{r}{d} \lvert x + w\rvert + \frac{r}{d}\Big)
\end{align*}
This then yields that 
\begin{align*}
    \sum_{x \in \frac{y}{n} + \Z^d} \int_B &\Big\lvert \int_{B} \int_0^1  (1-t) u^T \big(D^2 p_r (x + tu + w) - D^2 p_r (x + tu)\big) u\: \d t \d u \Big\rvert \d w\\
    &\leq C'(d, \varepsilon) \Big(\frac{d}{r}\Big)^3 \sum_{x \in \frac{y}{n} + \Z^d} \int_B p_r \Big(\frac{x+ w}{2}\Big) \Big(\lvert x + w\rvert^3 + \frac{r}{d} \lvert x + w\rvert + \frac{r}{d}\Big) \d w \\
    &\leq \hat C(d, \varepsilon) \Big(\frac{d}{r}\Big)^3 \bigg( \E\Big[\Big\lVert\mathcal{N}\Big(0, \frac{r}{d} \mathrm{Id}_d \Big)\Big\rVert^3 \Big] + \frac{r}{d} \E\Big[\Big\lVert\mathcal{N}\Big(0, \frac{r}{d} \mathrm{Id}_d \Big)\Big\rVert \Big] + \frac{r}{d} \bigg) \\
    &\leq \Bar C(d, \varepsilon) r^{-3/2}.
\end{align*}
Altogether, this part combined with \eqref{eq:J_n_term} show that 
\begin{align*}
    J_n(t) = \frac{n^2}{n^d} \int_{\varepsilon}^{\frac{t_*}{n^2}} \big( 1 \pm \mathcal{O}(1) r^{-3/2} \big) \d r + \mathcal{O}(n^{2-d}) = \frac{t_*}{n^d} \pm \mathcal{O} (n^{2-d})
\end{align*}
with $\mathcal{O}(1)$ uniform in $r > \varepsilon$. This concludes the proof of Claim 2 as desired. 
\end{proof}


\subsection{{``}Boundary{''} estimate for the zero-average Green's function in 3D}
\label{A:3D_bound}

In this subsection we prove the following result, which plays a central role in the proof of \eqref{eq:negativity_Greens_diff}.
\begin{proposition}
\label{A:prop:0avg_bdry_bound}
    Let $d=3$. Consider $\hat U \coloneqq (\partial [-\lfloor n/2\rfloor, \lfloor n/2\rfloor]^d) \cap \Z^d$, and let $U$ be its canonical projection onto $\Lambda = \T^d_n$. There exists an absolute constant $c_* < \frac{3^{2/3}}{2 (4\pi)^{2/3}}$\footnote{The optimal constant should be non-positive.} such that 
    \begin{align}
        \sup_{y \in U} G^{0\text{-}\mathrm{avg.}}_{\Lambda}(0,y) \leq \frac{c_*}{n} \label{A:eq:bdry_bound}
    \end{align}
    for all $n$ sufficiently large.
\end{proposition}

\begin{proof}
Without loss of generality, due to the symmetries of the torus, we may assume that $y_1 \geq \ldots \geq y_d \geq 0$. In particular, since $y \in \hat{U}$, $y_1 = \big\lfloor \frac{n}{2}\big\rfloor$. In Subsection \ref{A:zero_avg_polydecay} we showed that
\begin{equation*}
    \begin{aligned}
        2d \: n^{d-2} G^{0\text{-}\mathrm{avg.}}_\Lambda(0,y) &\leq \frac{d}{2\pi^2} + o(1) + c_d \varepsilon^{2-\frac{d}{2}} + \int_{\varepsilon}^{\frac{t_*}{n^2}} \d r \bigg(\sum_{\substack{k \in \Z^3 \\ \forall i: \lvert k_i\rvert \leq C_d \log n}} q_{r}\Big(\frac{y}{n} + k \Big) - 1\bigg),
    \end{aligned}
\end{equation*}
where $t_* = \frac{d \log(n^d)}{2\pi^2} n^2$ and $q_r(x) = \big(\frac{d}{2\pi r}\big)^{d/2} e^{-\frac{d\lvert x\rvert^2}{2r}}$. 

We begin by verifying that even if $n$ is odd, we can set $y_1 = n/2$, which would lead to an error of order at most $o(1)$. Note that since $\lvert k_1\rvert \leq C_d \log n$, $r >\varepsilon > 0$,
\begin{align*}
    q_{r}\Big(\frac{y}{n} + \Big( \frac{1}{2n}, 0, \ldots, 0\Big) + k \Big) = q_{r}\Big(\frac{y}{n} + k \Big) e^{-\frac{d}{2r}\big( \frac{1}{(2n)^2} + (k_1+ \frac{y_1}{n}) \frac{1}{n}\big)} = q_{r}\Big(\frac{y}{n} + k \Big) \Big( 1 + \mathcal{O}\Big(\frac{\log n}{\varepsilon n}\Big)\Big).
\end{align*}
Since we have already proved in Section \ref{A:zero_avg_polydecay}, that $\sum_{\substack{k \in \Z^3}} q_{r}\big(\frac{y}{n} + k \big) \leq 1 + \mathcal{O}(1) r^{-3/2} \leq c(\varepsilon)$, the above implies that
\begin{equation*}
    \mathcal{O}\Big(\frac{\log n}{\varepsilon n}\Big) \int_{\varepsilon}^{\frac{t_*}{n^2}} \d r \sum_{\substack{k \in \Z^3 \\ \forall i: \lvert k_i\rvert \leq C_d \log n}} q_{r}\Big(\frac{y}{n} + k \Big) = \mathcal{O} \Big( \frac{\log^2 n}{n} \Big) = o(1).
\end{equation*}
Therefore,
\begin{equation}
\label{eq:bound_on_0avgG_bdry}
    \begin{aligned}
        2d \: n^{d-2} G^{0\text{-}\mathrm{avg.}}_\Lambda(0,y) 
        \leq \frac{d}{2\pi^2} + o(1) + c_d \varepsilon^{2-\frac{d}{2}} + \int_{\varepsilon}^{\frac{t_*}{n^2}} \d r \bigg(\sum_{k \in \Z^3} q_{r}\Big(\frac{y^*}{n} + k \Big) - 1\bigg),
    \end{aligned}
\end{equation}
where $y^*_1 = n/2$ and the remaining coordinates are the same as of $y$.

We next show that $\sum_{k \in \Z^3} q_{r}\big(\frac{y^*}{n} + k \big)$ under our previous assumptions on $y^*$ is maximized at $\big(\frac{1}{2}, 0, \ldots, 0\big)$.
For this notice that $q_r(x) = \prod_{i=1}^d f_r(x_i)$ with $f_r$ being a density function of a centered normally distributed random variable of variance $r/d$. Hence,
\begin{align*}
    \sum_{k \in \Z^3} q_{r}\Big(\frac{y^*}{n} + k \Big) = \prod_{i=1}^d \bigg( \sum_{l \in \Z} f_{r}\Big(\frac{y^*_i}{n} + l \Big)\bigg) \eqqcolon \prod_{i=1}^d F_r\Big(\frac{y^*_i}{n}\Big).
\end{align*}
Note that $F_r$ is a function on a one-dimensional continuous torus $\T = [-1/2,1/2]/\sim$ (or alternatively, a periodic function on $\R$). Since furthermore it is in $L^2(\T)$, by Carleson's theorem it coincides with its Fourier series almost everywhere, and thus by continuity, everywhere. We have
\begin{align*}
    F_r\Big(\frac{y^*_i}{n}\Big) = \sum_{p \in \Z} e^{- \frac{r}{2d} (2\pi)^2 p^2} \cos\Big(2\pi p \frac{y^*_i}{n}\Big) \leq \sum_{p \in \Z} e^{- \frac{r}{2d} (2\pi)^2 p^2}
\end{align*}
with equality reached at $y^*_i = 0$. So, from now on let $y^* = \big(\frac{1}{2}, 0, \ldots, 0\big)$ (independently of the value of $y \in \hat{U}$). 

Let us further investigate the last summand of \eqref{eq:bound_on_0avgG_bdry}. Using the aforementioned Fourier series and noticing that all of them are absolutely convergent and integrable as functions of $r$ on the interval $(\varepsilon, \infty)$, we get 
\begin{align*}
    I &\coloneqq \int_{\varepsilon}^{\frac{t_*}{n^2}} \d r \bigg(\sum_{k \in \Z^3} q_{r}\Big(\frac{y^*}{n} + k \Big) - 1\bigg) \\
    &= \int_{\varepsilon}^{\frac{t_*}{n^2}} \d r \bigg[ \sum_{w \in \Z^2 \setminus\{0\}} e^{-\frac{r}{2d}(2\pi)^2 \lvert w\rvert^2} \bigg( \underbrace{1 - 2 \sum_{l=1}^\infty (-1)^{l-1} e^{-\frac{r}{2d} (2\pi)^2 l^2}}_{\in [0,1] \text{ unif. in } r > \varepsilon} \bigg) - 2 \sum_{l=1}^\infty (-1)^{l-1} e^{-\frac{r}{2d} (2\pi)^2 l^2} \bigg].
\end{align*}
Observe that $\int_{t_*/n^2}^\infty \sum_{l=1}^\infty e^{-\frac{r}{2d} (2\pi)^2 l^2} \d r = \mathcal{O}(n^{-1})$. Therefore, the above is bounded by the integral of the integrand over $(\varepsilon, \infty)$ plus $\mathcal{O}(n^{-1})$. By Fubini's theorem, the former integral equals to
\begin{align*}
    \sum_{w \in \Z^2 \setminus\{0\}} \int_{\varepsilon}^{\infty} \d r e^{-\frac{r}{2d}(2\pi)^2 \lvert w\rvert^2} \bigg( 1 - 2 \sum_{l=1}^\infty (-1)^{l-1} e^{-\frac{r}{2d} (2\pi)^2 l^2}\bigg) - 2 \sum_{l=1}^\infty (-1)^{l-1} \frac{2d}{(2\pi)^2 l^2} e^{-\frac{\varepsilon}{2d} (2\pi)^2 l^2}. 
\end{align*}
Note that for a fixed constant $M > 0$ (to be determined, but independent of $\varepsilon$),
\begin{align*}
    \sum_{\substack{w \in \Z^2 \setminus\{0\}\\ \lvert w\rvert > M/\sqrt{\varepsilon}}} \int_{\varepsilon}^{\infty} \d r e^{-\frac{r}{2d}(2\pi)^2 \lvert w\rvert^2} \bigg( \underbrace{1 - 2 \sum_{l=1}^\infty (-1)^{l-1} e^{-\frac{r}{2d} (2\pi)^2 l^2}}_{\leq 1}\bigg) 
    \leq \frac{2d}{(2\pi)^2} \sum_{\substack{w \in \Z^2 \setminus\{0\}\\ \lvert w\rvert > M/\sqrt{\varepsilon}}} \hspace{-0.3 cm} \frac{e^{-\frac{\varepsilon}{2d} (2\pi)^2 \lvert w\rvert^2}}{\lvert w\rvert^2} \\
    \leq \frac{2d}{(2\pi)^2} \bigg( 2\pi \int_{\frac{M}{\sqrt{\varepsilon}}}^\infty \d s \frac{e^{-\frac{\varepsilon}{2d} (2\pi)^2 s^2}}{s} + 4 \int_{\frac{M}{\sqrt{\varepsilon}}}^\infty \d s \frac{e^{-\frac{\varepsilon}{2d} (2\pi)^2 s^2}}{s^2} \bigg) 
    \leq \frac{(2d)^2 e^{-\frac{M^2}{2d} (2\pi)^2}}{(2\pi)^4 M^2} \Big(\pi + 2\frac{\sqrt{\varepsilon}}{M}\Big).
\end{align*}
Altogether, 
\begin{align*}
    I \leq \mathcal{O}\Big( \frac{1}{n}\Big) &+ \frac{(2d)^2 e^{-\frac{M^2}{2d} (2\pi)^2}}{(2\pi)^4 M^2} \Big(\pi + 2\frac{\sqrt{\varepsilon}}{M}\Big) \\
    &+ \frac{2d}{(2\pi)^2} \sum_{\substack{w \in \Z^2 \setminus\{0\}\\ \forall i: \lvert w_i\rvert \leq M/\sqrt{\varepsilon}}} e^{-\frac{\varepsilon}{2d} (2\pi)^2 \lvert w\rvert^2} \bigg[ \frac{1}{\lvert w\rvert^2} - 2 \sum_{\substack{l = 1\\ \text{odd}}}^{M/\sqrt{\varepsilon}} \bigg( \frac{e^{-\frac{\varepsilon}{2d} (2\pi)^2 l^2}}{\lvert w\rvert^2 + l^2} - \frac{e^{-\frac{\varepsilon}{2d} (2\pi)^2 (l+1)^2}}{\lvert w\rvert^2 + (l+1)^2}\bigg) \bigg]\\ 
    &- 2 \frac{2d}{(2\pi)^2}\sum_{\substack{l = 1\\ \text{odd}}}^{\infty} \bigg( \frac{e^{-\frac{\varepsilon}{2d} (2\pi)^2 l^2}}{l^2} - \frac{e^{-\frac{\varepsilon}{2d} (2\pi)^2 (l+1)^2}}{(l+1)^2}\bigg).
\end{align*}

We can write
\begin{align*}
    \frac{1}{\lvert w\rvert^2} &= 2 \frac{2 + \pi \lvert w\rvert \tanh \big( \frac{\pi \lvert w\rvert}{2}\big) - \pi \lvert w\rvert \coth \big( \frac{\pi \lvert w\rvert}{2}\big)}{4 \lvert w\rvert^2} + \pi \frac{\coth \big( \frac{\pi \lvert w\rvert}{2}\big)- \tanh \big( \frac{\pi \lvert w\rvert}{2}\big)}{2 \lvert w\rvert}\\
    &= 2\sum_{k=1}^{\infty} \frac{4k - 1}{(\lvert w\rvert^2 + (2k-1)^2)(\lvert w\rvert^2 + 4k^2)} + \pi \frac{1}{\lvert w\rvert \sinh ( \pi \lvert w\rvert)}\\
    &= 2\sum_{\substack{l=1\\ \text{odd}}}^{\infty} \bigg( \frac{1}{\lvert w\rvert^2 + l^2} - \frac{1}{\lvert w\rvert^2 + (l+1)^2}\bigg) + \pi \frac{1}{\lvert w\rvert \sinh ( \pi \lvert w\rvert)}
\end{align*}
using the following series' representations (cf. \cite[1.217 1, 1.421 2, 1.421 4]{Series_book}) for the hyperbolic trigonometric functions
\begin{align*}
    \tanh(x) &= 8x \sum_{k=1}^\infty \frac{1}{\pi^2 (2k-1)^2+ 4x^2} \quad \forall \; x \in \R;\\
    \coth(x) &= 2x \sum_{k=1}^\infty \frac{1}{\pi^2 k^2+x^2} + \frac{1}{x} \quad \forall \; x \in \R\setminus \{0\}. 
\end{align*}
These expansions follow directly from \cite{Dun09}. 

Thus far, we have shown that
\begin{align*}
    I &\leq C_M \sqrt{\varepsilon} + \frac{(2d)^2 e^{-\frac{M^2}{2d} (2\pi)^2}}{(2\pi)^4 M^2} \pi \\
    &+ \frac{2d}{(2\pi)^2} \sum_{\substack{w \in \Z^2 \setminus\{0\}\\ \lvert w\rvert \leq M/\sqrt{\varepsilon}}} e^{-\frac{\varepsilon}{2d} (2\pi)^2 \lvert w\rvert^2} 2 \Bigg[ \sum_{\substack{l = 1\\ \text{odd}}}^{M/ \sqrt{\varepsilon}} \bigg( \frac{1 -e^{-\frac{\varepsilon}{2d} (2\pi)^2 l^2}}{\lvert w\rvert^2 + l^2} - \frac{1 - e^{-\frac{\varepsilon}{2d} (2\pi)^2 (l+1)^2}}{\lvert w\rvert^2 + (l+1)^2}\bigg) \\
    &\phantom{I+ \frac{2d}{(2\pi)^2} \sum_{\substack{w \in \Z^2 \setminus\{0\}\\ \lvert w\rvert \leq M/\sqrt{\varepsilon}}} e^{-\frac{\varepsilon}{2d} (2\pi)^2 \lvert w\rvert^2}2\bigg[} + \sum_{\substack{l=1 + M/\sqrt{\varepsilon}\\ \text{odd}}}^{\infty} \bigg( \frac{1}{\lvert w\rvert^2 + l^2} - \frac{1}{\lvert w\rvert^2 + (l+1)^2}\bigg) \Bigg] &&\eqqcolon A_1 \\ 
    &+ \frac{2d}{(2\pi)^2} \sum_{\substack{w \in \Z^2 \setminus\{0\}\\ \lvert w\rvert \leq M/\sqrt{\varepsilon}}} e^{-\frac{\varepsilon}{2d} (2\pi)^2 \lvert w\rvert^2} \pi \frac{1}{\lvert w\rvert \sinh ( \pi \lvert w\rvert)} &&\eqqcolon A_2\\
    &- 2 \frac{2d}{(2\pi)^2}\sum_{\substack{l = 1\\ \text{odd}}}^{\infty} \bigg( \frac{e^{-\frac{\varepsilon}{2d} (2\pi)^2 l^2}}{l^2} - \frac{e^{-\frac{\varepsilon}{2d} (2\pi)^2 (l+1)^2}}{(l+1)^2}\bigg) &&\eqqcolon A_3.
\end{align*}
for an appropriate $C_M > 0$. For the term $A_3$, notice that by Fatou's lemma
\begin{align*}
    \liminf_{\varepsilon \rightarrow 0} \sum_{\substack{l = 1\\ \text{odd}}}^{\infty} \bigg( \frac{e^{-\frac{\varepsilon}{2d} (2\pi)^2 l^2}}{l^2} - \frac{e^{-\frac{\varepsilon}{2d} (2\pi)^2 (l+1)^2}}{(l+1)^2}\bigg) \geq \sum_{\substack{l = 1\\ \text{odd}}}^{\infty} \bigg( \frac{1}{l^2} - \frac{1}{(l+1)^2}\bigg) = \frac{\pi^2}{12}.
\end{align*}
Therefore, for all $\varepsilon > 0$ very small, there exists $c(\varepsilon) > 0$ which converges to $0$ as $\varepsilon$ does, such that $A_3 \leq -\frac{2d}{(2\pi)^2} \frac{\pi^2}{6} + c(\varepsilon)$. As for $A_2$, 
\begin{align*}
    \frac{(2\pi)^2}{2d} A_2 &\leq \sum_{w \in \Z^2 \setminus\{0\}} \frac{\pi}{\lvert w\rvert \sinh ( \pi \lvert w\rvert)} \leq 2\pi\int_{2}^\infty \d s \frac{\pi s}{s \sinh (\pi s)} + \sum_{k=1}^\infty \frac{4\pi}{k \sinh (\pi k)} \\
    &\qquad+ \frac{4\pi}{\lvert (1,1)\rvert \sinh ( \pi \lvert (1,1)\rvert)} + \frac{4\pi}{\lvert (2,2)\rvert \sinh ( \pi \lvert (2,2)\rvert)} + \frac{8\pi}{\lvert (1,2)\rvert \sinh ( \pi \lvert (1,2)\rvert)}\\
    &\leq -2\pi \log(\tanh(\pi)) + \frac{4\pi}{\sinh(\pi)} + \int_1^\infty \d s\frac{2\pi}{\sinh (\pi s)} \\
    &\qquad+ 4\pi \bigg( \frac{1}{\sqrt{2} \sinh ( \pi \sqrt{2})} + \frac{1}{2\sqrt{2} \sinh ( \pi 2\sqrt{2})} + \frac{2}{\sqrt{5} \sinh ( \pi \sqrt{5})}\bigg)\\
    &= -2\pi \log(\tanh(\pi)) - 2\log(\tanh(\pi/2))\\
    &\qquad+ 4\pi \bigg(\frac{1}{\sinh(\pi)} + \frac{1}{\sqrt{2} \sinh ( \pi \sqrt{2})} + \frac{1}{2\sqrt{2} \sinh ( \pi 2\sqrt{2})} + \frac{2}{\sqrt{5} \sinh ( \pi \sqrt{5})}\bigg).
\end{align*}
We have $\frac{(2\pi)^2}{2d} (A_2 + A_3) \leq c'(\varepsilon) + 2\pi( 0.2411 - \pi/12) \leq c'(\varepsilon) - 0.04\pi$. 

It remains to estimate $A_1$. We begin by analyzing $l \mapsto \frac{1 -e^{-\frac{\varepsilon}{2d} (2\pi)^2 l^2}}{\lvert w\rvert^2 + l^2} \eqqcolon g(l)$—of interest to us are its regions of monotonicity. For $l \geq 1$, 
\begin{align*}
    \frac{(\lvert w\rvert^2 + l^2)}{2l} g'(l) &= e^{-\frac{\varepsilon}{2d} (2\pi)^2 l^2} \Big(1 + \frac{\varepsilon}{2d} (2\pi)^2(l^2 + \lvert w\rvert^2) \Big) - 1\\
    &\geq \Big( 1- \frac{\varepsilon}{2d} (2\pi)^2 l^2\Big) \Big(1 + \frac{\varepsilon}{2d} (2\pi)^2(l^2 + \lvert w\rvert^2) \Big) - 1\\
    &= -\Big(\frac{\varepsilon (2\pi)^2}{2d} \Big)^2 \bigg( l^4 + l^2 \lvert w\rvert^2 - \frac{2d}{(2\pi)^2\varepsilon} \lvert w\rvert^2 \bigg).
\end{align*}
In particular, the subinterval of $[1, \infty)$ on which $g(l)$ is decreasing ($\{l \geq 1: g'(l) < 0\}$) is therefore included in the range of solutions of the following inequality
\begin{align*}
    l^4 + l^2 \lvert w\rvert^2 &\geq \frac{2d}{(2\pi)^2\varepsilon} \lvert w\rvert^2\\
    &\Updownarrow\\
    l^2 + \frac{\lvert w\rvert^2}{2} &\geq \lvert w\rvert^2 \sqrt{ \frac{2d}{(2\pi)^2\varepsilon} \frac{1}{\lvert w\rvert^2} + \frac{1}{4}}\\
    &\Updownarrow l\geq 1\\
    l &\geq \lvert w\rvert \sqrt{\sqrt{ \frac{2d}{(2\pi)^2\varepsilon} \frac{1}{\lvert w\rvert^2} + \frac{1}{4}} - \frac{1}{2}} \eqqcolon w_*.
\end{align*}
Furthermore, observe that 
\begin{align*}
    g(l) - g(l+1) &= \frac{(2l+1)(1- e^{-\frac{\varepsilon}{2d} (2\pi)^2 l^2}) - (\lvert w\rvert^2 + l^2) (e^{-\frac{\varepsilon}{2d} (2\pi)^2 l^2} - e^{-\frac{\varepsilon}{2d} (2\pi)^2 (l+1)^2})}{(\lvert w\rvert^2 + l^2)(\lvert w\rvert^2 + (l+1)^2)}\\
    &\leq \frac{2l+1}{(\lvert w\rvert^2 + l^2)(\lvert w\rvert^2 + (l+1)^2)}.
\end{align*}
Hence, since on $\{l \geq 1: g'(l) \geq 0\}, \quad g(l) - g(l+1) \leq 0$ (up to a borderline value of $l$), 
\begin{align*}
    2\sum_{\substack{l = 1\\ \text{odd}}}^{M/ \sqrt{\varepsilon}} \bigg( \frac{1 -e^{-\frac{\varepsilon}{2d} (2\pi)^2 l^2}}{\lvert w\rvert^2 + l^2} &- \frac{1 - e^{-\frac{\varepsilon}{2d} (2\pi)^2 (l+1)^2}}{\lvert w\rvert^2 + (l+1)^2}\bigg) + 2\sum_{\substack{l=1 + M/\sqrt{\varepsilon}\\ \text{odd}}}^{\infty} \bigg( \frac{1}{\lvert w\rvert^2 + l^2} - \frac{1}{\lvert w\rvert^2 + (l+1)^2}\bigg)\\ 
    &\leq 2\sum_{\substack{l=w_* - 1\\ \text{odd}}}^{\infty} \frac{2l+1}{(\lvert w\rvert^2 + l^2)(\lvert w\rvert^2 + (l+1)^2)} \\
    &\leq \int_{w_* - 1}^\infty \frac{(2x+1) \d x}{(\lvert w\rvert^2 + x^2)(\lvert w\rvert^2 + (x+1)^2)} + \frac{C}{\lvert w\rvert^3} \ind_{\big\{\lvert w\rvert \geq \frac{c}{\sqrt{\varepsilon}}\big\}}\\
    &= \frac{1}{\lvert w\rvert} \bigg[ \arctan \Big( \frac{w_*}{\lvert w\rvert}\Big) - \arctan \Big( \frac{w_*-1}{\lvert w\rvert}\Big)\bigg] + \frac{C}{\lvert w\rvert^3} \ind_{\big\{\lvert w\rvert \geq \frac{c}{\sqrt{\varepsilon}}\big\}}\\
    &\leq \frac{1}{\lvert w\rvert^2 + (w_*-1)^2} + \frac{C}{\lvert w\rvert^3} \ind_{\big\{\lvert w\rvert \geq \frac{c}{\sqrt{\varepsilon}}\big\}}
\end{align*}
for appropriate constants $c, C>0$. The second summand in the third line comes from the evaluation of the integrand at its local maximum point (which clearly has the leading term proportional to $\lvert w\rvert$\footnote{An explicit computation shows that the pre-coefficient of the leading term of the local maximum point is $1/\sqrt{3}$.}) if it is greater or equal to $w_*-1$. Note further that $w_* - 1 \geq w_* (1 - \sqrt{c} \varepsilon^{1/4})$ for some $c>0$ and 
\begin{align*}
    \lvert w\rvert^2 + w_*^2 (1 - \sqrt{c} \varepsilon^{1/4})^2 &\geq \lvert w\rvert^2 \Bigg( 1 + (1 - c \sqrt{\varepsilon}) \bigg(\sqrt{\frac{1}{4}+ \frac{2d}{(2\pi)^2\varepsilon} \frac{1}{\lvert w\rvert^2}} - \frac{1}{2}\bigg) \Bigg) \\
    &\geq \frac{1}{2} \Bigg(\lvert w\rvert^2 + 2 \frac{1-c\sqrt{\varepsilon}}{1+c\sqrt{\varepsilon}} \sqrt{\frac{2d}{(2\pi)^2\varepsilon}} \lvert w\rvert \Bigg).
\end{align*} 
For simplicity set $a \coloneqq 2 \frac{1-c\sqrt{\varepsilon}}{1+c\sqrt{\varepsilon}} \sqrt{\frac{2d}{(2\pi)^2\varepsilon}} \approx 2 \sqrt{\frac{2d}{(2\pi)^2\varepsilon}}$. It is now possible to estimate $A_1$, 
\begin{align*}
    \frac{(2\pi)^2}{2d} A_1 &\leq \sum_{\substack{w \in \Z^2 \setminus\{0\}\\ \lvert w\rvert \leq M/\sqrt{\varepsilon}}} \bigg[ \frac{2}{\lvert w\rvert^2 + a \lvert w\rvert} + \frac{C}{\lvert w\rvert^3} \ind_{\big\{\lvert w\rvert \geq \frac{c}{\sqrt{\varepsilon}}\big\}}\bigg] \\ 
    &\leq 2\pi \int_{1}^{\frac{M}{\sqrt{\varepsilon}}} \frac{2r \d r}{r^2 + a r} + 4 \int_{1}^{\infty} \frac{2\d r}{r^2 + a r} + C' \int_{\frac{c}{\sqrt{\varepsilon}}}^\infty \frac{r\d r}{r^3} + \frac{8}{1 + a}\\
    &\leq C\sqrt{\varepsilon} + 4\pi \log \bigg( \frac{a + M/\sqrt{\varepsilon}}{1 + a}\bigg) + 8 \frac{\log(a+1)}{a}\\
    &\leq 4\pi \log \Big( 1 + \frac{2\pi M}{2 \sqrt{2d}}\Big) + c(\varepsilon)
\end{align*}
with $c(\varepsilon) \rightarrow 0$ as $\varepsilon \rightarrow 0$. 

To sum up, for all sufficiently large $n$,
\begin{align*}
    n^{d-2} G^{0\text{-}\mathrm{avg.}}_{\Lambda}(0,y) \leq \frac{1 - 0.04\pi + 4\pi \log \Big( 1 + \frac{2\pi M}{2 \sqrt{2d}}\Big) + \pi \frac{2d}{(2\pi)^2 M^2} e^{-\frac{M^2}{2d} (2\pi)^2}}{(2\pi)^2} + c(\varepsilon) + c_d \varepsilon^{2-\frac{d}{2}}
\end{align*}
with $c(\varepsilon)>0$ that can be made arbitrarily small by choosing $\varepsilon > 0$ sufficiently small (but non-zero). Note that since $d= 3$, the last summand is as well arbitrary small. Choose $M> 0$ such that $\frac{2\pi M}{\sqrt{2d}} = 1$. Then,
\begin{align}
\label{eq:final_bound_G0avg}
    n G^{0\text{-}\mathrm{avg.}}_{\Lambda}(0,y) \leq \frac{1 - 0.04\pi + 4\pi \log(3/2) + \pi e^{-1}}{(2\pi)^2} + c(\varepsilon) \eqqcolon c_* < \frac{3^{2/3}}{2 (4\pi)^{2/3}}.
\end{align}
\end{proof} 


\subsection{Multivariate local central limit theorem: Proof of Proposition \ref{prop:dens_conv}}
\label{A:proof_densconv}

In this subsection, we prove a local central limit theorem for the densities of i.i.d. random vectors used in the study of the infinite spin-dimensionality limit of the spin $O(N)$ model. 
\begin{proposition}
    Let $(X^i)_{i\in \N}$ be a sequence of i.i.d. centered random vectors with the probability density function (w.r.t. Lebesgue measure on $\R^k$) $f: \R^k \rightarrow \R$. Assume that 
    \begin{enumerate}
        \item $f \in L^r(\R^k)$ for some $r \in (1,2]$;
        \item All the entries of the covariance matrix $C$ of $X^1 = (X^1_j)_{j=1}^k$ are well-defined and finite, or equivalently, for all $1\leq j\leq k$, $X_j^1 \in L^2(\P)$;
        \item $C$ is positive definite.
    \end{enumerate}
    Then the relation 
    \begin{align*}
        n^{k/2} f^{(n)} (\sqrt{n} v) \xrightarrow[]{n\rightarrow \infty} \frac{1}{(2\pi)^{k/2}\sqrt{\det C}} \: e^{-\frac{1}{2} (v, C^{-1}v)}
    \end{align*}
    holds uniformly over $v \in \R^k$. Here, $f^{(n)}$ is the $n$-fold convolution of $f$, as well as the density function of $\sum_{i=1}^n {X}^i$.
\end{proposition}
The following proof follows the lines of the one for the case $k=1$ in \cite[Chp.8 $\S$46 Theorem 1]{gnedenko1968limit}. Note that, deviating from the notation used in the main body of the paper, we write $(x,t)$ to denote the canonical inner product on~$\R^k$.

\begin{proof}
    Let $\psi$ denote the characteristic function of $X^1$, and $\psi_n$ the characteristic function of $\sum_{i=1}^n X^i$. In particular, 
    \begin{align*}
        \psi({t}) = \int_{\R^k} e^{i (x,t)} f({x}) \d{x} \quad \text{and} \quad \psi_n({t}) = \int_{\R^k} e^{i (x,t)} f^{(n)}({x}) \d{x} = (\psi({t}))^n.
    \end{align*}
    By assumption $f$ belongs to $L^r(\R^k)$ and to $L^1(\R^k)$ as a density function. Therefore, also $f \in L^u(\R^k)$ for any $u \in (1,r)$. This follows from Hölder's inequality, as we now explain. Let $\theta = \frac{r}{r-1} \frac{u-1}{u} \in (0,1)$, $p = \frac{r-1}{u-1} \in (1,\infty)$ and $p'= p/(p-1)$ such that $\frac{1}{p} + \frac{1}{p'} = 1$. Then, by Hölder's inequality,
    \begin{align*}
        \int_{\R^k} \lvert f\rvert^{u} \d {x} \leq \norm{ \lvert f\rvert^{u\theta}}_{L^p} \norm{\lvert f\rvert^{u(1-\theta)}}_{L^{p'}} = \norm{f}_{L^r}^{r/p} \norm{f}_{L^1}^{1/{p'}} <\infty.
    \end{align*}
    
    Let $\hat{f}({t}) \coloneqq \int_{\R^k} e^{-2\pi i(x,t)} f({x}) \d{x}$ be the Fourier transform of $f$. Then, by Hausdorff-Young inequality (see, e.g., \cite[Proposition 2.2.16]{Grafakos_2014}) we know that if $f \in L^{p}(\R^k)$ for some $p \in [1,2]$, then $\hat{f} \in L^{p'}(\R^k)$ for $p'=\frac{p}{p-1}$. Thus, the above observations about $f$ and the fact that $\hat{f}({t}) = \psi(-2\pi {t})$ yield that $\psi$ is in $L^{q}(\R^k)$ for any $q \geq \frac{r}{r-1}$. In particular, we get that for any $n \geq \frac{r}{r-1}$, $\psi_n$ is integrable. Hence, by inversion formula theorem,
    \begin{align*}
        f^{(n)}({x}) = \frac{1}{(2\pi)^k} \int_{\R^k} e^{- i (x,t)} \psi_n({t}) \d{t},
    \end{align*}
    and thus,
    \begin{align*}
        n^{k/2} f^{(n)}(\sqrt{n}{x}) = \frac{1}{(2\pi)^k} \int_{\R^k} e^{- i (x,t)} \psi_n({t}/\sqrt{n}) \d{t}.
    \end{align*}
    
    On the other hand, we know that
    \begin{align*}
        \frac{1}{(2\pi)^{k/2}\sqrt{\det(C)}} \: e^{-\frac{1}{2} ({x}, C^{-1}{x})} = \frac{1}{(2\pi)^k} \int_{\R^k} e^{- i (x,t) - \frac{1}{2} ({t}, C{t})} \d{t}.
    \end{align*}
    Therefore, to conclude the proof it suffices to show that 
    \begin{align*}
        R_n({x}) \coloneqq \int_{\R^k} e^{- i (x,t)} \left( \psi_n({t}/\sqrt{n}) - e^{- \frac{1}{2} ({t}, C{t})}\right)\d{t}
    \end{align*}
    converges to zero as $n$ tends to infinity uniformly over ${x} \in \R^k$. By central limit theorem (multivariate version) and dominated convergence theorem ($\norm{\psi_n}_{L^\infty} \leq 1$), for any $r> 0$ fixed,
    \begin{align*}
        \left\lvert \int_{B^k(0,r)} e^{- i (x,t)} \left( \psi_n({t}/\sqrt{n}) - e^{- \frac{1}{2} ({t}, C{t})}\right)\d{t} \right\rvert \xrightarrow{n\rightarrow \infty} 0 \quad\text{uniformly over } \:{x}.
    \end{align*}
    Moreover, by taking $r$ sufficiently large we can make 
    $\left\lvert \int_{\lvert {t}\rvert > r} e^{- i (x,t)- \frac{1}{2} ({t}, C{t})}\d{t} \right\rvert$
    arbitrarily small. It remains to treat the term 
    \begin{align*}
        I \coloneqq \int_{\lvert {t}\rvert > r} e^{- i (x,t)} \psi_n({t}/\sqrt{n}) \d{t},
    \end{align*}
    which we further split into two integrals $I_1$ and $I_2$ over $r < \lvert {t}\rvert \leq \varepsilon \sqrt{n}$ and $\lvert {t}\rvert > \varepsilon\sqrt{n}$, respectively, for some $\varepsilon>0$ small. 
    
    We start with $I_1$. By our assumptions, $\psi$ is twice differentiable with 
    \begin{align*}
        \nabla\psi ({0}) = {0}, \quad D^2 \psi({0}) = -C.
    \end{align*}
    Hence, by Taylor's theorem, in the neighbourhood of ${0}$,
    \begin{align*}
        \psi({t}) - 1 = (\nabla\psi ({0}))^T {t} + \frac{1}{2} {t}^T (D^2\psi({0})) {t} + o(\lvert {t}\rvert^2) = -\frac{1}{2} {t}^T C {t} + o(\lvert {t}\rvert^2).
    \end{align*}
    By assumption, $C$ is positive definite, so it is possible to choose $\varepsilon>0$ sufficiently small such that for all $\lvert{t}\rvert \leq \varepsilon$, the remainder of the expansion $o(\lvert {t}\rvert^2)$ is smaller than $\frac{1}{4} {t}^T C {t}$. Then, $\lvert \psi({t})\rvert \leq 1 - \frac{1}{4} {t}^T C {t} \leq e^{-\frac{1}{4} {t}^T C {t}}$ and
    \begin{align*}
        \lvert I_1\rvert \leq \int_{r < \lvert {t}\rvert \leq \varepsilon \sqrt{n}} \lvert\psi_n({t}/\sqrt{n}) \rvert\d{t} 
        \leq \int_{r < \lvert {t}\rvert \leq \varepsilon \sqrt{n}} e^{-n\frac{1}{4n} {t}^T C {t}} \d{t} 
        \leq \int_{r < \lvert {t}\rvert} e^{-\frac{1}{4} {t}^T C {t}} \d{t}. 
    \end{align*}
   Thus, by taking $r$ sufficiently large we can make $\lvert I_1\rvert$ arbitrarily small. 
   
   Finally, we treat $I_2$. Recall that $\hat{f}({t}) = \psi(-2\pi {t})$, and thus by Riemann-Lebesgue lemma (see \cite[Proposition 2.2.17]{Grafakos_2014}), $\lvert \psi({t}) \rvert \rightarrow 0$ as $\lvert{t}\rvert \rightarrow \infty$. Moreover, $\lvert \psi({t})\rvert < 1$ for ${t} \neq{0}$. Indeed, suppose first that there exists a non-zero ${t}$ such that $\psi({t}) = 1$. Then the following must hold:
   \begin{align*}
       \int \underbrace{[1- \cos({(x,t)})]}_{\geq 0} f({x}) \d{x} = \int \underbrace{[1-\sin({(x,t)})]}_{\geq 0} f({x}) \d{x} = 0.
   \end{align*}
   And hence, $1- \cos({(x,t)}) = 0$ for $f({x}) \d{x}$-a.e. ${x}$, which is clearly impossible unless ${t}$ is zero. Suppose now that $\lvert\psi({t})\rvert = 1$, and consider $Z \coloneqq X^1 - X^2$. Then $\psi_Z({t}) = 1$, and we get that for $f_Z({x}) \d{x}$-a.e. ${x}$, $1- \cos({(x,t)}) = 0$. This again implies that ${t} = 0$. \\
   Altogether, we proved that there exists $c>0$ such that $\lvert\psi({t})\rvert < e^{-c}$ for all $\lvert {t}\rvert > \varepsilon$. Thus, for a fixed $q > \frac{r}{r-1}$,
   \begin{align*}
       \lvert I_2\rvert \leq e^{-c(n-q)} \int_{\lvert {t}\rvert > \varepsilon \sqrt{n}} \lvert\psi({t}/\sqrt{n}) \rvert^q \d{t} = \underbrace{e^{-c(n-q)} n^{k/2}}_{\xrightarrow{n\rightarrow\infty} 0} \underbrace{\int_{\lvert {t}\rvert > \varepsilon} \lvert\psi({t}) \rvert^q \d{t}}_{<\infty}.
   \end{align*}
\end{proof}


\subsection{Local CLT for triangular arrays of independent random variables: Proof of Proposition \ref{prop:dens_conv_1D}}
\label{A:proof_densconv_1D}

This subsection establishes a local CLT for the normalized sum of independent, non-identically distributed random variables.
\begin{proposition}
    Let $(X_{i,n})_{i\leq n, n \in \N}$ be a triangular array of independent centered random variables with the probability density functions $f_{i,n}$ (w.r.t. Lebesgue measure on $\R$). Assume that 
    \begin{enumerate}
        \item $f_{i,n} \in L^r(\R)$ for some $r \in (1,2]$ (independent of $i,n$) such that $\sup_{i, n} \norm{f_{i,n}}_{L^r} \leq M$ for some $M> 0$;
        \item For all $i \leq n, n \in \N$, $\sigma_{i,n}^2 \coloneqq \mathrm{Var}[X_{i,n}] < \infty$ ordered such that $\sigma_{1,n}^2 \geq \ldots \geq \sigma_{n,n}^2$;
        \item Lindeberg's condition is satisfied, that is, for any $\varepsilon > 0$,
        \begin{align*}
            \frac{\sum_{i \geq 1} \E[X^2_{i,n} \ind_{\{ \lvert X_{i,n}\rvert > \varepsilon s_n\}}]}{s_n^2} \xrightarrow{n \rightarrow \infty} 0,
        \end{align*}
        where $s_n^2 \coloneqq \sum_{i=1}^n \sigma_{i,n}^2$.
        \item There exist $\delta > 0$ (uniform), $K(n) \geq 1$, $l_*(n) \geq 1, n \geq l^*(n) \geq 2 \lceil\frac{r}{r-1}\rceil$ such that for all $n$ sufficiently large
        \begin{align*}
            (a) &\; \frac{\sum_{i \geq l_*(n)} \sigma_{i,n}^2}{\sum_{i=1}^n \sigma_{i,n}^2} \geq \delta;\\
            (b) &\; \frac{\sum_{i \geq l_*(n)} \E[X^2_{i,n} \ind_{\{ \lvert X_{i,n}\rvert > K(n)\}}]}{\sum_{i \geq l_*(n)} \sigma_{i,n}^2} \leq \frac{1}{8}; \\
            (c) &\; \frac{n - l^*(n)}{\sigma_{l^*, n}^2 \vee K(n)^2} \gg \log ( s_n^2).
        \end{align*}
    \end{enumerate}
    Then the relation 
    \begin{align*}
        s_n f^{(n)} (s_n x) \xrightarrow{n\rightarrow \infty} \frac{1}{\sqrt{2\pi}} \: e^{-\frac{x^2}{2}}
    \end{align*}
    holds uniformly over $\R$. Here, $f^{(n)}$ is the convolution of $f_{1,n}, \ldots, f_{n,n}$, as well as the density function of $\sum_{i=1}^n X_{i,n}$. Note that $s_n f^{(n)}(s_n x)$ is the density of the normalized sum $\frac{1}{s_n} \sum_{i=1}^n X_{i,n}$. 
\end{proposition}

\begin{proof}
    Let $\psi_{i,n}$ denote the characteristic function of $X_{i,n}$, and $\psi^{(n)}$ the characteristic function of $\sum_{i=1}^n X_{i,n}$, $\hat{g}(t) \coloneqq \int_{\R} e^{-2\pi i x t} g(x) \d x$ be the Fourier transform of $g: \R \rightarrow \R$ (for an appropriate $g$).
    By the same argument as in Section \ref{A:proof_densconv}, $\hat f_{i,n}, \psi_{i,n} \in L^{q}(\R)$ for any $q \geq \frac{r}{r-1}$. In particular, by Hölder's inequality, for any fixed $q \geq \frac{r}{r-1}$ and $n \geq q$, 
    \begin{align*}
        \norm{\psi^{(n)}}_{L^1} = \Big\lVert\prod_{i=1}^n \psi_{i,n}\Big\rVert_{L^1} \leq \prod_{i=1}^n \norm{\psi_{i,n}}_{L^n} < \infty.
    \end{align*} 
    Hence, by inversion formula theorem, $f^{(n)}(x) = \frac{1}{2\pi} \int_{\R} e^{- i x t} \psi^{(n)}(t) \d t$.
    
    To conclude the proof, it suffices to show that 
    \begin{align*}
        D_n({x}) \coloneqq \int_{\R} e^{- i x t} \left( \psi^{(n)} (t/s_n) - e^{- \frac{t^2}{2}}\right)\d t
    \end{align*}
    converges to zero as $n$ tends to infinity uniformly over $x \in \R$. We first split $D_n$ into two integrals over $\lvert t \rvert \leq r$ and $\rvert t \rvert > r$ for some $r>0$ (to be specified later). By Lindeberg-Feller central limit theorem and dominated convergence theorem ($\norm{\psi_n}_{L^\infty} \leq 1$), the former converges to zero as $n$ tends to infinity uniformly in $x \in \R$. By taking $r$ sufficiently large we can also make 
    $\Big\lvert \int_{\lvert t\rvert > r} e^{- i x t- \frac{t^2}{2}}\d {t} \Big\rvert$ arbitrarily small. 
    
    The remaining part $\int_{\lvert t\rvert > r} e^{- i xt} \psi^{(n)} (t/s_n) \d t$ we split into two integrals $I_1$ and $I_2$ over $r < \lvert t\rvert \leq \delta_n s_n$ and $\lvert t\rvert > \delta_n s_n$, respectively, for some $\delta_n >0$ small (to be chosen later in dependence of $K(n)$). By our assumptions, $\psi_{i,n}' (0) = 0$ and $\psi_{i,n}''(0) = - \sigma_{i,n}^2$ for all $i,n$. Hence, by Taylor's theorem, in the neighbourhood of $0$ ($\lvert t \rvert \leq \delta_n$),
    \begin{align*}
        \psi_{i,n}(t) = 1 -\frac{\sigma_{i,n}^2 t^2}{2} + R_{i,n}(t).
    \end{align*}
    Since $e^{ix} = 1 + ix -\frac{x^2}{2} + R(x)$ with $\lvert R(x)\rvert \leq \min(\lvert x\rvert^3/6, x^2)$, the remainder term $R_{i,n}(t)$ in the above expansion of $\psi_{i,n}$ satisfies
    \begin{align*}
        \lvert R_{i,n}(t)\rvert &\leq \frac{1}{6} \E[\lvert t X_{i,n}\rvert^3 \ind_{\{\lvert X_{i,n}\rvert \leq K\}}] + \E[\lvert t X_{i,n}\rvert^2 \ind_{\{\lvert X_{i,n}\rvert \geq K\}}] \\
        &\leq \frac{K \delta_n}{6} t^2 \sigma_{i,n}^2 + t^2 \E[\lvert X_{i,n}\rvert^2 \ind_{\{\lvert X_{i,n}\rvert \geq K\}}]
    \end{align*}
    for any $K > 0$. In particular, also for $K(n)$ as in the fourth assumption. By choosing $\delta_n = \frac{3}{4 K(n)} > 0$, we obtain that $\lvert \psi_{i,n}(t)\rvert \leq 1 - \frac{3}{8} \sigma_{i,n}^2 t^2 + t^2 \E[\lvert X_{i,n}\rvert^2 \ind_{\{\lvert X_{i,n}\rvert \geq K(n)\}}]\leq \exp(-\frac{3}{8} \sigma_{i,n}^2 t^2 + t^2 \E[\lvert X_{i,n}\rvert^2 \ind_{\{\lvert X_{i,n}\rvert \geq K(n)\}}])$ for all $\lvert t\rvert \leq \delta_n$. Hence,  
    \begin{align*}
        \lvert I_1\rvert &\leq \int_{r < \lvert t\rvert \leq \delta_n s_n} \lvert\psi^{(n)} (t/s_n) \rvert \d t
        \leq \int_{r < \lvert t\rvert \leq \delta_n s_n} \prod_{i=l_*(n)}^n e^{-\frac{3}{8} \frac{\sigma_{i,n}^2}{s_n^2} t^2 + t^2 \frac{\E[\lvert X_{i,n}\rvert^2 \ind_{\{\lvert X_{i,n}\rvert \geq K(n)\}}]}{s_n^2}} \d t\\
        &\leq \int_{r < \lvert t\rvert \leq \delta_n s_n} \exp \Bigg(-t^2 \underbrace{\frac{\sum_{i \geq l_*(n)} \sigma_{i,n}^2}{s_n^2}}_{\geq \delta} \underbrace{\Bigg[\frac{3}{8}  - \frac{\sum_{i \geq l_*(n)}\E[\lvert X_{i,n}\rvert^2 \ind_{\{\lvert X_{i,n}\rvert \geq K(n)\}}]}{\sum_{i \geq l_*(n)} \sigma_{i,n}^2} \Bigg]}_{\geq 1/4} \Bigg) \d t\\
        &\leq \int_{r < \lvert t\rvert} \exp (- t^2 \delta/ 4) \d t
    \end{align*}
    By further increasing $r$, we can make $\lvert I_1\rvert$ arbitrarily small. 
    
    We now turn to the remaining $I_2$. Let $\Mm_{i,n}$ be the median of $f_{i,n}(X_{i,n})$, i.e., $\P[f_{i,n}(X_{i,n}) \geq \Mm_{i,n}] \geq 1/2$ and $\P[f_{i,n}(X_{i,n}) \leq \Mm_{i,n}] \geq 1/2$. 
    By \cite[Theorem 2]{BCG12}, there exist two absolute constants $c_1, c_2 > 0$ such that 
   \begin{align*}
       \lvert \psi_{i,n}(t)\rvert < \begin{cases}
           1 - c_1/(\Mm_{i,n}^2 \sigma_{i,n}^2), & \lvert t\rvert \geq {\pi}/(4 \sigma_{i,n});\\
           1 - c_2 t^2/\Mm_{i,n}^2, & 0 < \lvert t\rvert < {\pi}/({4 \sigma_{i,n}}).
       \end{cases}
   \end{align*}
   Since $f_{i,n}(X_{i,n})$ is non-negative, Markov's inequality implies that for any $\lambda > 0$,
   \begin{align*}
       \P[f_{i,n}(X_{i,n}) \geq \lambda] \leq \frac{\E[f_{i,n}(X_{i,n})^{r-1}]}{\lambda^{r-1}} = \frac{\int f_{i,n}^{r}(x) \d x}{\lambda^{r-1}} \leq \frac{M^r}{\lambda^{r-1}},
   \end{align*}
   which directly implies that $0< \Mm_{i,n} \leq (2 M^r)^{1/(r-1)} \eqqcolon C^{1/2}$. Thus, for all $i \leq n, n \in \N$:
   \begin{align*}
       \lvert \psi_{i,n}(t)\rvert < \begin{cases}
           1 - c_1/(C \sigma_{i,n}^2), & \lvert t\rvert \geq {\pi}/(4 \sigma_{i,n});\\
           1 - c_2 \delta_n^2/C, & \delta_n < \lvert t\rvert < {\pi}/({4 \sigma_{i,n}});
       \end{cases}
   \end{align*}
   and $\sup_{i \geq l^*(n)} \lvert \psi_{i,n}(t)\rvert < 1 - \min(c_1/(C \sigma_{l^*,n}^2), c_2 \delta_n^2/C) \leq e^{-c' \min(\sigma_{l^*,n}^{-2}, \delta_n^2)}$ for some $c'>0$ and all $\lvert t \rvert \geq \delta_n$. 
   Thus, for any fixed integer $2\lceil \frac{r}{r-1}\rceil > q > \frac{r}{r-1}$,
   \begin{align*}
       \lvert I_2\rvert &\leq s_n \int_{\lvert t\rvert > \delta_n} \prod_{i=1}^n \lvert \psi_{i,n}(t) \rvert \d t 
       \leq e^{-(n-l^*) c' \min(\sigma_{l^*,n}^{-2}, \delta_n^2)} s_n \int_{\lvert t\rvert > \delta_n} \prod_{i=1}^q \lvert \psi_{i,n}(t) \rvert \d t \\
       &\leq e^{-(n-l^*) c' \min(\sigma_{l^*,n}^{-2}, \delta_n^2)} s_n \prod_{i=1}^q \norm{\psi_{i,n}}_{L^q}
       \leq e^{-(n-l^*) c' \min(\sigma_{l^*,n}^{-2}, \delta_n^2)} s_n \prod_{i=1}^q (2\pi)^{1/q} \norm{f_{i,n}}_{L^{q/(q-1)}} \\
       &\leq e^{-(n-l^*) c' \min(\sigma_{l^*,n}^{-2}, \delta_n^2)} s_n 2 \pi \prod_{i=1}^q \Big[ \norm{f_{i,n}}_{L^r}^{r/(q(r-1))} \norm{f_{i,n}}_{L^1}^{1-r/(q(r-1))}\Big]\\
       &\leq 2 \pi e^{-(n-l^*) c' \min(\sigma_{l^*,n}^{-2}, \delta_n^2)} s_n M^{r/(r-1)} \xrightarrow[\text{unif. in }x]{n \rightarrow \infty} 0,
   \end{align*}
   where we used H\"older's and Hausdorff-Young inequalities and the third assumption ($\delta_n = c/K(n)$).   
\end{proof}


\bibliographystyle{alpha}
\bibliography{references}

\end{document}